\documentclass[aap]{imsart}

\RequirePackage[OT1]{fontenc}
\RequirePackage{amsthm,amsmath}
\RequirePackage[numbers]{natbib}
\RequirePackage[colorlinks,citecolor=blue,urlcolor=blue]{hyperref}

\usepackage{amssymb}
\usepackage{graphicx}
\usepackage{color}
\usepackage{tikz}
\usetikzlibrary{decorations.markings}
\usetikzlibrary{decorations.pathmorphing}
\usetikzlibrary{arrows, snakes}
\usetikzlibrary{shapes.geometric, calc,decorations.pathreplacing}
\usetikzlibrary{backgrounds}
\usetikzlibrary{fit}
\usetikzlibrary{decorations.pathreplacing}
\input xy
\xyoption{all}

\tikzset{middlearrow/.style={
		decoration={markings,
			mark= at position 0.5 with {\arrow{#1}} ,
		},
		postaction={decorate}
	}
}

\theoremstyle{plain}
\newtheorem{theorem}{Theorem}[section]

\newtheorem{corollary}[theorem]{Corollary}

\newtheorem{lemma}[theorem]{Lemma}

\newtheorem{proposition}[theorem]{Proposition}

\newtheorem{definition}[theorem]{Definition}

\newtheorem{assumption}[theorem]{Assumption}
\theoremstyle{remark}
\newtheorem{remark}[theorem]{Remark}
\newtheorem{example}[theorem]{Example}

\numberwithin{equation}{section}

\newcommand{\PP}{\mathbb{P}}
\newcommand\E{\mathbb{E}}

\arxiv{arXiv:1910.13668}

\startlocaldefs
\numberwithin{equation}{section}
\theoremstyle{plain}

\endlocaldefs

\begin{document}

\begin{frontmatter}
\title{Random concave functions} 
\runtitle{Random concave functions}

\begin{aug}
\author{\fnms{Peter} \snm{Baxendale}
	\ead[label=e1]{baxendal@usc.edu}}
\and
\author{\fnms{Ting-Kam Leonard} \snm{Wong}
\ead[label=e3]{tkl.wong@utoronto.ca}
}

\runauthor{P.~Baxendale and T.-K.~L.~Wong}

\affiliation{University of Southern California and University of Toronto}

\address{Department of Mathematics\\
University of Southern California\\
Los Angeles, CA 90089\\
United States\\
\printead{e1}\\
}

\address{Department of Statistical Sciences\\
University of Toronto\\
Toronto, ON M5S 3G3\\
Canada\\
\printead{e3}\\
}
\end{aug}

\begin{abstract}
Spaces of convex and concave functions appear naturally in theory and applications. For example, convex regression and log-concave density estimation are important topics in nonparametric statistics. In stochastic portfolio theory, concave functions on the unit simplex measure the concentration of capital, and their gradient maps define novel investment strategies. The gradient maps may also be regarded as optimal transport maps on the simplex. In this paper we construct and study probability measures supported on spaces of concave functions. These measures may serve as prior distributions in Bayesian statistics and Cover's universal portfolio, and induce distribution-valued random variables via optimal transport. The random concave functions are constructed on the unit simplex by taking a suitably scaled (mollified, or soft) minimum of random hyperplanes. Depending on the regime of the parameters, we show that as the number of hyperplanes tends to infinity there are several possible limiting behaviors. In particular, there is a transition from a deterministic almost sure limit to a non-trivial limiting distribution that can be characterized using convex duality and Poisson point processes.
\end{abstract}



\end{frontmatter}

\section{Introduction} \label{sec:intro}

\subsection{Motivations} \label{sec:motivations}
In this paper we study probability measures on spaces of concave functions. We first describe some applications that motivated our study. In the first two applications there is an infinite-dimensional parameter space consisting of convex or concave functions, and the problem is to find mathematically tractable prior distributions on the space.

\subsubsection{Nonparametric Bayesian statistics}
Consider a nonlinear regression problem where data is drawn according to the model
\begin{equation} \label{eqn:regression}
Y_i = f(X_i) + \epsilon_i.
\end{equation}
In many applications the regression function $f$ is known to satisfy certain shape constraints such as monotonicity or convexity/concavity. Without assuming further structures on $f$, this problem is nonparametric as $f$ is an element of an infinite dimensional function space. See \cite{JD18} for various shape constraints in economics and operations research. Although the shape-constrained estimation problem can be studied by various methods (see for example \cite{SS11, HD11, HD13} in the references therein), it is both important and interesting to consider the Bayesian approach. To do this we need suitable prior distributions for the convex function $f$. In \cite{HD11} Hannah and Dunson proposed to generate random convex functions on $\mathbb{R}^n$ by taking the maximum of a (random) number of random hyperplanes, and established rates of convergence of the Bayes estimator. While in \cite{HD11} the main concern is the support and concentration properties of the prior, we will establish concrete results about the limiting distributions as the number of hyperplanes tends to infinity.

%
%

Another important class of shape-constrained inference problems is density estimation. A classic example, studied in \cite{DR09, CSS10} among many other papers, is log-concave density estimation. Here we observe data $X_1, \ldots, X_N$ with values in $\mathbb{R}^n$, where
\[
X_i \stackrel{\text{i.i.d.}}{\sim} f
\]
and $f$ is a log-concave density, i.e., $\log f$ is concave. For example, the densities of the normal and gamma distributions are log-concave. Again, to use the Bayesian approach we need to introduce suitable prior distributions on the space of concave functions. So far there is little work on this topic except the one dimensional case \cite{MRS17}. For recent progress in log-concavity in general and density estimation we refer the reader to \cite{S17, SW14}.

In nonparametric Bayesian statistics a very useful class of prior distributions is the Dirichlet process introduced by Ferguson \cite{F73}. Realizations of the Dirichlet process are random discrete probability measures on a given state space. In one dimension, the Dirichlet process can be used to enforce shape constraints. For example, a convex function on an interval has a non-decreasing first derivative which can be identified with the distribution function of a measure. However, similar arguments do not extend immediately to multi-dimensions as the second derivative of a convex function, if exists, is matrix-valued.\footnote{The subgradient (as a set-valued mapping) satisfies a condition known as cyclical monotonicity; see \cite[Section 24]{R70}.} See Section \ref{sec:optimal.transport} for more discussion involving ideas from optimal transport. 

\subsubsection{Stochastic portfolio theory and Cover's universal portfolio} \label{sec:spt}
Consider the open unit simplex in $\mathbb{R}^n$, $n\geq 2$, defined by
\begin{equation} \label{eqn:simplex.open} 
\Delta_n := \{p \in (0, 1)^n : p_1 + \cdots + p_n = 1\}.
\end{equation}
We let $\overline{\Delta}_n$ be its closure in $\mathbb{R}^n$. Let $e_1, \ldots, e_n$ be the standard Euclidean basis of $\mathbb{R}^n$ which represents the vertices of the simplex. We denote by $\overline{e} := \left( \frac{1}{n}, \ldots, \frac{1}{n}\right)$ the barycenter of the simplex.

In stochastic portfolio theory (see \cite{F02, KF09} for introductions) the open simplex $\Delta_n$ represents the state space of an equity market with $n$ stocks. If $X_i(t) > 0$ denotes the market capitalization of stock $i$ at time $t$, we call
\begin{equation} \label{eqn:market.weight}
m_i(t) = \frac{X_i(t)}{X_1(t) + \cdots + X_n(t)}
\end{equation}
the market weight of stock $i$. The vector $m(t) = (m_i(t))_{1 \leq i \leq n}$ then defines a process evolving in the simplex $\Delta_n$. Let $\Phi : \Delta_n \rightarrow (0, \infty)$ be a positive concave function on $\Delta_n$. In this context the function $\Phi$ plays two related roles. First, $\Phi$ can be regarded as a generalized measure of diversity (analogous to the Shannon entropy) which quantifies the concentration of capital in the market \cite[Chapter 3]{F02}. Second, the concave function $\Phi$ can be used to define an investment strategy, called functionally generated portfolio, with remarkable properties. Here is how the strategy is defined when $\Phi$ is differentiable. If the market weight is $m(t) = p \in \Delta_n$ at time $t$, invest $100 \pi_i \%$ of the current capital in stock $i$, where
\begin{equation} \label{eqn:fgp}
\pi_i = p_i (1 + D_{e_i - p} \log \Phi(p)),
\end{equation}
and $D_{e_i - p}$ is the directional derivative. We call the mapping $p \mapsto \boldsymbol{\pi}(p) = \pi \in \overline{\Delta}_n$ the portfolio map generated by $\Phi$.\footnote{We may also use the concave function $\Phi$ to define portfolios generated by ranked weights; see \cite[Chapter 4]{F02}.} To give an example which will be relevant later, the geometric mean $\Phi(p) = p_1^{\pi_1} \cdots p_n^{\pi_n}$, where $\pi = (\pi_1, \ldots, \pi_n) \in \overline{\Delta}_n$ is fixed, generates the constant-weighted portfolio $\boldsymbol{\pi}(p) \equiv \pi$ \cite[Example 3.1.6]{F02}. As shown in \cite{F02, KF09, PW15, W19}, the concavity of $\Phi$ allows the portfolio to diversify and capture volatility of the market through systematic rebalancing.

In the seminal paper \cite{C91} Cover constructed what is now called an online investment algorithm by forming a Bayesian average over the constant-weighted portfolios. The main idea is that strategies which have been performing well receive additional weights that are computed using an algorithm analogous to Bayes's theorem (where the portfolio value plays the role of the likelihood). To start the algorithm one needs an initial (i.e., prior) distribution on the space of portfolio strategies.  In a nonprobabilistic framework it can be shown that Cover's universal portfolio tracks asymptotically the best strategy in the given (finite-dimensional) family, in the sense that the average regret (measured by the growth rate) with respect to the best strategy tends to zero as the number of time steps tends to infinity. In \cite{CSW16, wong2015universal} the second author and his collaborators extended Cover's approach to the nonparametric family of functionally generated portfolios. Nevertheless, for practical applications and to obtain quantitative estimates we need tractable prior distributions for the generating function $\Phi$. The distributions constructed in this paper may serve as a building block of the prior. Also see \cite{IL20} for a study of minimums of hyperplanes in the context of robust portfolio optimization in stochastic portfolio theory.


\subsubsection{Optimal transport} \label{sec:optimal.transport}
Convex and concave functions are also interesting from the viewpoint of optimal transport (see \cite{V03, V08} for in-depth overviews). Given a cost function $c : \mathcal{X} \times \mathcal{Y} \rightarrow \mathbb{R}$ and probability measures $P$ on $\mathcal{X}$ and $Q$ on $\mathcal{Y}$, the Monge-Kantorovich problem is the minimization of the transport cost
\[
\int_{\mathcal{X} \times \mathcal{Y}} c(x, y) dR(x, y)
\]
over all couplings $R$ of $(P, Q)$. When $\mathcal{X} = \mathcal{Y} = \mathbb{R}^n$ and $c(x, y) = |x - y|^2$ is the squared Euclidean distance, Brenier's theorem \cite{B91} asserts that there is a deterministic optimal transport map of the form
\begin{equation} \label{eqn:Brenier.theorem}
y = \nabla \phi(x),
\end{equation}
where $\phi$ is a convex function (this holds, for example, when $P$ and $Q$ have finite second moments and $P$ is absolutely continuous). Conversely, given $P$ fixed (e.g.~standard normal) and a convex function $\phi$, the transport map \eqref{eqn:Brenier.theorem} is optimal with respect to $P$ and the pushforward $Q = (\nabla \phi)_{\#} P$. Thus a probability distribution over $Q$ (i.e., an element of $\mathcal{P}(\mathcal{P}(\mathbb{R}^n))$ can be used to define a distribution over the space of convex functions on $\mathbb{R}^n$.

In a series of papers \cite{PW18, PW15, PW16} Pal and the second author studied a novel optimal transport problem, that we call the Dirichlet transport, on the unit simplex $\Delta_n$.\footnote{As shown in \cite{PW18, W18, WY19}, this transport problem also has remarkable properties from the information geometric point of view.} The cost function is given by
\[
c(p, q) = \log \left( \frac{1}{n} \sum_{i = 1}^n \frac{q_i}{p_i} \right) - \sum_{i = 1}^n \frac{1}{n} \log \frac{q_i}{p_i}, \quad p, q \in \Delta_n.
\]
For this cost function we proved an analogue of Brenier's theorem in \cite[Theorem 4]{PW18}: under natural conditions on $P$ and $Q$, there exists a positive concave function $\Phi$ on $\Delta_n$ such that the optimal transport map is given by
\begin{equation} \label{eqn:L.optimal.transport.map}
q = p \odot \boldsymbol{\pi} (p^{-1}), \quad p \in \Delta_n,
\end{equation}
where $\boldsymbol{\pi}$ is the portfolio map generated by $\Phi$ in the sense of \eqref{eqn:fgp},
\[
p^{-1} = \left( \frac{1/p_i}{\sum_j 1/p_j} \right)_{1 \leq i \leq n} \quad\text{and}\quad a \odot b = \left(\frac{a_ib_i}{\sum_j a_j b_j} \right)_{1 \leq i \leq n}.
\]
When $P = Q$, the identity transport $q = p$ is realized by the geometric mean $\Phi(p) = (p_1 \cdots p_n)^{1/n}$. The weighted geometric mean $\Phi(p) = p_1^{\pi_1} \cdots p_n^{\pi_n}$ corresponds to a deterministic translation under an exponential coordinate system \cite[Proposition 2.7(iii)]{PW16}. 

It follows that an element of $\mathcal{P}(\mathcal{P}(\Delta_n))$ induces a probability distribution over positive concave functions on $\Delta_n$. Measures over spaces of probability distributions are important in optimal transport, information geometry and statistics. For example, in \cite{vRS09} von Renesse and Sturm constructed an entropic measure on the Wasserstein space and defined a Wasserstein diffusion. In principle, one can use Dirichlet processes on $\Delta_n$ to define random concave functions via the Dirichlet transport problem. Further properties of this construction are left for future research. 

\begin{remark}
	Since a convex function can be identified with its epigraph, the results of this paper can be formulated in terms of random convex sets in $\mathbb{R}^{n}$ such that part of the boundary is fixed to be the unit simplex. While random convex sets (given for example by the convex hulls of random points) have been studied extensively in the literature (see for example \cite{MM05, SW08, KMTT19} and the references therein), the motivations and questions studied in this paper are quite different.
\end{remark}

\subsection{Summary of the paper}
Motivated by the applications described in Sections \ref{sec:spt} and \ref{sec:optimal.transport}, in this paper we focus on random non-negative concave functions on the unit simplex $\overline{\Delta}_n$. Thus we let
\begin{equation} \label{eqn:family.C}
\mathcal{C} := \{\psi: \overline{\Delta}_n \rightarrow [0, \infty) \text{ continuous and concave}\},
\end{equation}
and our aim is to construct and study probability measures on $\mathcal{C}$. We equip $\mathcal{C}$ with the topology of uniform convergence on compact subsets of $\Delta_n$ and the associated Borel $\sigma$-algebra. Properties of $\mathcal{C}$ are given in Section \ref{sec:notations}.

\begin{figure}[t]
	\begin{tikzpicture}[scale = 0.5]
	\draw (0, 0) to (5, 1);
	\draw (5, 1) to (10, -1);
	\draw (10, -1) to (0, 0);
	
	\draw[dashed] (0, 0) to (0, 7);
	\draw[dashed] (5, 1) to (5, 8);
	\draw[dashed] (10, -1) to (10, 5.5);
	
	\draw [draw=gray, fill=gray, opacity=0.25] (0,2) -- (5,7) -- (10,3) -- cycle;
	\draw [draw=gray, fill=gray, opacity=0.25] (0,5) -- (5,5) -- (10,2) -- cycle;
	
	\draw [draw=blue!50, fill=blue!50, opacity=0.4] (0,2) -- (3,5) -- (7.6, 2.73) -- cycle;
	\draw [draw=blue!50, fill=blue!50, opacity=0.6] (3, 5) -- (5,5) -- (10,2) -- (7.6, 2.73) -- cycle;
	96
	\draw[blue] (0, 2) -- (3, 5);
	\draw[blue] (3, 5) -- (5, 5);
	\draw[blue] (3, 5) -- (7.6, 2.73);
	\draw[blue] (0, 2) -- (7.6, 2.73);
	\draw[blue] (10, 2) --(7.6, 2.73);
	\draw[blue] (10, 2) -- (5, 5);
	
	\node [black, left] at (0, 5) {{\footnotesize $C_1$}};
	\node [black, right] at (10, 2) {{\footnotesize $C_2$}};
	\node [black, right] at (5, 5.1) {{\footnotesize $C_3$}};
	
	\end{tikzpicture}
	\caption{A random non-negative concave function on $\overline{\Delta}_n$ given as the minimum of several hyperplanes. Here $C = (C_1, \ldots, C_n)$ is a random vector which determines the coefficients of the hyperplane.} \label{fig:planes}
\end{figure}
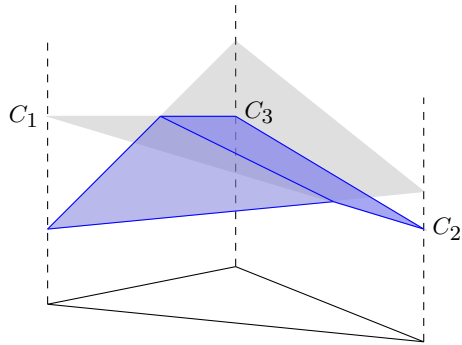

We study a natural probabilistic model for generating random concave functions. Namely, they are given by suitably scaled minimums of i.i.d.~random hyperplanes (see Figure \ref{fig:planes}). More generally, we also consider a soft minimum $\mathsf{m}_{\lambda}$ where $\lambda \in (0, \infty]$ is an inverse smoothness parameter and $\lim_{\lambda \rightarrow \infty} \mathsf{m}_{\lambda} = \mathsf{m}_{\infty} = \min$ (see Definition \ref{def:soft.min}). Thus, given a positive integer $K$, the number of hyperplanes, we consider the random concave function given by
\begin{equation} \label{eqn:our.model}
\Psi_K = a_K \mathsf{m}_{\lambda_K}(\ell_1, \ldots, \ell_K),
\end{equation}
where $a_K > 0$ is a scaling constant, $\lambda_K \in (0, \infty]$, and $\ell_1, \ldots, \ell_K$ are i.i.d.~random hyperplanes. This model is rigorously defined in Section \ref{sec:model}. Our main objective in this paper is to study the limiting behavior of the distribution of $\Psi_K$ as the number of hyperplanes tends to infinity.

In Section \ref{sec:as.limit} we consider the softmin case $\Psi_K = \mathsf{m}_{\lambda}(\ell_1, \ldots, \ell_K)$, where $\lambda \in (0, \infty)$ is finite and fixed. We show that there exists a deterministic concave function $\Psi_{\infty}$, given in terms of the distribution of $\ell_k$, such that $\Psi_K \rightarrow \Psi_{\infty}$ almost surely.

Section \ref{sec:limit.lambda.inf} studies the case of hardmin, i.e., $\lambda_K \equiv \infty$ for all $K$. Under suitable conditions on the distribution of the hyperplanes $\ell_k$, we show that the distribution of $\Phi_K$ converges weakly to a non-trivial limit $\mu$ as $K \rightarrow \infty$. Moreoever, for some constant $\alpha > 0$, if $\Psi \sim \mu$ then 
the distribution of $\Psi(p)^{n + \alpha}$ for each fixed $p \in \Delta_n$ is exponential (see Proposition \ref{prop:finite.dimensional}). This can be regarded as a functional limit theorem in the context of extreme value theory. Pointwise extremas of i.i.d.~stochastic processes have been studied extensively in the literature (see for example \cite{LLR82, R87, DF07} and the references therein). Here, the novelty of our results lies in the shape constraint and an elegant convex duality for concave functions on the unit simplex which allows us to study the limiting distribution in detail. See Remarks \ref{rmk:EVT1}  and \ref{rmk:EVT2} for more discussions about connections with extreme value theory.

Various properties of this limiting distribution $\mu$ are established, under additional conditions, in Section \ref{sec:further.properties}. In particular, we show that the geometric mean, which plays a special role in stochastic portfolio theory and the Dirichlet transport, arises as the expected value of the limiting random concave function. Using differential geometric methods, we also give an interesting explicit formula for the tail probability $\mathbb{P}(\Psi \geq \psi)$ for a given $\psi \in \mathcal{C}$.

Finally, in Section \ref{sec:diagonal}, we consider the mathematically more challenging case where the smoothness parameter $\lambda_K$ depends on $K$. We identify regimes which give different limiting behaviors. Our analysis involves studying laws of large numbers for soft minimums of i.i.d.~random variables, again related to Poisson point processes, which may be of independent interest.

In this paper we studied some probabilistic properties of random concave functions defined by the model \eqref{eqn:our.model}. To address the applications described in Section \ref{sec:motivations} we need to develop efficient computational methods; the model \eqref{eqn:our.model} may also need to be modified to suit the specific needs. We plan to study these questions in future research.

\section{Concave functions on the simplex} \label{sec:model}
\subsection{Preliminaries} \label{sec:notations}
As noted in Section \ref{sec:intro} we will focus on the space $\mathcal{C}$, defined by \eqref{eqn:family.C}, consisting of non-negative continuous concave functions on $\overline{\Delta}_n$. We also let
\begin{equation*}
\mathcal{C}_+ := \{ \psi \in \mathcal{C} : \psi > 0 \text{ on } \Delta_n \}
\end{equation*}
be those functions in $\mathcal{C}$ that are strictly positive in the (relative) interior.

Our choice of using the simplex as the domain has the following mathematical advantages apart from the motivations described above. First, the simplex $\overline{\Delta}_n$ is a symmetric polyhedron, and in this case the duality of concave function takes a special form which is useful for our analysis. Second, if we specify a finite number of points $p^{(i)} \in \Delta_n$ and constants $r^{(i)} > 0$, the smallest function $\psi \in \mathcal{C}$ such that $\psi(p^{(i)}) \geq r^{(i)}$ for all $i$ is polyhedral, i.e., it is the minimum of a finite collection of hyperplanes. This is not the case if the boundary is smooth. Last but not least, the duality allows us to connect the limiting distributions of our model with Poisson point processes on the positive quadrant. While it may be possible to extend some results to general convex domains, we believe the unit simplex is of special interest.

Functions in $\mathcal{C}$ enjoy strong analytical properties (we refer the reader to \cite{R70} for standard results in convex analysis). For example, if $\psi \in \mathcal{C}$, then $\psi$ is locally Lipschitz on $\Delta_n$. Moreover, the superdifferential
\[
\partial \psi(p) := \{\xi \in \mathbb{R}^n: \xi_1 + \cdots + \xi_n = 0, \ \psi(p) + \langle \xi, q - p \rangle \geq \psi(q) \text{ for all } q \in \overline{\Delta}_n \}
\]
is non-empty, convex and compact for every $p \in \Delta_n$; moreover $\psi$ is differentiable (i.e., the superdifferential $\partial \psi(p)$ reduces to a singleton) Lebesgue almost everywhere on $\Delta_n$. By Aleksandrov's theorem (see e.g.~\cite[Theorem 6.9]{EG15}) even the Hessian can be defined almost everywhere, but this result is not needed in this paper.

We equip the space $\mathcal{C}$ with the topology of local uniform convergence. By definition, a sequence $\{\psi_k\}$ converges to $\psi$ in $\mathcal{C}$ if and only if for any compact subset $\Omega$ of $\Delta_n$ we have $\psi_k \rightarrow \psi$ uniformly on $\Omega$. A metric which is compatible with this topology is
\begin{equation} \label{eqn:C.metric}
d(\varphi, \psi) := \sum_{k = 1}^{\infty} 2^{-k} \left(\sup_{p \in \Delta_{n, k}} | \varphi(p) - \psi(p) | \wedge 1 \right),
\end{equation}
where $ \Delta_{n, k} = \{p \in \Delta_n : p_i \geq 1/k, i = 1, \ldots, n\}$ is compact in $\Delta_n$. Note that by \cite[Theorem 10.3]{R70}, any non-negative concave function on $\Delta_n$ has a unique continuous extension to $\overline{\Delta}_n$. This implies that if $\varphi, \psi \in \mathcal{C}$ and $d(\varphi, \psi) = 0$ then $\varphi \equiv \psi$ on $\overline{\Delta}_n$; thus the metric is well-defined on $\mathcal{C}$ even though the boundary is not explicitly included in \eqref{eqn:C.metric}. It is easy to verify that $(\mathcal{C}, d)$ is complete and separable.  The following lemma is standard and a proof (which uses \cite[Theorem 10.6]{R70}) can be found in \cite{W15}.

\begin{lemma} \label{lem:compact}
	For any $M \geq 0$ the set $\{\psi \in \mathcal{C}: \psi \leq M\}$ is compact in $\mathcal{C}$.
\end{lemma}

Let $\mathcal{B} = \mathcal{B}(\mathcal{C})$ be the Borel $\sigma$-field generated by this topology. In this paper we are interested in probabilistic models for generating random elements of $\mathcal{C}$, or equivalently  probability measures on $(\mathcal{C}, \mathcal{B})$. It is easy to see that $\mathcal{B}$ is generated by the collection of finite-dimensional cylinder sets. This implies the following lemma.

\begin{lemma} \label{lem:unique}
	Let $\nu$ and $\widetilde{\nu}$ be probability measures on $\mathcal{C}$. If they have the same finite-dimensional distributions, then $\nu = \widetilde{\nu}$.
\end{lemma}

\begin{remark}
	Apart from the topology of uniform convergence over compact subsets of $\Delta_n$ as in \eqref{eqn:C.metric}, one may consider, for example, the topology of uniform convergence on $\overline{\Delta}_n$. We argue that our choice is more natural, and the main reason is that convergence theorems in convex analysis (such as \cite[Theorem 10.9]{R70}) are usually formulated in the topology of local uniform convergence. To give a concrete example, consider on $[0, 1]$ the sequence $\{\psi_k\}_{k \geq 2}$ of concave functions given by
	\[
	\psi_k(x) = \left\{\begin{array}{ll}
	kx, & \text{for } 0\leq x\leq \frac{1}{k};\\
	1, & \text{for } \frac{1}{k} \leq x \leq 1 - \frac{1}{k};\\
	k(1 - x), & \text{for } \leq 1 - \frac{1}{k} \leq x \leq 1.
	\end{array}\right.
	\]
	Then $\psi_k$ converges with respect to metric $d$, but not uniformly, to the constant function $\psi(x) \equiv 1$.
\end{remark}

Let $\mathcal{A}_+$ denote the set
\[
\mathcal{A}_+ := \{\ell : \overline{\Delta}_n \rightarrow (0, \infty) \text{ affine}\}
\]
consisting of (strictly) positive affine functions on $\overline{\Delta}_n$. Clearly $\mathcal{A}_+ \subset \mathcal{C}_+ \subset \mathcal{C}$. Note that every element of $\mathcal{A}_+$ can be written in the form
\begin{equation} \label{eqn:affine}
\ell(p) = \sum_{i = 1}^n p_i x_i =: \langle p, x \rangle,
\end{equation}
for some positive constants $x_1, \ldots, x_n > 0$, where $x_i = \ell(e_i)$ is the value of $\ell$ at the vertex $e_i$. Thus we may identify $\mathcal{A}_+$ with the positive quadrant $\mathbb{R}_{+}^n := (0, \infty)^n$. By concavity, for any $\psi \in \mathcal{C}$ we have
\begin{equation} \label{eqn:concave.duality}
\psi = \inf\{\ell \in \mathcal{A}_+ : \ell \geq \psi\}.
\end{equation}
Since every element of $\mathcal{C}$ can be written as the infimum of a collection of hyperplanes, to generate a random concave function in $\mathcal{C}$ it suffices to generate a random collection of hyperplanes in $\mathcal{A}_+$.

\medskip

While every concave function in $\mathcal{C}$ can be generated in the form \eqref{eqn:concave.duality}, in applications (e.g.~in stochastic portfolio theory) it may be desirable to use a smooth approximation of the minimum operation, so that each realization is itself smooth (when the number of planes is finite). For this reason we introduce the softmin which is often used in convex optimization and machine learning. The smoothness parameter also adds an extra dimension to the mathematical analysis.

\begin{definition}[Softmin] \label{def:soft.min}
	Let $\lambda > 0$. For $K \geq 1$ and $x_1, \ldots, x_K \in \mathbb{R}$ we define the softmin (with parameter $\lambda$) by
	\begin{equation} \label{eqn:softmin}
	\mathsf{m}_{\lambda}(x_1, \ldots, x_K) := \frac{-1}{\lambda} \log \left( \frac{1}{K} \sum_{k = 1}^K e^{-\lambda x_k} \right).
	\end{equation}
	By continuity, we define
	\[
	\mathsf{m}_{\infty}(x_1, \ldots, x_K) := \min\{x_1, \ldots, x_K\}
	\]
	and call this the hardmin (see Figure \ref{fig:softmin} for an illustration). We also write $\mathsf{m}_{\lambda}(x_1, \ldots, x_K) = \mathsf{m}_{\lambda}\{x_k\}$ when the context is clear.
\end{definition}

\begin{figure}[t]
	\centering
	\includegraphics[scale=0.45]{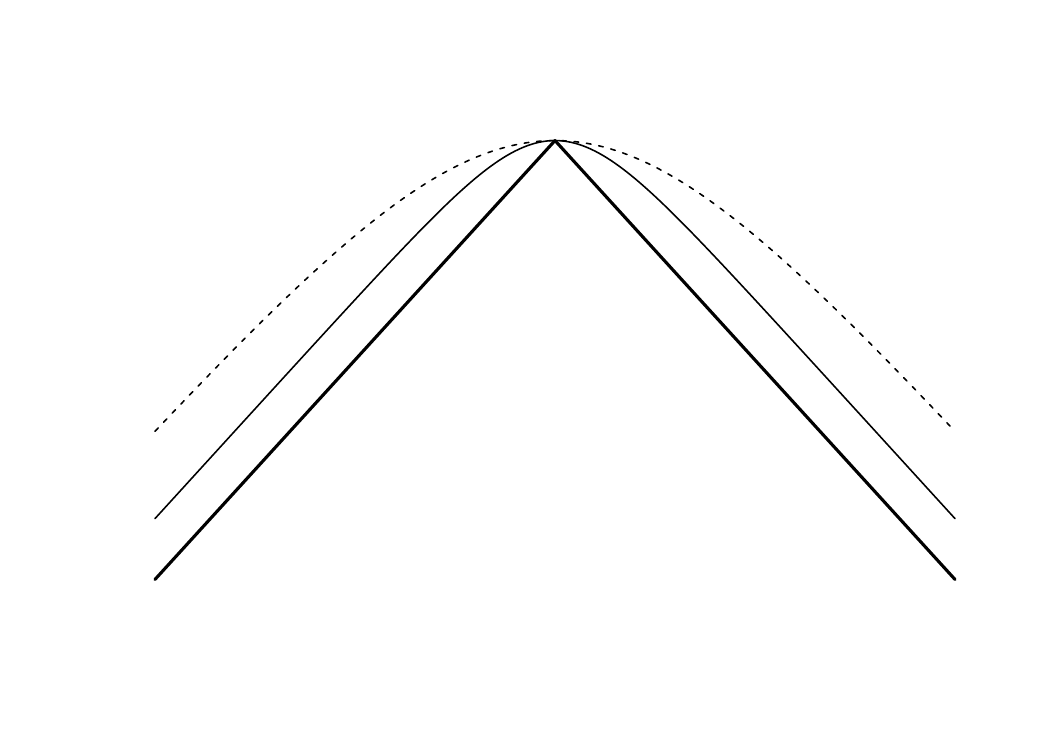}
	\vspace{-1cm}
	\caption{Graphs of $\mathsf{m}_{\lambda}(x, 2 - x)$ on $[0, 2]$ for $\lambda = 2$ (dashed), $\lambda = 5$ (thin solid) and $\lambda = \infty$ (thick solid).} \label{fig:softmin}
\end{figure}

\begin{lemma}[Properties of softmin] \label{lem:softmin} Let $\lambda > 0$ be fixed.
	\begin{itemize}
		\item[(i)] For $x_1, \ldots, x_K \in \mathbb{R}$ we have
		\begin{equation} \label{eqn:softmin.bound}
		\min\{x_1, \ldots, x_K\} \leq \mathsf{m}_{\lambda}(x_1, \ldots, x_K) \leq \min\{x_1, \ldots, x_K\} + \frac{1}{\lambda} \log K.
		\end{equation}
		Also, for any $x \in \mathbb{R}$ we have $\mathsf{m}_{\lambda}(x, \ldots, x) = x$ and
		\[
		\mathsf{m}_{\lambda}(x_1 + c, \ldots, x_K + c) = \mathsf{m}_{\lambda}(x_1, \ldots, x_K) + c, \quad c \in \mathbb{R}.
		\]
		\item[(ii)] For $K \geq 1$ fixed, the softmin $\mathsf{m}_{\lambda}$ is a smooth and symmetric concave function of $x_1, \ldots, x_K$.
		\item[(iii)] If $\Phi^{(1)}, \ldots, \Phi^{(K)}$ are finite concave functions defined on a convex set, then so is
		\begin{equation} \label{eqn:softmin.concave}
		\Phi = \mathsf{m}_{\lambda}(\Phi^{(1)}, \ldots, \Phi^{(K)}).
		\end{equation}
	\end{itemize}
\end{lemma}
\begin{proof}
	All statements can be proved by elementary means, and for completeness we give the proof of (iii). First we observe that if $y_k \geq x_k$ for all $k$ (possibly after a permutation of the elements), then
	\[
	\mathsf{m}_{\lambda} \{y_k\}  \geq \mathsf{m}_{\lambda} \{x_k\}.
	\]
	Let $p$, $q$ be elements of the domain of the $\Phi^{(k)}$, and let $0 < \alpha < 1$. By the previous remark as well as (ii) and the concavity of the $\Phi^{(k)}$, we have
	\begin{equation*}
	\begin{split}
	\Phi((1 - \alpha) p + \alpha q) &= \mathsf{m}_{\lambda}\{ \Phi^{(k)}( (1 - \alpha) p + \alpha q) \} \\
	&\geq \mathsf{m}_{\lambda}\{ (1 - \alpha) \Phi^{(k)}(p) + \alpha \Phi^{(k)}(q)\} \\
	&\geq (1 - \alpha) \mathsf{m}_{\lambda}\{\Phi^{(k)}(p)\} +\alpha \mathsf{m}_{\lambda}\{\Phi^{(k)}(q)\} \\
	&= (1 - \alpha) \Phi(p) + \alpha \Phi(q).
	\end{split}
	\end{equation*}
	This proves that the softmin $\Phi$ is concave as well.
\end{proof}

By Lemma \ref{lem:softmin}, if $\Phi^{(1)}, \ldots, \Phi^{(K)} \in \mathcal{C}$ then so is $\Phi = \mathsf{m}_{\lambda}(\Phi^{(1)}, \ldots, \Phi^{(K)})$. Moreover, by the differentiability of the softmin, if each $\Phi^{(k)}$ is differentiable, then so is $\Phi$.

\begin{remark}
	As explained in Section \ref{sec:spt}, every element of $\mathcal{C}$ can be regarded as a portfolio generating function. Suppose $\Phi^{(1)}, \ldots, \Phi^{(K)} \in \mathcal{C}$ are differentiable, and let $\Phi = \mathsf{m}_{\lambda}(\Phi^{(1)}, \ldots, \Phi^{(K)})$ be their softmin. Also let $\boldsymbol{\pi}^{(k)}$ be the portfolio map generated by $\Phi^{(k)}$ in the sense of \eqref{eqn:fgp}. Then it can be shown by a straightforward computation that the portfolio map generated by $\Phi$ is given by
	\begin{equation} \label{prop:softmin.portfolio}
	\boldsymbol{\pi}(p) = \left(1 - \sum_{k = 1}^K a_k(p)\right) p + \sum_{k = 1}^K a_k(p) \boldsymbol{\pi}^{(k)}(p), \quad p \in \Delta_n,
	\end{equation}
	where
	\begin{equation} \label{prop:softmin.portfolio.coeff}
	a_k(p) = \frac{e^{-\lambda \Phi_k(p)}}{\sum_{\ell = 1}^K e^{-\lambda \Phi_{\ell}(p)}} \cdot \frac{\Phi_k(p)}{\Phi(p)} = \frac{\Phi_k(p) e^{-\lambda \Phi_k(p)}}{K  \Phi(p) e^{-\lambda \Phi(p)}}, \quad k = 1, \ldots, K.
	\end{equation}
	Thus $\boldsymbol{\pi}$ is a linear combination of the market portfolio and the portfolios generated by $\{\Phi^{(k)}\}$. In the limiting case $\lambda \rightarrow \infty$ (i.e., the hardmin), \eqref{prop:softmin.portfolio} gives
	\[
	\boldsymbol{\pi}(p) = \boldsymbol{\pi}^{(k)}(p), \quad \text{if } \Phi^{(k)}(p) < \Phi^{(\ell)}(p) \text{ for } \ell \neq k.
	\]
\end{remark}

To conclude this subsection, we state some estimates of concave functions which are useful in several results below. The proof is given in the appendix.

\begin{lemma} \label{lem:concave.bound} { \ }
	\begin{enumerate}
		\item[(i)] For each $q \in \Delta_n$ there exists an explicit constant $M_q > 0$ such that
		\begin{equation} \label{eqn:concave.bound}
		\psi(p) \leq M_q \psi(q) \quad \mbox{ for all } p \in \overline{\Delta}_n  \mbox{ and all } \psi \in \mathcal{C}.
		\end{equation}
		
		\item[(ii)] For $1 \leq j \leq n$ and $0 < \epsilon \leq \frac{1}{n}$ we define $c^{(j, \epsilon)} \in \Delta_n$ by
		\begin{equation} \label{eqn:c.j.epsilon}
		c^{(j, \epsilon)}_i = \left\{ \begin{array}{cl} \epsilon & i = j, \\[1ex]
		\frac{1- \epsilon}{n-1} & i \neq j.
		\end{array} \right.
		\end{equation}
		Thus $c^{(j, \epsilon)}$ is the center of the slice $p_j = \epsilon$ through the simplex $\Delta_n$. Suppose $p \in \overline{\Delta}_n$ satisfies $0 \le  p_j \leq \epsilon \leq 1/n$ for some $j$.  Then
		\begin{equation} \label{eqn:concave.bound2}
		\psi(p)  \le n\psi(c^{(j,\epsilon)}) \quad \mbox{ for all }  \psi \in \mathcal{C}.
		\end{equation}
		In particular we have
		\[
		\psi(p) \leq n \psi(\overline{e}) \quad \mbox{ for all } p \in \overline{\Delta}_n  \mbox{ and } \psi \in \mathcal{C}.
		\]
	\end{enumerate}
\end{lemma}

\subsection{The probabilistic model} \label{sec:model.specifics}
In this paper we study a natural implementation of the representation \eqref{eqn:concave.duality}. Namely, we consider random concave functions given as (soft) minimums of i.i.d.~random hyperplanes.

Let $C = (C_1, \ldots, C_n)$ be a random vector with values in the quadrant $\mathbb{R}_{+}^n$. Note that the components of $C$ may be dependent. Throughout the paper we let $(\Omega, \mathcal{F}, \mathbb{P})$ be a probability space on which the required random elements are defined. Given $C$, we define a random element $\ell$ of $\mathcal{A}_+$ given by
\begin{equation} \label{eqn:random.affine}
\ell(p) = \langle C, p \rangle = \sum_{i = 1}^n p_i C_i, \quad p \in \overline{\Delta}_n.
\end{equation}

For $K = 1, 2, \ldots$, let $\ell_1, \ell_2, \ldots, \ell_K$ be independent copies of $\ell$. Define a random concave function $\Psi_K$ by
\begin{equation} \label{eqn:phi}
\Psi_K = a_K \mathsf{m}_{\lambda}(\ell_1, \ldots, \ell_K),
\end{equation}
where $\lambda = \lambda_K \in (0, \infty]$ possibly depends on $K$ and $a_K > 0$ is a scaling constant to be chosen. The law of $\Psi_K$ defines a probability measure $\nu_K$ on $\mathcal{C}$ which depends on $K$, $\lambda_K$, $a_K$ and the distribution of $C$. We are interested in the limiting behavior of $\nu_K$ as $K \rightarrow \infty$.

\subsection{Deterministic limit for softmin}\label{sec:as.limit}


To give quickly a concrete result, in this subsection we consider the model \eqref{eqn:phi} where $0 < \lambda < \infty$ is fixed and independent of $K$, and there is no scaling, i.e., $a_K \equiv 1$. Using the strong law of large numbers, we show that there is a deterministic almost sure limit as $K \rightarrow \infty$.

\begin{theorem} \label{thm:softmin.fixed.lambda}
	Fix $\lambda > 0$. Let
	\begin{equation} \label{eqn:cgf}
	\psi(t_1, \ldots, t_n) := \log \mathbb{E} e^{t_1 C_1 + \cdots + t_n C_n}
	\end{equation}
	be the cumulant generating function of $C = (C_1, \ldots, C_n)$ which is finite and convex on $(-\infty, 0]^n$. Let $\{\ell_k\}_{k = 1}^{\infty}$ be a sequence of independent copies of $\ell$ as in \eqref{eqn:random.affine}, and, for each $K \geq 1$, let $\Psi_K$ be the random concave function defined by $
	\Psi_K = \mathsf{m}_{\lambda}\left(\ell_1, \ldots, \ell_K\right)$. Then
	\[
	\lim_{K \rightarrow \infty} \Psi_K = \Psi_{\infty} \quad \text{$\mathbb{P}$-a.s.},
	\]
	where $\Psi_{\infty} \in \mathcal{C}_+$ is the deterministic concave function given by
	\begin{equation} \label{eqn:softmin.limit}
	\Psi_{\infty}(p) =  \frac{-1}{\lambda} \psi(-\lambda p), \quad p \in \Delta_n.
	\end{equation}
	Here the convergence means that $d(\Psi_K, \Psi_{\infty}) \rightarrow 0$ a.s., where $d$ is the metric on $\mathcal{C}$ defined by \eqref{eqn:C.metric}.
\end{theorem}
\begin{proof}
	Let $p \in \Delta_n$ be fixed. By definition, we have
	\[
	\Psi_K(p) = \frac{-1}{\lambda} \log \left( \frac{1}{K} \sum_{k = 1}^K e^{-\lambda \ell_k(p)}\right).
	\]
	By the strong law of large numbers, we have the almost sure limit
	\begin{equation*}
	\begin{split}
	\lim_{K \rightarrow \infty} \frac{1}{K} \sum_{k = 1}^K e^{-\lambda \ell_k(p)} = \mathbb{E} e^{-\lambda \sum_{i = 1}^n C_i p_i} =  e^{\psi(-\lambda p)}.
	\end{split}
	\end{equation*}
	Taking logarithm and dividing by $-\lambda$ shows that $\Psi_K(p) \rightarrow \Psi_{\infty}(p)$ almost surely. The pointwise convergence then holds, with probability $1$, over a countable dense subset of $\Delta_n$. By \cite[Theorem 10.8]{R70} we have uniform convergence over compact subsets which implies convergence in the metric $d$.
\end{proof}

\begin{example} \label{ex:softmin.limit}
	In the context of Theorem \ref{thm:softmin.fixed.lambda}, suppose $C_1, \ldots, C_n$ are i.i.d.~exponential random variables with rate $\alpha > 0$. The cumulant generating function is given by
	\[
	\psi(t) =  \sum_{i = 1}^n \log \frac{\alpha}{\alpha - t_i}.
	\]
	It follows that the limiting function \eqref{eqn:softmin.limit} is given by
	\begin{equation} \label{eqn:exp.limit}
	\Psi_{\infty}(p) = \sum_{i = 1}^n \frac{1}{\lambda} \log \left(1 + \frac{\lambda}{\alpha}{p_i}\right).
	\end{equation}
	Some samples from this model are given in Figure \ref{fig:softminsamples}.
	
	\begin{figure}[t]
		\centering
		\includegraphics[scale=0.45]{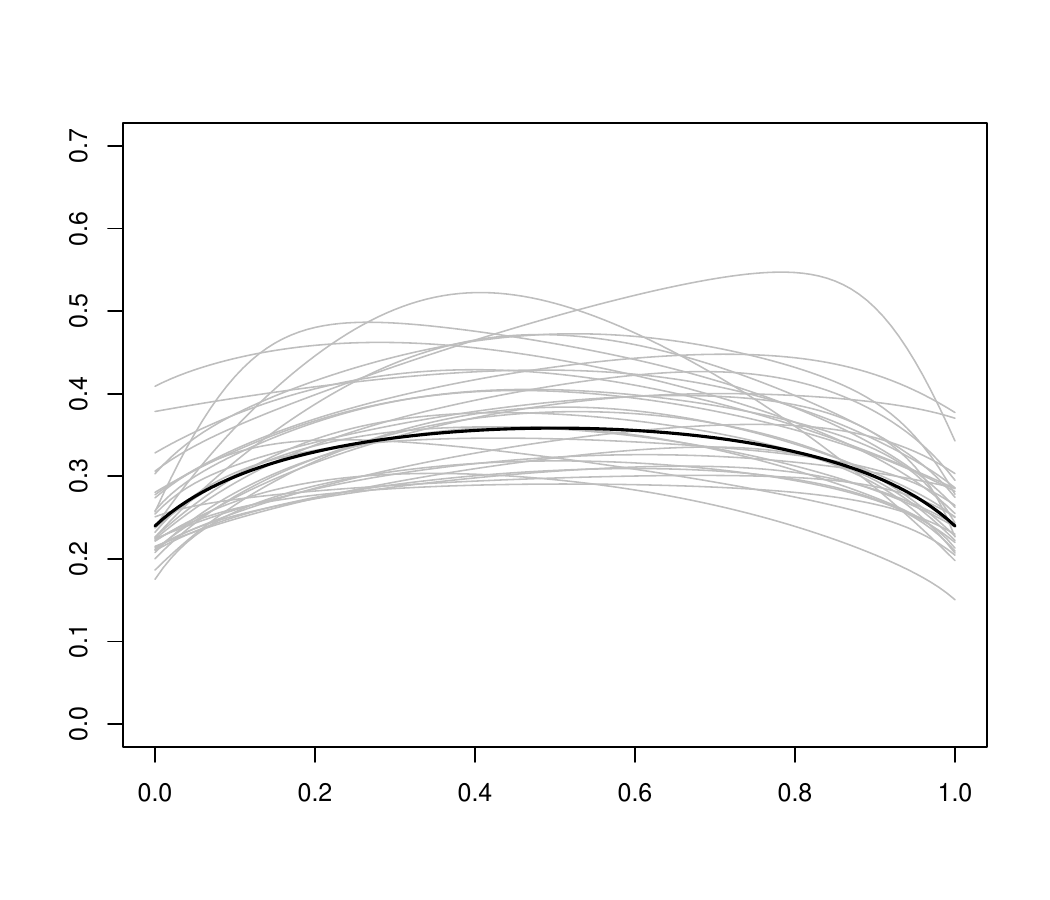}
		\vspace{-0.6cm}
		\caption{Samples (in grey) of $\Psi_K$ in Example \ref{ex:softmin.limit}. Here $n = 2$ (so that the simplex can be identified with the unit interval), $\alpha = 1$, $\lambda = 10$ and $K = 30$. The limiting function $\Psi_{\infty}$ is shown by the thick black curve.} \label{fig:softminsamples}
	\end{figure}
	
	An interesting question is what happens when $\lambda \rightarrow \infty$. From \eqref{eqn:fgp}, $\Psi$ and $c \Psi$ generate the same portfolio map for any $c > 0$. Thus we consider instead the limit of $\Psi_{\infty}(p) / \Psi_{\infty}(\overline{e})$ as $\lambda \rightarrow \infty$ (recall that $\overline{e}$ is the barycenter of $\Delta_n$). It turns out that
	\begin{equation} \label{eqn:lambda.limit}
	\lim_{\lambda \rightarrow \infty} \frac{\Psi_{\infty}(p)}{\Psi_{\infty}(\overline{e})} = \lim_{\lambda \rightarrow \infty} \frac{\sum_{i = 1}^n \log \left(1 + \frac{\lambda}{\alpha}{p_i}\right)}{n \log \left(1 + \frac{\lambda}{\alpha}\frac{1}{n}\right)} = 1, \quad p \in \Delta_n.
	\end{equation}
	Accordingly, as $\lambda \rightarrow \infty$ the corresponding portfolio converges to the market portfolio $\boldsymbol{\pi}(p) \equiv p$. This result suggests that the limits $\lim_{\lambda \rightarrow \infty} \lim_{K \rightarrow \infty}$ and $\lim_{K \rightarrow \infty} \lim_{\lambda \rightarrow \infty}$ are different in our model; the difference will become clear in the next section.
\end{example}


\section{Weak limit for hardmin} \label{sec:limit.lambda.inf}
Now we consider the case of hardmin $\mathsf{m}_{\infty} = \min$ so that $\lambda_K = \infty$ for all $K$. We show that a suitable scaling gives a non-trivial limiting distribution. This distribution can be characterized in terms of its tail probability, i.e., $\mathbb{P}\left(\Psi \geq \psi\right)$ for $\psi \in \mathcal{C}$, and can be realized by a Poisson point process via a novel duality.

\subsection{Main results}
In this section we impose the following conditions on the random vector $C$ which defines the random hyperplane $\ell(\cdot) = \langle C, \cdot\rangle$.

\begin{assumption} \label{ass:density}
	The random vector $C = (C_1, \ldots, C_n)$ has a joint density $\rho$ on $\mathbb{R}_{+}^n$ which is asymptotically homogeneous of order $\alpha$ near the origin. More precisely, there exist $\alpha \in \mathbb{R}$ and a non-negative measurable function $h$ on $\mathbb{R}_{+}^n$ such that $\int_{\mathbb{R}_+^n} h(x) dx \in (0, \infty]$ and
	\begin{equation} \label{eqn:origin}
	\lim_{\kappa \rightarrow 0^+} \frac{1}{\kappa^{\alpha}}\rho(\kappa x) =  h(x), \quad x \in \mathbb{R}_+^n,
	\end{equation}
	uniformly for $x \in \Delta_n$.
\end{assumption}

\begin{remark}
	Since we use the hardmin, as $K$ grows the (unnormalized) minimum $\min\{ \ell_1, \ldots, \ell_K \}$ becomes smaller and smaller. Consequently, the weak limit of the scaled minimum, if exists, only depends on the distribution of $C$ in a neighborhood of the origin. This consideration motivates Assumption \ref{ass:density}.
\end{remark}

Assumption \ref{ass:density} imposes rather strong conditions on the function $h$. In the next lemma we gather some properties that are used subsequently. The proof is given in the Appendix.

\begin{lemma} \label{lem:assumption.consequences}
	Under Assumption \ref{ass:density} the exponent $\alpha$ is uniquely determined and $\alpha > -n$. Also $h(\kappa x) = \kappa^{\alpha} h(x)$ for all $\kappa > 0$ and $x \in \mathbb{R}_+^n$, so that $h$ is homogeneous of order $\alpha$. Moreover, we have
	\begin{equation} \label{eqn:h.integral}
	\lim_{\kappa \rightarrow 0^+} \frac{1}{\kappa^{\alpha}} \int_A \rho(\kappa x) dx = \int_A h(x) dx < \infty
	\end{equation}
	for every bounded Borel set $A \subset \mathbb{R}_+^n$. Thus $h$ is locally integrable.
\end{lemma}

Here we give some examples of random vectors that satisfy Assumption \ref{ass:density}.

\begin{example} \label{ex:density} {\ }
	\begin{enumerate}
		\item[(i)] Suppose $C$ has a continuous density $\rho$ on $\mathbb{R}_+^n$ with $\lim_{x \rightarrow 0} \rho(x)  = \gamma > 0$. Then \eqref{eqn:origin} holds with $\alpha = 0$ and $h(x) \equiv \gamma$.
		\item[(ii)] Suppose $C_1, \ldots, C_n$ are independent and $C_i$ has the gamma distribution with shape parameter $1 + \alpha_i > 0$ and scale parameter $\beta_i$. The density of $C$ is given by
		\[
		\rho(x) = \prod_{i = 1}^n \frac{{\beta_i}^{1+\alpha_i}}{\Gamma(1+\alpha_i)} x_i^{\alpha_i} e^{-\beta_i x_i}.
		\]
		Then \eqref{eqn:origin} holds with $\alpha = \sum_{i = 1}^n \alpha_i $ and
		\[
		h(x) = \prod_{i = 1}^n \frac{\beta_i^{1+\alpha_i}}{\Gamma(1+\alpha_i)} x_i^{\alpha_i}.
		\]
	\end{enumerate}
\end{example}

Our first main result is that a weak limit exists under a suitable scaling. 

\begin{theorem} \label{thm:weak.convergence}
	Under Assumption \ref{ass:density}, as $K \rightarrow \infty$ the distribution $\nu_K$ of the random function $\Psi_K := K^{\frac{1}{n+\alpha}} \min\{  \ell_1, \ldots, \ell_K \}$ converges weakly to a probability measure $\mu$ supported on $\mathcal{C}_+ \subset \mathcal{C}$.
\end{theorem}  

Further properties of the limiting distribution $\mu$ will be studied in this and the next sections. To help the reader navigate through the hardmin case, here are pointers to the main results:
\begin{itemize}
\item In Proposition \ref{prop:finite.dimensional} we prove weak convergence of the finite dimensional distributions. In particular, if $\Psi \sim \mu$, then for $p \in \Delta_n$ fixed the random variable $\Psi(p)^{n + \alpha}$ is exponentially distributed. We compute $\mathbb{E}[ \Psi]$ in Corollary \ref{cor:one.point}.
\item In Theorem \ref{thm:tail.probability} we characterize $\mu$ in terms of the tail probability functional $\mathsf{T}_{\mu}(\psi) = \mu\{\omega \in \mathcal{C}: \omega \geq \psi\}$, $\psi \in \mathcal{C}$. Explicit formulas are given for an important special case in Theorems \ref{thm:stoke} and \ref{thm:stokes2}.
\item In Theorem \ref{thm:Poisson} we provide an explicit representation of the limit distribution via a Poisson point process on the positive quadrant $\mathbb{R}_+^n$. Consequently, if $\Psi \sim \mu$ then $\Psi$ is locally piecewise affine.
\item Further properties, including boundary behaviors and connections to portfolio maps, are given in Theorem \ref{thm:boundary.behavior} and Proposition \ref{prop:expectation.portfolio}.
\end{itemize}

\begin{remark} [Connections with extreme value theory] \label{rmk:EVT1}
Theorem \ref{thm:weak.convergence} is a limit theorem for the minimum of (properly rescaled) i.i.d.~hyerplanes over the simplex. As shown in the classic references \cite{LLR82, R87, DF07}, extrema of i.i.d.~random variables, random vectors and stochastic processes have been studied extensively in extreme value theory. It is customary in this theory to consider maxima instead of minima. Expressing our results in terms of maximums (by considering instead negative convex functions), we have that the marginal $-\Psi(p)$, for $p \in \Delta_n$, has a Weibull-type distribution with extreme value index $-1$ (see for example \cite[Theorem 1.1.3]{DF07}). To the best of our knowledge, the constraint of concavity and the unit simplex have not been considered in the literature. In particular, by using an elegant duality for positive concave functions on $\Delta_n$, to be described in the next subsection, we are able to obtain deep properties of the limiting distribution. These results suggest that concave functions on the unit simplex are natural and fundamental objects to study. Also see Remark \ref{rmk:EVT2} for more discussions about the Poisson representation (Theorem \ref{thm:Poisson}). Finally, we remark that our general framework \eqref{eqn:phi} involves the softmin which is rarely, if at all, considered in extreme value theory.
\end{remark}

\subsection{Duality for non-negative concave functions on $\overline{\Delta}_n$} \label{sec:duality}
Duality plays a major role in the proof of Theorem \ref{thm:weak.convergence} and several other results of this paper. For non-negative concave functions on the simplex the duality has an elegant form which is useful for our analysis. The reason is that each positive affine function on $\overline{\Delta}_n$ is specified by its values over the vertices $e_1, \ldots, e_n$ (see \eqref{eqn:affine} and Figure \ref{fig:planes}), so we do not need to specify the constant term separately.

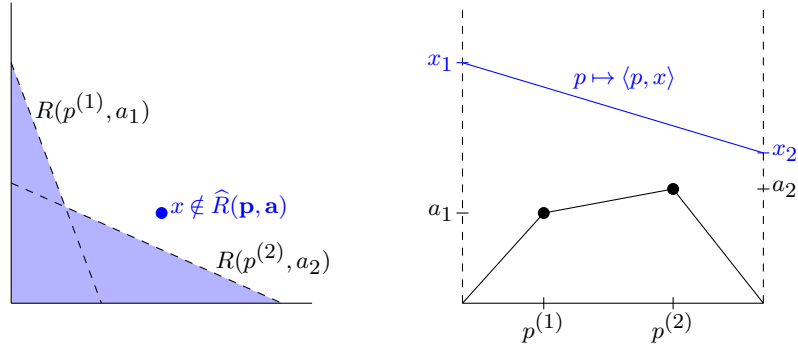
\begin{figure}[t]
	\begin{tikzpicture}[scale = 0.4]
	
	\fill[blue!30] (0, 0) to (0, 8) to (3, 0);
	\fill[blue!30] (0, 0) to (0, 4) to (9, 0);
	\draw (0, 0) to (0, 10);
	\draw (0, 0) to (10, 0);
	\draw[dashed] (0, 8) to (3, 0);
	\draw[dashed] (0, 4) to (9, 0);
	
	\node [black, right] at (0.5, 6.5) {{\footnotesize $R(p^{(1)}, a_1)$}};
	\node [black, right] at (6.5, 1.5) {{\footnotesize $R(p^{(2)}, a_2)$}};
	
	\draw[blue, fill = blue, radius = 0.18] (5, 3) circle;
	\node[blue, right] at (5, 3.3) {{\footnotesize $x \notin \widehat{R}({\bf p}, {\bf a})$}};
	
	\draw (15, 0) to (25, 0);
	\draw[dashed] (15, 0) to (15, 10);
	\draw[dashed] (25, 0) to (25, 10);
	
	
	\draw[fill = black, radius = 0.18] (17.7, 3) circle;
	\draw (17.7, -0.2) -- (17.7, 0.2);
	\node[black, below] at (17.7, 0) {{\footnotesize $p^{(1)}$}};
	\draw (14.8, 3) -- (15.2, 3);
	\node[black, left] at (15, 3) {{\footnotesize $a_1$}};
	
	\draw[fill = black, radius = 0.18] (22, 3.8) circle;
	\draw (22, -0.2) -- (22, 0.2);
	\node[black, below] at (22, 0) {{\footnotesize $p^{(2)}$}};
	\draw (24.8, 3.8) -- (25.2, 3.8);
	\node[black, right] at (25, 3.8) {{\footnotesize $a_2$}};
	
	\draw (15, 0) -- (17.7,3) -- (22, 3.8) -- (25, 0);
	
	\draw[blue] (15, 8) -- (25, 5);
	\draw[blue] (14.8, 8) -- (15.2, 8);
	\draw[blue] (24.8, 5) -- (25.2, 5);
	\node[blue, left]  at (15, 8) {{\footnotesize $x_1$}};
	\node[blue, right] at (25, 5) {{\footnotesize $x_2$}};
	\node[blue, above] at (20.4, 6.8) {{\footnotesize $p \mapsto \langle p, x \rangle$}};

	\end{tikzpicture}
	\caption{Left: The region $\widehat{R}({\bf p}, {\bf a})$ when $r = 2$, and a point $x \notin \widehat{R}({\bf p}, {\bf a})$. Right: $x \notin \widehat{R}({\bf p}, {\bf a})$ if and only if $\langle p, x \rangle \geq \psi_{{\bf p}, {\bf a}}(p)$.} \label{fig:duality}
\end{figure}

For $p \in \Delta_n$ and $a > 0$, we denote by $R(p, a)$ the region
\begin{equation} \label{eqn:region.R}
R(p, a) := \left\{x \in \mathbb{R}_{+}^n : \langle p, x \rangle < a \right\},
\end{equation}
which is an open convex set in $\mathbb{R}_{+}^n$. Since the hyperplane $\langle p, x \rangle = a$ intersects the $i$th coordinate axis at $x_i = a/p_i$, the Euclidean volume of $R(p, a)$ is
\begin{equation} \label{eqn:vol.R}
\text{vol}(R(p, a)) = \frac{a^n}{n!p_1 p_2 \cdots p_n}.
\end{equation}

If ${\bf p} = (p^{(1)}, p^{(2)}, \ldots, p^{(r)})$ is a collection of $r$ distinct points in $\Delta_n$ and ${\bf a} = (a_1, \ldots, a_r)$ is a collection of positive real numbers, we define
\begin{equation} \label{eqn:R.hat}
\widehat{R}({\bf p}, {\bf a}) := R(p^{(1)}, a_1) \cup \cdots \cup R(p^{(r)}, a_r).
\end{equation}
Identifying $x \in \mathbb{R}_{+}^n$ with the positive affine function $p \mapsto \langle p, x \rangle$ on $\Delta_n$, we see that $x \notin \widehat{R}({\bf p}, {\bf a})$ if and only if $\langle p^{(i)}, x\rangle \geq a_i$ for all $i$. This implies that $\langle p, x \rangle \geq \psi_{{\bf p}, {\bf a}}(p)$ on $\overline{\Delta}_n$, where
\begin{equation} \label{eqn:psi.p.a}
\psi_{{\bf p}, {\bf a}} := \inf\{\psi \in \mathcal{C}: \psi(p^{(i)}) \geq a_i, i = 1, \ldots, r\}
\end{equation}
is the smallest non-negative concave function generated by the given data. See  Figure \ref{fig:duality} for an illustration where $n = 2$ and we identify $\overline{\Delta}_2$ with the interval from $e_1$ to $e_2$.  Since the domain is the unit simplex the function $\psi_{{\bf p}, {\bf a}}$ is polyhedral, i.e., it is the minimum of a finite collection of hyperplanes.

More generally, for $\psi \in \mathcal{C}_+$ we define
\begin{equation} \label{eqn:R.hat.general}
\widehat{R}(\psi) := \bigcup_{p \in \Delta_n} R(p, \psi(p))
\end{equation}
and
\begin{equation} \label{eqn:S.hat}
\widehat{S}(\psi) := \mathbb{R}_{+}^n \setminus \widehat{R}(\psi) = \{x \in \mathbb{R}_{+}^n : \langle p, x \rangle \geq \psi(p) \text{ for all } p \in \Delta_n\}.
\end{equation}
Note that $ \widehat{S}(\psi)$ is a convex set, and by the identification $x \leftrightarrow (p \mapsto \langle p, x \rangle)$ we see that $\widehat{S}(\psi)$ is equivalent to the set $\{\ell \in \mathcal{A}_{+} : \ell \geq \psi\}$, so the operation $\psi \in \mathcal{C}_{+} \mapsto \widehat{S}(\psi)$ is one-to-one. Thus the set $\widehat{S}(\psi)$ plays the role of the conjugate. This duality can be formalized by adapting the concept of support function from convex analysis (see \cite[Section 13]{R70}).

\begin{proposition}
	For any $\psi \in \mathcal{C}_{+}$ we have
	$$
	\psi(p) = \inf_{x \in \widehat{S}(\psi)} \langle p,x \rangle, \quad \mbox{ for }p \in \overline{\Delta}_n.
	$$
\end{proposition}
\begin{proof}
	Let $x \in \widehat{S}(\psi)$. By definition we have $\langle p, x \rangle \geq \psi(p)$ for $p \in \Delta_n$, and by continuity the inequality extends to $\overline{\Delta}_n$. Taking the infimum over $x$ we have
	\begin{equation} \label{eqn:support.function}
	\inf_{x \in \widehat{S}(\psi)} \langle p,x \rangle \geq \psi(p), \quad \mbox{ for }p \in \overline{\Delta}_n. \end{equation}
	
	For the other direction, let $p \in \Delta_n$ be an interior point so that $\partial \psi(p)$ is non-empty. So there exists $x \in [0,\infty)^n$ such that $\langle q, x \rangle \geq \psi(q)$ for all $q \in \overline{\Delta}_n$ and $\langle p, x \rangle = \psi(p)$. If $x \in \mathbb{R}_{+}^n$, then $x \in \widehat{S}(\psi)$ and we have
	$$
	\inf_{y \in \widehat{S}(\psi)} \langle p,y \rangle \le \langle p,x \rangle = \psi(p).
	$$
	If not, then $x^{(r)} = x + \left(\frac{1}{r}, \ldots, \frac{1}{r}\right) \in \widehat{S}(\psi)$ for all $r \geq 1$,        and so
	$$
	\inf_{y \in \widehat{S}(\psi)} \langle p,y \rangle \le \liminf_{r \to \infty} \langle p,x^{(r)} \rangle =\langle p,x \rangle = \psi(p).
	$$
	Being the infimum of a collection of hyperplanes, $\inf_{x \in \widehat{S}(\psi)} \langle p, x \rangle$ is concave whose hypograph is closed. By \cite[Theorem 10.2]{R70}, it is continuous on $\overline{\Delta}_n$. Since $\psi \in \mathcal{C}_+$ is also continuous, the equality extends to the boundary.
\end{proof}

\subsection{Proof of Theorem \ref{thm:weak.convergence}}
Recall that $\Psi_K = K^{\frac{1}{n+\alpha}} \min_{1 \leq k \leq K} \ell_k$ is the normalized minimum of $K$ i.i.d.~hyperplanes. First we show convergence of the finite dimensional distributions.

\begin{proposition} \label{prop:finite.dimensional}
	Let $p^{(1)}, \ldots, p^{(r)}$ be a collection of distinct points in $\Delta_n$, and let $a_1, \ldots, a_r > 0$. Denote the data by $({\bf p}, {\bf a})$. Then
	\begin{equation} \label{eqn:finite.dimensional}
	\lim_{K \rightarrow \infty} \mathbb{P} \left( \Psi_K(p^{(i)}) \geq a_i, i = 1, \ldots, r\right) = \exp\left(- \int_{\widehat{R}({\bf p}, {\bf a})} h(x) dx \right),
	\end{equation}
	where $h$ is given by \eqref{eqn:origin}. Thus, the joint distribution of the random vector $\left( \Psi_K(p^{(i)}) \right)_{1 \leq i \leq r}$ converges weakly to the distribution defined by the right hand side of \eqref{eqn:finite.dimensional}.
	
	In particular, for $p \in \Delta_n$ fixed, let $\Psi(p)$ (which is a real-valued random variable) be distributed as the weak limit of $\Psi_K(p)$ as $K \rightarrow \infty$. Then the random variable $\Psi(p)^{n + \alpha}$ is exponentially distributed with rate $\int_{R(p, 1)} h(y) dy$.
\end{proposition}
\begin{proof}
	We have
	\begin{equation*}
	\begin{split}
	\mathbb{P} \left( \Psi_K(p^{(i)}) \geq a_i, \; i = 1, \ldots, r\right) &= \mathbb{P}\left( \min_{1 \leq k \leq K} \ell_k(p^{(i)}) \geq a_i K^{\frac{-1}{n + \alpha}} \; \forall i\right) \\
	&= \mathbb{P}\left( \ell (p^{(i)}) \geq a_i K^{ \frac{-1}{n + \alpha}} \; \forall i\right)^K.
	\end{split}
	\end{equation*}
	Note that the last inequality holds since the random vector $C$ (and hence $\ell(p) = \langle p, C \rangle $) has a density. We write
	\begin{equation} \label{eqn:finite.dimensional.dist.derivation}
	\mathbb{P}\left( \ell( p^{(i)}) \geq a_i K^{ \frac{-1}{n + \alpha}} \ \forall i\right)^K = \left(1 - \mathbb{P}\left(C \in K^{ \frac{-1}{n + \alpha}} \widehat{R}({\bf p}, {\bf a})\right)\right)^K.
	\end{equation}
	
	We claim that
	\begin{equation} \label{eqn:finite.dimensional.dist.claim}
	\lim_{K \rightarrow \infty} K \mathbb{P}\left(C \in K^{ \frac{-1}{n + \alpha}} \widehat{R}({\bf p}, {\bf a})\right) = \int_{\widehat{R}({\bf p}, {\bf a})} h(x) dx.
	\end{equation}
	Assuming \eqref{eqn:finite.dimensional.dist.claim}, we may take limit in \eqref{eqn:finite.dimensional.dist.derivation} to get
	\begin{equation*}
	\begin{split}
	\lim_{K \rightarrow \infty} \mathbb{P}\left( \Psi_K(p^{(i)}) \geq a_i, \quad i = 1, \ldots, r\right) &= \lim_{K \rightarrow \infty} \left(1 - \frac{1}{K} \int_{\widehat{R}({\bf p}, {\bf a})} h(x) dx  \right)^K \\
	&= \exp \left( -  \int_{\widehat{R}({\bf p}, {\bf a})} h(x) dx \right),
	\end{split}
	\end{equation*}
	which is the desired limit.
	
	To prove \eqref{eqn:finite.dimensional.dist.claim}, note that
	\begin{equation*}
	\begin{split}
	K \mathbb{P}\left(C \in K^{ \frac{-1}{n + \alpha}} \widehat{R}({\bf p}, {\bf a})\right) &= K \int_{ K^{ \frac{-1}{n + \alpha}} \widehat{R}({\bf p}, {\bf a}) } \rho(x) dx \\
	&= K \int_{  \widehat{R}({\bf p}, {\bf a}) } \rho(K^{ \frac{-1}{n + \alpha}} y) K^{ \frac{-n}{n + \alpha}}dy \\
	&= \int_{  \widehat{R}({\bf p}, {\bf a}) } \frac{\rho(K^{ \frac{-1}{n + \alpha}} y)}{K^{\frac{-\alpha}{n + \alpha}}} dy.
	\end{split}
	\end{equation*}
	By Lemma \ref{lem:assumption.consequences}, this converges to the integral of $h$ over  $\widehat{R}({\bf p}, {\bf a})$.
	
	Now let $p \in \Delta_n$ be given. Applying the above with $r = 1$, for $x > 0$ we have
	\[
	\mathbb{P}(\Psi(p) \geq x) = \exp\left( - \int_{R(p, x)} h(y) dy\right).
	\]
	Since $R(p, x) = x R(p, 1)$ and $h$ is homogeneous with order $\alpha$, we have
	\[
	\int_{R(p, x)} h(y) dy = x^{n + \alpha} \int_{R(p, 1)} h(y) dy.
	\]
	Thus
	\[
	\mathbb{P}(\Psi(p)^{n + \alpha} \geq x) = \mathbb{P}(\Psi(p) \ge x^{\frac{1}{n+\alpha}})= \exp \left( - x \int_{R(p, 1)} h(y) dy\right),
	\]
	i.e., $\Psi(p)^{n + \alpha}$ is exponentially distributed with rate $\int_{R(p, 1)} h(y) dy$.
\end{proof}

Let $\nu_K$ be the law of $\Psi_K$ regarded as random elements of the metric space $(\mathcal{C}, d)$. 

\begin{lemma} \label{lem:tightness}
	The sequence $\{\nu_K\}_{K \geq 1}$ is tight.
\end{lemma}
\begin{proof}
	By Lemma \ref{lem:concave.bound}, we have the bound
	\begin{equation} \label{eqn:uniform.estimate}
	\psi(p) \leq  n\psi(\overline{e})
	\end{equation}
	which holds for any $\psi \in \mathcal{C}$.
	
	Let $\epsilon > 0$ be given. By \eqref{eqn:finite.dimensional}, the family of univariate distributions corresponding to $\{ \Psi_K(\overline{e})\}_{K \geq 1}$ is tight. Thus, there exists $M > 0$ such that
	\[
	\mathbb{P}( \Psi_K(\overline{e}) \leq M) \geq 1 - \epsilon, \quad \forall  \; K \geq 1.
	\]
	
	By Lemma \ref{lem:compact}, the set $\mathcal{K} = \{\psi \in \mathcal{C}: \psi \leq n M\}$ is compact in $\mathcal{C}$. Using the uniform estimate \eqref{eqn:uniform.estimate}, for any $K \geq 1$ we have
	\[
	\nu_K(\mathcal{K}) \geq \mathbb{P}(\Psi_K(\overline{e}) \leq M) \geq 1 - \epsilon.
	\]
	This establishes the tightness of $\{\nu_K\}_{K \geq 1}$.
\end{proof}

Now we are ready to prove Theorem \ref{thm:weak.convergence}.

\begin{proof}[Proof of Theorem \ref{thm:weak.convergence}]
	By Lemma \ref{lem:tightness} and Prokhorov's theorem, the sequence $\{\nu_K\}$ is relatively compact in the topology of weak convergence. This means that for any subsequence $\{\nu_{K'}\}$ of $\{\nu_K\}$, there exists a further subsequence $\{\nu_{K''}\}$ that converges weakly to some probability measure, say $\nu^*$, on $\mathcal{C}$. However, by Proposition \ref{prop:finite.dimensional} the finite dimensional distributions of $\nu^*$ are given by the right hand side of \eqref{eqn:finite.dimensional} which does not depend on the subsequence chosen. Thus by Lemma \ref{lem:unique} there is a unique weak limit point. Consequently, the original sequence $\{\nu_K\}$ converges weakly to $\mu = \nu^*$ whose finite dimensional marginals are given by the right hand side of \eqref{eqn:finite.dimensional}. 
	
	It remains to verify that $\mu$ is supported on $\mathcal{C}_+$, the subset of functions in $\mathcal{C}$ that are strictly positive on $\Delta_n$. By Lemma \ref{lem:concave.bound}, if $\psi \in \mathcal{C}$ is positive at some $p \in \Delta_n$, then $\psi(q) > 0$ for all $q \in \Delta_n$. From \eqref{eqn:finite.dimensional} it is clear that if $\Psi \sim \mu$ then $\Psi(p) > 0$ with probability $1$ for any $p \in \Delta_n$. This implies that $\mu(\mathcal{C}_+) = 1$ and the theorem is proved.
\end{proof}

\subsection{Tail probability}
Consider the limiting distribution $\mu$ given in Theorem \ref{thm:weak.convergence}. Proposition \ref{prop:finite.dimensional} characterizes the finite dimensional distributions of $\mu$. Now we extend this result to the tail probability defined as follows.

\begin{definition} \label{def:tail.probability}
	Given a Borel probability measure $\nu$ on $\mathcal{C}$, we define its tail probability as the functional $\mathsf{T}_{\nu}: \mathcal{C} \rightarrow [0, 1]$ defined by
	\begin{equation} \label{eqn:tail.probability}
	\mathsf{T}_{\nu}(\psi) := \nu\{ \omega \in \mathcal{C} : \omega \geq \psi\}, \quad \psi \in \mathcal{C}.
	\end{equation}
\end{definition}

We show that the tail probability $\mathsf{T}_{\nu}$ characterizes $\nu$. Recall from \eqref{eqn:psi.p.a} that for a collection of points ${\bf p} = (p^{(1)}, \ldots, p^{(r)})$ in $\Delta_n$ and ${\bf a} = (a_1, \ldots, a_r)$ in $(0, \infty)$, $\psi_{{\bf p}, {\bf a}}$ is the smallest concave function such that $\psi(p^{(i)}) \geq a_i$ for $i = 1, \ldots, r$.

\begin{lemma} \label{lem:tail.finite.dim}
	Let $\nu$ be a probability measure on $\mathcal{C}$. Given ${\bf p}$ and ${\bf a}$ as above, we have
	\[
	\nu\{\omega: \omega(p^{(i)}) \geq a_i\} = \nu\{\omega: \omega \geq \psi_{{\bf p}, {\bf a}}\} = \mathsf{T}_{\nu}(\psi_{{\bf p}, {\bf a}}).
	\]
	Consequently the tail probability $\mathsf{T}_{\nu}$ fully characterizes the measure $\nu$.
\end{lemma}
\begin{proof}
	By definition, if $\omega \geq \psi_{{\bf p}, {\bf a}}$ then clearly $\omega(p^{(i)}) \geq a_i$ for all $i$, and so
	\[
	\nu\{\omega: \omega(p^{(i)}) \geq a_i\} \geq \mathsf{T}_{\nu}(\psi_{{\bf p}, {\bf a}}).
	\]
	On the other hand, if $\omega \in \mathcal{C}$ is such that $\omega(p^{(i)}) \geq a_i$ for all $i$, then by concavity of $\omega$ we have $\omega \geq \psi_{{\bf p}, {\bf a}}$. This gives the reverse inequality. The last assertion follows from Lemma \ref{lem:unique}.
\end{proof}

Now we characterize the tail probability functional of the limiting measure $\mu$. Recall that for $\psi \in \mathcal{C}$ we define $\widehat{R}(\psi) = \bigcup_{p \in \Delta_n} R(p, \psi(p)) \subset \mathbb{R}_+^n$.

\begin{theorem} \label{thm:tail.probability}
	Consider the limiting measure $\mu$ as in Theorem \ref{thm:weak.convergence}. We have
	\begin{equation} \label{eqn:tail.limit}
	\mathsf{T}_{\mu}(\psi) = \exp \left( - \int_{\widehat{R}(\psi)} h(x) dx \right), \quad \text{for } \psi \in \mathcal{C}.
	\end{equation}
\end{theorem}
\begin{proof}
	By Proposition \ref{prop:finite.dimensional} and Lemma \ref{lem:tail.finite.dim}, we have that \eqref{eqn:tail.limit} holds whenever $\psi = \psi_{{\bf p}, {\bf a}}$ for some ${\bf p}$ and ${\bf a}$. We now extend this identity to an arbitrary $\psi \in \mathcal{C}$.
	
	Let $D = \{p^{(1)}, p^{(2)}, \ldots\}$ be a countable dense set of $\Delta_n$. By continuity of $\omega \in \mathcal{C}$, we have
	\[
	\{\omega : \omega \geq \psi\} = \bigcap_{r = 1}^{\infty} \{\omega: \omega(p^{(i)}) \geq \psi(p^{(i)}) \text{ for } 1 \leq i \leq r\}.
	\]
	Since the events on the right hand side are decreasing in $r$, we have
	\begin{eqnarray}
	\mathbb{P}\{\omega: \omega \geq \psi\} &=& \lim_{r \rightarrow \infty} \mathbb{P}\{\omega : \omega(p^{(i)}) \geq \psi(p^{(i)}) \text{ for } 1 \leq i \leq r\} \\
	&=& \lim_{r \rightarrow \infty} \exp\left(- \int_{\widehat{R}_{r}} h(x) dx \right),
	\end{eqnarray}
	where $\widehat{R}_{r}$ is the $\widehat{R}({\bf p}, {\bf a})$ generated by $p^{(1)}, \ldots, p^{(r)}$ and the values of $\psi$ at these points.
	
	It remains to show that $\lim_{r \rightarrow \infty} \int_{\widehat{R}_{r}} h(x) dx = \int_{\widehat{R}(\psi)} h(x) dx$. Indeed, we have
	\begin{equation} \label{eqn:set.limit}
	\widehat{R}(\psi) = \bigcup_{i = 1}^{\infty} R(p^{(i)}, \psi(p^{(i)})).
	\end{equation}
	To see this, suppose by way of contradiction that there exists an element $x$ of $\widehat{R}(\psi) \setminus \bigcup_{i = 1}^{\infty} R(p^{(i)}, \psi(p^{(i)}))$. Then there exists $p \in \Delta_n$ such that $\langle p, x \rangle < \psi(p)$ but $\langle p', x\rangle \geq \psi(p')$ for all $p' \in D$. This is clearly a contradiction since we can approach $p$ by a sequence of points $p'$ in $D$. Now \eqref{eqn:set.limit}  and the monotone convergence theorem imply that $\lim_{r \rightarrow \infty} \int_{\widehat{R}_{r}} h(x) dx = \int_{\widehat{R}(\psi)} h(x) dx$.
\end{proof}



\subsection{Realization of the weak limit by Poisson point processes} \label{sec:poisson}
In Theorem \ref{thm:tail.probability} we may interpret the Borel set function
\begin{equation} \label{eqn:intensity.measure}
A \mapsto \int_{A} h(x) dx
\end{equation}
as a Radon measure $m$ on $\mathbb{R}^n$, so that the integral in \eqref{eqn:tail.limit} is equal to $m(\widehat{R}(\psi))$. More generally, given an arbitrary Radon measure $m$ on $\mathbb{R}^n_+$, can we construct a probability measure $\nu$ on $\mathcal{C}$ whose tail probability is given by
\begin{equation} \label{eqn:new.tail.probability}
\mathsf{T}_{\nu}(\psi) = \exp \left( - m(\widehat{R}(\psi)\right)?
\end{equation}
We show that this can be achieved by a Poisson point process. This gives a direct probabilistic construction of the limiting distribution $\mu$ in Theorem \ref{thm:weak.convergence} without going through the limiting process, and suggests an algorithm for simulating samples from $\mu$. We leave practical implementation (possibly tailored to the financial applications) as a problem for future research.

We begin by recalling the defining property of a Poisson point process (for details see \cite{K92, R87}). Let $m$ be a Radon measure on $\mathbb{R}^n_+$. A Poisson point process with intensity measure $m$ is a random closed set $N$ such that for any bounded Borel set $A \subset \mathbb{R}^n_+$, the random variable $|N \cap A|$ (here $|\cdot|$ denotes the cardinality of a set) follows the Poisson distribution with rate $m(A)$, and if $A_1, \ldots, A_m$ are disjoint, then $|N \cap A_1|, \ldots, |N \cap A_m|$ are independent.

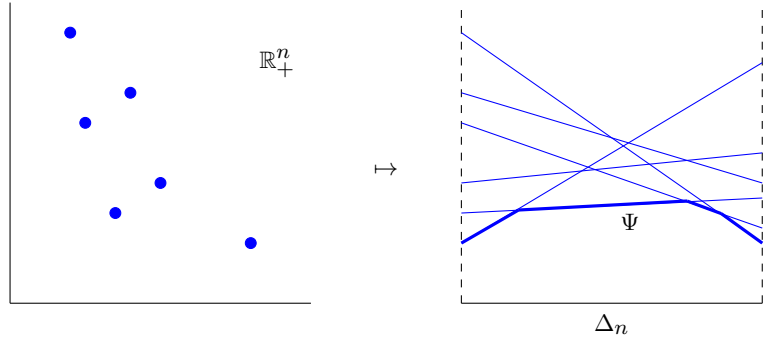
\begin{figure}[t]
	\begin{tikzpicture}[scale = 0.4]
	
	\draw (0, 0) to (0, 10);
	\draw (0, 0) to (10, 0);
	
	
	\draw[blue, fill = blue, radius = 0.18] (3.5, 3) circle;
	\draw[blue, fill = blue, radius = 0.18] (4, 7) circle;
	\draw[blue, fill = blue, radius = 0.18] (2.5, 6) circle;
	\draw[blue, fill = blue, radius = 0.18] (8, 2) circle;
	\draw[blue, fill = blue, radius = 0.18] (2, 9) circle;
	\draw[blue, fill = blue, radius = 0.18] (5, 4) circle;

	\node[black, right] at (8,8) {{\footnotesize $\mathbb{R}^n_{+}$}};
	
	\node[black, above] at (12.5, 4) {{\footnotesize $\mapsto$}};
	
	\draw (15, 0) to (25, 0);
	\draw[dashed] (15, 0) to (15, 10);
	\draw[dashed] (25, 0) to (25, 10);
	
	\draw[blue] (15, 3) -- (25, 3.5);
	\draw[blue] (15, 7) -- (25, 4);
	\draw[blue] (15, 6) -- (25, 2.5);
	\draw[blue] (15, 2) -- (25, 8);
	\draw[blue] (15, 9) -- (25, 2);
	\draw[blue] (15, 4) -- (25, 5);
	
	\draw[blue, very thick] (15, 2) -- (16.9, 3.1);
	\draw[blue, very thick] (16.9, 3.1) -- (22.5, 3.4);
	\draw[blue, very thick] (22.5, 3.4) -- (23.64, 2.98);
	\draw[blue, very thick] (23.64, 2.98) -- (25, 2);
	
	\node[black, right] at (20,2.7) {{\footnotesize $\Psi$}};
	
	\node[black, below] at (20,-0.1) {{\footnotesize $\Delta_n$}};
	\end{tikzpicture} \caption{A realization of the point process $N$ defines a non-negative concave function via the correspondence between $\mathbb{R}_{+}^n$ and $\mathcal{A}_+$.} \label{fig:point.process}
\end{figure}

Recall from Section \ref{sec:notations} that any point of $\mathbb{R}_{+}^n$ can be identified with a positive affine function on $\Delta_n$. Given the random set $N$, we can define a random concave function $\Psi$ by
\begin{equation} \label{eqn:point.to.psi}
\Psi(p) = \inf_{x \in N} \langle p, x \rangle, \quad p \in \Delta_n.
\end{equation}
See Figure \ref{fig:point.process} for an illustration of this construction. It is easy to see that, as long as $N$ is locally finite, the random concave function $\Psi$ is locally piecewise affine on $\Delta_n$. Note that this property may not hold in the diagonal softmin case studied in Section \ref{sec:diagonal}; see in particular Figure \ref{fig:diagonal} where the realizations are almost surely differentiable on $\Delta_n$.

\begin{remark}
	It is clear that \eqref{eqn:point.to.psi} makes sense for any point process (not necessarily Poisson). However, at this generality it is difficult to obtain concrete results about the resulting distribution apart from the tail probability.
\end{remark}

\begin{theorem} \label{thm:Poisson}
	Let $\nu$ be the distribution of the random concave function $\Psi$ defined in \eqref{eqn:point.to.psi}, where $N$ is a Poisson point process on $\mathbb{R}^n_+$ with intensity measure $m$. Then its tail probability functional is given by \eqref{eqn:new.tail.probability}. In particular, the limiting distribution $\mu$ in Theorem \ref{thm:weak.convergence} is the special case where the intensity measure $m$ has density $h$ with respect to Lebesgue measure.
\end{theorem}
\begin{proof}
	Let $\psi \in \mathcal{C}$ be given and consider the event $\{\Psi \geq \psi\}$. Note that its complement $\{\Psi \geq \psi\}^c$ occurs if and only if there exists $p \in \Delta_n$ such that
	\[
	\Psi(p) = \inf_{x \in N} \langle p, x \rangle < \psi(p),
	\]
	i.e., $N \cap \widehat{R}(\psi) \neq \emptyset$. It follows that
	\[
	\mathbb{P}\left(\Psi \geq \psi\right) = \mathbb{P}\left(N \cap \widehat{R}(\psi) = \emptyset\right) = e^{-m(\widehat{R}(\psi))}
	\]
	and the theorem is proved.
\end{proof}

\begin{remark} \label{rmk:EVT2}
Poisson point processes arise naturally in characterizations of extreme value distributions. In \cite[Theorem 2.4]{GHV90}, a representation of max-infinitely divisible and sample continuous processes (which arise as limits of pointwise maxima of i.i.d.~processes) on a compact metric space is given in terms of Poisson point processes. In this general result, the point process lives on the function space itself, and the Poisson representation is not easy to apply directly. Using the duality in Section \ref{sec:duality}, we are able to obtain a clean and explicit probabilistic representation of the limiting distribution $\mu$. 
\end{remark}

\section{Further properties of the limiting distribution}  \label{sec:further.properties}
Consider the limiting distribution $\mu$ in Theorem \ref{thm:weak.convergence}. Equivalently, we may consider the distribution constructed in Theorem \ref{thm:Poisson} where the intensity measure has the form \eqref{eqn:intensity.measure} and $h$ satisfies the conditions of Lemma \ref{lem:assumption.consequences}. In this section we develop further properties of $\mu$, sometimes under additional conditions. Throughout this section we let $\Psi$ be a random element in $\mathcal{C}$ with distribution $\mu$.

\medskip

\subsection{Exponential distribution and the geometric mean}
In Proposition \ref{prop:finite.dimensional} we have shown that $\Phi^{n + \alpha}(p)$ is exponentially distributed for $p \in \Delta_n$ fixed.  A direct calculation using the exponential distribution gives a formula for the expected value of $\Psi(\cdot)$.

%

\begin{corollary} \label{cor:one.point}  For any $p \in  \Delta_n$, we have
	\begin{equation} \label{eqn:expectation}
	\mathbb{E}[\Psi(p)] = \Gamma\left(\frac{1}{n+\alpha} +1\right)\left(\int_{R(p,1)}h(y)dy\right)^{-\frac{1}{n+\alpha}},
	\end{equation}
	where $\Gamma(\cdot)$ is the gamma function.
\end{corollary}

\begin{remark}
	As a by-product we have the following result: for any $h$ on $\mathbb{R}^n_+$ which is homogeneous of order $\alpha$ and locally integrable, the function defined by the right hand side of \eqref{eqn:expectation} is concave. In Proposition \ref{prop:expectation.portfolio} below we compute the derivative of $\mathbb{E}[\Psi(p)]$ and the corresponding portfolio map.
\end{remark}

Under the following condition we can derive an explicit formula for the expected value of $\Psi(\cdot)$.

\begin{assumption} \label{ass:density.B}
	In Assumption \ref{ass:density} we assume that $h$ has the form
	\begin{equation} \label{eqn:h.assumption.B}
	h(x) = \gamma \prod_{i = 1}^n x_i^{\alpha_i}
	\end{equation}
	for some $\gamma > 0$ and exponents $\alpha_i > -1$ (see Example \ref{ex:density}).
\end{assumption}

Note that the homogeneity of $h$ implies that $\alpha = \sum_{i = 1}^n \alpha_i$, and we have $n+\alpha > 0$. 

\begin{proposition} \label{prop:one.point.expectation}
	Under Assumption \ref{ass:density.B}, let $\pi_i = \frac{1+ \alpha_i}{n + \alpha}$ so that $\pi = (\pi_1, \ldots, \pi_n) \in \Delta_n$. Also let
	\begin{equation} \label{eqn:const.M}
	M = \gamma \int_{R(\overline{e}, 1/n)} \prod_{i = 1}^n y_i^{\alpha_i} dy = \int_{R(\overline{e}, 1/n)} h(y) dy > 0.
	\end{equation}
	Then for $p \in \overline{\Delta}_n$ we have
	\begin{equation} \label{eqn:exp.expectation}
	\mathbb{E}\left[ \Psi(p)\right] = L p_1^{\pi_1} \cdots p_n^{\pi_n},
	\end{equation}
	where $L =\left.\Gamma\left(\frac{1}{n + \alpha}+1\right)\right/M^{1/(n + \alpha)}$.
\end{proposition}
\begin{proof}
	For $p \in \Delta_n$, letting $u_i = p_iy_i$, we have
	\begin{equation} \label{eqn:h.integral.Rp}
	\int_{R(p,1)}h(y)dy = \frac{\gamma}{p_1^{1+\alpha_1} \cdots p_n^{1+\alpha_n}}\int_{R(\overline{e},1/n)} \prod_{i = 1}^n u_i^{\alpha_i}du = \frac{M}{p_1^{1+\alpha_1} \cdots p_n^{1+\alpha_n}},
	\end{equation}
	and \eqref{eqn:exp.expectation} follows by Corollary \ref{cor:one.point}.  By the dominated convergence theorem we may extend \eqref{eqn:exp.expectation} to the boundary.
\end{proof}

\begin{figure}[t]
	\centering
	\includegraphics[scale=0.45]{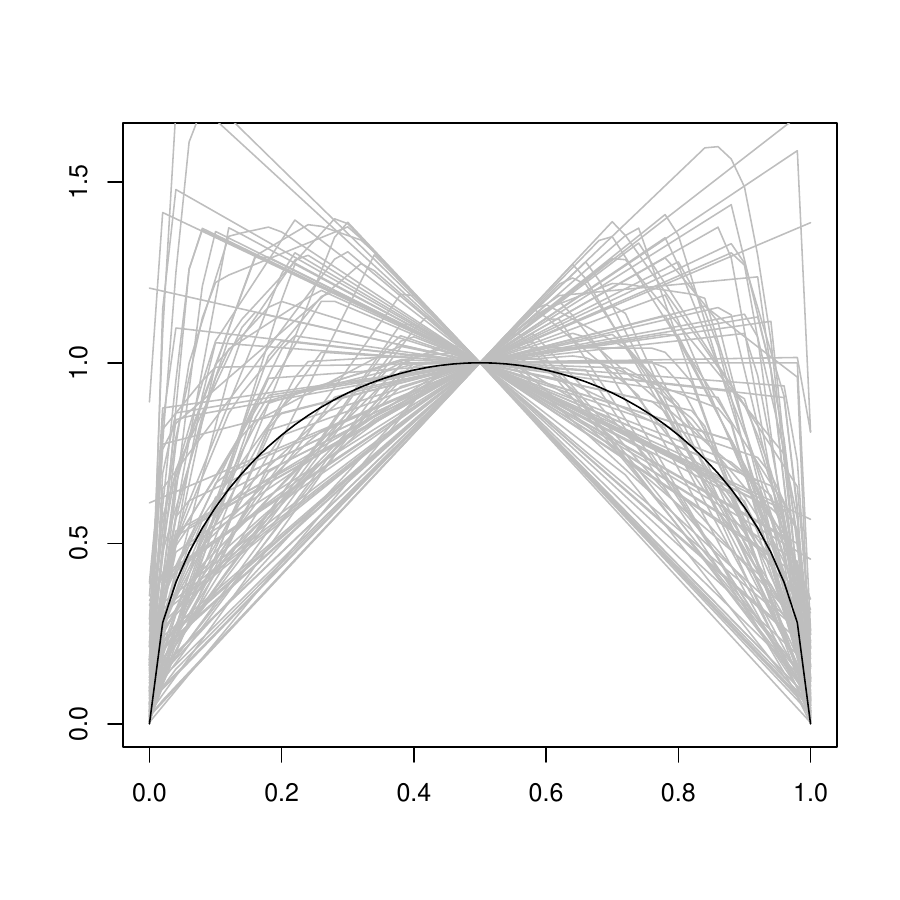}
	\caption{Samples of $\Psi_K(\cdot) / \Psi_K(\overline{e})$ where $n = 2$, $K = 300$ and the components of $C$ are independent $U(0, 1)$ random variables (thus $\alpha_1 = \alpha_2 = 0$). The $x$-axis is $p_1 = 1-p_2$. The solid black curve is the (limiting) expected value given by the geometric mean $2 \sqrt{p_1 p_2}$.} \label{fig:sample}
\end{figure}

From Proposition \ref{prop:one.point.expectation}, the geometric mean emerges as a limiting average shape of the random concave functions $\Psi_K = K^{\frac{1}{n + \alpha}} \min\{\ell_1, \ldots, \ell_K\}$. See Figure \ref{fig:sample} for an illustration where we plot instead the normalized value $\Psi_K(\cdot) / \Psi_K(\overline{e})$ (so that the value is $1$ at the barycenter). This result is interesting because the geometric mean \eqref{eqn:exp.expectation} plays a fundamental role in stochastic portfolio theory and the Dirichlet transport. Indeed, as a portfolio generating function it generates, in the sense of \eqref{eqn:fgp}, the constant-weighted portfolio $\boldsymbol{\pi}(p) \equiv \pi$. In Section \ref{sec:portfolio.weight} we give further results about the random portfolio map generated by $\Psi$. 

\subsection{Boundary behavior}
Next we study the behavior of the random function $\Psi$ near the boundary $\partial \Delta_n$ of the simplex. In this subsection we work under Assumption \ref{ass:density.B}. First we show that $\Psi$ vanishes on the boundary almost surely.

\begin{proposition} \label{prop:boundary}
	Under Assumption \ref{ass:density.B} we have $\mathbb{P}( \Psi|_{\partial \Delta_n} \equiv 0 ) = 1$.
\end{proposition}
\begin{proof}
	Let $p \in \partial \Delta_n$ and let $\{p^{(r)}\}_{r = 1}^{\infty}$ be a sequence in $\Delta_n$ converging to $p$. By continuity of $\Psi$, we have the almost sure limit
	\begin{equation} \label{eqn:pointwise.limit}
	\Psi(p) = \lim_{r \rightarrow \infty} \Psi(p^{(r)}).
	\end{equation}
	Using the convergence \eqref{eqn:pointwise.limit} (which implies weak convergence), the Portmanteau lemma and Proposition \ref{prop:finite.dimensional}, we have
	\[
	\mathbb{P} (\Psi(p) > x) \leq \limsup_{r \rightarrow \infty} \exp \left( -x^{n + \alpha} \int_{R(p^{(r)}, 1)} h(y) dy \right)
	\]
	for all $x > 0$. By \eqref{eqn:h.integral.Rp}, we have $\int_{R(p^{(r)}, 1)} h(y) dy \rightarrow \infty$ as $r \rightarrow \infty$. Thus $\Psi(p) = 0$ almost surely for any fixed $p \in \partial \Delta_n$. The previous argument then implies that $\mathbb{P} (\Psi|_D \equiv 0 ) = 1$, where $D$ is a countable dense set of $\partial \Delta_n$. By the continuity of $\Psi$ we have $\mathbb{P} (\Psi|_{\partial \Delta_n} \equiv 0) = 1$.
\end{proof}

Indeed, under the same assumptions we can derive a H\"{o}lder estimate for $\Psi$ near the boundary.

\begin{theorem} \label{thm:boundary.behavior}
	Suppose Assumption \ref{ass:density.B} holds. Let $\pi_i = \frac{1+ \alpha_i}{n + \alpha}$, $\kappa > 0$ and $\beta > 0$. Fix $j \in \{1, 2, \ldots, n\}$. Then, almost surely, there exists $\delta = \delta(\omega) > 0$ depending on $\omega \in \Omega$ such that
	\begin{equation} \label{eqn:boundary.Holder}
	\Psi(p) \leq \kappa p_j ^{(1 - \beta)\pi_j}, \quad \text{for } p_j  \leq \delta.
	\end{equation}
\end{theorem}
\begin{proof}
We consider the behavior of $\Psi$ near the face $p_j = 0$. Recall the notation $c^{(j, \epsilon)}$ defined in \eqref{eqn:c.j.epsilon}. Define the events
	\[
	B(a, \epsilon) = \{\Psi(c^{(j, \epsilon)}) > a \}.
	\]
	By Proposition \ref{prop:finite.dimensional} and \eqref{eqn:h.integral.Rp}, we have
	\[
	\mathbb{P}(B(a, \epsilon))  =  \exp \left\{  -\frac{M a^{n+\alpha}}{\epsilon^{1 + \alpha_j}} \left( \frac{n-1}{1 - \epsilon} \right)^{n-1+\alpha-\alpha_j} \right\},
	\]
	where $M$ is the constant defined by \eqref{eqn:const.M}.  Now let $B_k = B(a_k, \epsilon_k)$ where $\epsilon_k = \frac{1}{k}$ and
	\[
	a_k = \frac{\kappa}{n} \left( \frac{1}{k + 1} \right)^{ \frac{1+\alpha_j}{n+\alpha}(1 - \beta)} = \frac{\kappa}{n} \left( \frac{1}{k + 1} \right)^{(1 - \beta)\pi_j}.
	\]
	With this choice of $\epsilon_k$ and $a_k$ we have $\sum_{k = 1}^{\infty} \mathbb{P}(B_k) < \infty$. By the Borel-Cantelli lemma, we have $\mathbb{P}(B_k \; \text{i.o.}) = 0$, so that almost surely, there exists $k_0 = k_0(\omega)$ such that
	\[
	\Psi(c^{(j, 1/k)}) \leq \frac{\kappa}{n} \left( \frac{1}{k + 1} \right)^{(1 - \beta)\pi_j}, \text{ for } k \geq k_0.
	\]
	
	Now define $\delta = \min\{1/k_0, 1/n\}$. If $p \in \Delta_n$ satisfies $p_j \leq \delta$, let $k \geq 1/\delta \geq k_0$ such that $\frac{1}{k + 1} \leq p_j \leq \frac{1}{k}$. Then, from \eqref{eqn:concave.bound2} in Lemma \ref{lem:concave.bound}, we have
	\[
	\Psi(p) \leq n \Psi(c^{(j, 1/k)}) \leq  \kappa p_j ^{(1 - \beta)\pi_j}.
	\]
\end{proof}

\subsection{Explicit formula of the tail probability}
In this subsection we specialize to the case where $h(x) \equiv \gamma > 0$ on $\mathbb{R}_+^n$, i.e., the intensity measure is a constant multiple of the Lebesgue measure. By Theorem \ref{thm:tail.probability}, the tail probability of $\mu$ is given by
\[
\mathsf{T}_{\mu}(\psi) = \exp \left( - \gamma \mathrm{vol}(\widehat{R}(\psi))\right), \quad \psi \in \mathcal{C},
\]
where the volume $\mathrm{vol}(\widehat{R}(\psi))$ means the Lebesgue measure of $\widehat{R}(\psi)$. To understand it further it is desirable to have an explicit formula of the volume of the region $\widehat{R}(\psi)$. This will be computed under the assumption that $\psi$ is $C^2$ (twice continuously differentiable). The computation, which is differential geometric, reveals an interesting geometric structure of the duality $\psi \leftrightarrow \widehat{R}(\psi)$ which may be of independent interest. We believe that the resulting formula can be extended to the general nonsmooth case by using the Monge-Amp\`{e}re measure (see \cite[Section 2.1]{DF14}), but this will not be pursued further in this paper.

First we introduce some notations that will be used in this subsection. We let $(q_1, \ldots, q_{n - 1})$ be the coordinate system on $\Delta_n$ obtained by dropping the last component of $p = (p_1, \ldots, p_n)$. The domain of $q$ is
\[
D_{n-1} := \left\{(q_1, \ldots, q_{n-1}) \in (0, 1): \sum_{i = 1}^{n-1} q_i < 1\right\}.
\]
Given $\psi \in \mathcal{C}$, by an abuse of notation we also regard it as a function of $q$ on $D_{n-1}$:
\[
\psi(q_1, \ldots, q_{n-1}) = \psi \left(q_1, \ldots, q_{n-1}, 1 - \sum_{i = 1}^{n-1}q_i \right).
\]
Also, by $D^2 \psi(q)$ we mean the Hessian matrix of $\psi$ as a function of the $(n - 1)$-dimensional variable $q$.

To illustrate the technique we first assume that $\widehat{R}(\psi)$ is bounded. Note that by Proposition \ref{prop:boundary}, if $\Psi \sim \mu$ then $\Psi$ vanishes on the boundary. So in the tail probability it suffices to consider only functions that vanish on $\partial \Delta_n$.

\begin{theorem} \label{thm:stoke}
	Suppose $\psi \in \mathcal{C}$ is $C^2$ on $\Delta_n$, $\psi|_{\partial \Delta_n} \equiv 0$ and $\widehat{R}(\psi)$ is bounded. Then
	\begin{equation} \label{eqn:volume}
	\mathrm{vol}(\widehat{R}(\psi)) = \frac{1}{n} \int_{D_{n-1}} \psi(q) \det (-D^2 \psi(q)) dq.
	\end{equation}
	In particular, when $h(x) \equiv \gamma > 0$ we have
	\begin{equation} \label{eqn:tail.C2}
	\mathsf{T}_{\mu}(\psi) = \exp \left( \frac{-\gamma}{n}  \int_{D_{n-1}} \psi(q) \det (-D^2 \psi(q)) dq \right).
	\end{equation}
\end{theorem}


\begin{proof}
	Under the stated assumptions, the closure of $\widehat{R}(\psi)$ is an $n$-dimensional orientable compact manifold, denoted by $M$ and embedded in $\mathbb{R}^n$, with piecewise $C^1$ boundary. The Euclidean coordinates $(x_1, \ldots, x_n)$ of $\mathbb{R}^n$ is a global coordinate system of $M$.
	
	We will apply Stokes' theorem (see for example \cite{Lee12})
	\begin{equation} \label{eqn:stokes}
	\int_{\partial M} \omega = \int_{M} d\omega
	\end{equation}
	with the differential $(n - 1)$-form given by
	\begin{equation} \label{eqn:omega}
	\omega = \sum_{k = 1}^n (-1)^{k+1} x_k dx_1 \wedge \cdots \wedge \widehat{dx_k} \wedge \cdots \wedge dx_n.
	\end{equation}
	Here the notation $\widehat{dx_k}$ means that the term $dx_k$ is omitted in the wedge product. It is immediate to check that the exterior derivative of $\omega$ is given by
	\[
	d\omega = n dx_1 \wedge \cdots \wedge dx_n,
	\]
	so that the right hand side of \eqref{eqn:stokes} is
	\[
	\int_{M} d\omega = n \mathrm{vol}(\widehat{R}(\psi)).
	\]
	
	It remains to compute $\int_{\partial M} \omega$. The boundary $\partial M$ of $M$ consists of parts of the coordinate hyperplanes as well as the (curved) part $S \subset \partial M$ parameterized by
	\begin{equation} \label{eqn:curved.part}
	X_i(p) = \psi(p) + D\psi(p) \cdot (e_i - p), \quad 1 \leq i \leq n, \ p \in \Delta_n.
	\end{equation}
	
	First we reparameterize $N$ in terms of $q \in D_{n-1}$. Denote the partial derivatives of $\psi$ (as a function of $q$) by
	\begin{equation} \label{eqn:psi.derivative}
	\psi_i(q) = \frac{\partial \psi}{\partial q_i}(q) = D\psi(p) \cdot (e_i - e_n).
	\end{equation}
	Since
	\[
	e_i - p = e_i - e_n - \sum_{r = 1}^{n - 1} p_r (e_r - e_n),
	\]
	we have
	\[
	D\psi(p) \cdot (e_i - p) =
	\left\{\begin{array}{ll}
	\psi_i(q) - \sum_{r = 1}^{n-1} q_r \psi_r(q), & \text{if } 1 \leq i \leq n - 1 \\
	-\sum_{r = 1}^{n-1} q_r \psi_r(q), & \text{if } i = n.\\
	\end{array}\right.
	\]
	Writing $X$ as a function of $q$, we can rewrite \eqref{eqn:curved.part} as
	\[
	X_i(q) =
	\left\{\begin{array}{ll}
	\psi(q) + \psi_i(q) - \sum_{r = 1}^{n-1} q_r \psi_r(q), & \text{if } 1 \leq i \leq n - 1 \\
	\psi(q) -\sum_{r = 1}^{n-1} q_r \psi_r(q), & \text{if } i = n.\\
	\end{array}\right.
	\]
	On each of the coordinate planes $x_i = 0$ we have, from \eqref{eqn:omega},
	\[
	\int_{\partial M \cap \{x_i = 0\}} \omega = 0.
	\]
	On the curved part $S$ where $x = X(q)$, consider the pullback
	\[
	dx_i = \sum_{j = 1}^{n-1} \frac{\partial X_i}{\partial q_j} dq_j.
	\]
	Plugging this into \eqref{eqn:omega} and computing the pullback of $\omega$, we see that
	\[
	\int_{N} \omega = \int_{D_{n-1}} | \mathrm{det}(A(q))| dq_1  \cdots  dq_{n-1},
	\]
	where $A(q)$ is the $n \times n$ matrix with entries $\frac{\partial}{\partial q_j} X_i(q)$ for $1 \leq i \leq n$, $1 \leq j \leq n - 1$, and $n$th column $X(q)$.   Now for $1 \le j \le n-1$ we have
	$$
	\frac{\partial X_i}{\partial q_j}(q) =   \left\{ \begin{array}{ll}
	\psi_{ij}(q) - \sum_{r=1}^{n-1} q_r \psi_{rj}(q) & \mbox{ if } 1 \leq i \le n-1\\[2ex]
	- \sum_{r=1}^{n-1} q_r \psi_{rj}(q) & \mbox{ if } i=n .
	\end{array} \right.
	$$
	The matrix $A(q)$ can be written in block $((n-1)+1) \times ((n-1)+1)$ form as
	$$
	A(q) = \left[ \begin{array}{ccc}
	\psi_{ij}(q) - \sum_{r=1}^{n-1} q_r \psi_{rj}(q) & \hspace{1ex} &  \psi(q)+\psi_i(q)- \sum_{r=1}^{n-1} q_r\psi_r(q) \\[3ex]
	- \sum_{r=1}^{n-1} q_r \psi_{rj}(q) & &\psi(q)- \sum_{r=1}^{n-1} q_r\psi_r(q)
	\end{array} \right].
	$$
	Subtracting the bottom row from all the other rows, we get $\det(A(q)) = \det(B(q))$ where
	$$
	B(q) = \left[ \begin{array}{cc}
	\psi_{ij}(q) &   \psi_i(q) \\[2ex]
	- \sum_{r=1}^{n-1} q_r \psi_{rj}(q) & \psi(q)- \sum_{r=1}^{n-1} q_r\psi_r(q)
	\end{array} \right].
	$$
	In the matrix $B(q)$ the term $\psi(q)$ appears as an additive term in $(n,n)$ entry, and nowhere else, so the calculation of $\det(B(q))$ involves some terms which do not involve $\psi(q)$, and all the other terms contain a simple factor $\psi(q)$.  We get
	\begin{equation*}
	\begin{split}
		\det(B(q)) &= \det\left[\begin{array}{cc}
			\psi_{ij}(q) &   \psi_i(q) \\ [1ex]
			- \sum_{r=1}^{n-1} q_r \psi_{rj}(q) &  - \sum_{r=1}^{n-1} q_r\psi_r(q)
		\end{array}\right]+\psi(q)\det\Big[ \psi_{ij}(q)\Bigr]\\[1ex]
		&=:  \det(C(q)) + \psi(q)\det\Big[ \psi_{ij}(q)\Bigr],
	\end{split}
	\end{equation*}
  say. 	Now we can write
\begin{equation*}
\begin{split}
	C(q) &=  \left[ \begin{array}{cc}
	\psi_{ij}(q) &   \psi_i(q) \\[2ex]
	- \sum_{r=1}^{n-1} q_r \psi_{rj}(q) &  - \sum_{r=1}^{n-1} q_r\psi_r(q)
	\end{array} \right] \\ &=\left[ \begin{array}{c}
	\delta_{ij} \\[2ex]
	-q_j \end{array} \right]\Bigl[ \begin{array}{cc}
	\psi_{ij}(q) &  \psi_i(q)
	\end{array} \Bigr]
\end{split}
\end{equation*}
	as the product of an $n \times (n-1)$ and an $(n-1) \times n$ matrix.  This implies that the rank of $C(q)$ is at most $n-1$ and so $\det(C(q)) = 0$.
	Therefore
	$$
	\det(A(q)) = \det(B(q))  = \psi(q)\det\Big[ \psi_{ij}(q)\Bigr]
	$$
	and the proof is complete.
\end{proof}

Now we relax the boundedness assumption.

\begin{theorem} \label{thm:stokes2}
	The volume formula \eqref{eqn:volume} holds for any $\psi \in \mathcal{C}$ which is $C^2$ and satisfies $\psi|_{\partial{\Delta_n}} \equiv 0$.
\end{theorem}
\noindent
{\it Proof.}
Here we apply Stokes' theorem with the same $(n-1)$-form $\omega$ on the region $\widehat{R}(\psi) \cap (0,K)^n$. The boundary now consists of $S_K = S \cap [0,K]^n$ together with parts of the coordinate planes $x_i = 0$ and parts of the planes $x_i = K$.  For given $K$ define sets
\[
A_1= \{(x_2,x_3, \ldots,x_n) \in(0,K)^{n-1}: (K,x_2,x_3,\ldots,x_n) \in \widehat{R}(\psi) \},
\]
and
\[
B_1 = (0,K) \times A_1 = \{x \in (0,K)^n : (K,x_2,x_3,\ldots,x_n) \in \widehat{R}(\psi) \}.
\]
Note that $\{K\} \times A_1$ is the part of the boundary of $\widehat{R}(\psi) \cap (0,K)^n$ which lies in the plane $x_1= K$. Since $x_1=K$ and $dx_1 = 0$ on the submanifold $\{K\} \times A_1$ we get
$$\omega =  \sum_{k=1}^n (-1)^{k+1}x_k dx_1 \wedge \cdots \wedge \widehat{dx_k} \wedge \cdots \wedge dx_n = K dx_2 \wedge \cdots \wedge dx_n.
$$
Since $dx_2 \wedge \cdots \wedge dx_n$ is the volume form on $\{K\} \times A_1$ given as the boundary of the oriented domain $\widehat{R}(\psi) \cap (0,K)^n$, we obtain
$$
\int_{\{K \} \times A_1} \omega = K \mbox{vol}(A_1) = \mbox{vol}(B_1).
$$
Similar calculations are valid for the other coordinates, involving the sets
$$
B_i = \{x \in (0,K)^n : (x_1,\ldots,x_{i-1},K,x_{i+1},\dots,x_n) \in \widehat{R}(\psi) \}.
$$
Now apply Stokes' theorem:
$$
\int_{\widehat{R}(\psi) \cap (0,K)^n} d\omega = \int_{\partial(\widehat{R}(\psi) \cap (0,K)^n)} \omega.
$$
Using the calculations above, we get
$$
n \mbox{vol}(\widehat{R} \cap (0,K)^n) = \sum_{i=1}^n \mbox{vol}(B_i) + \int_{S_K} \omega.
$$
Since $B_i \subset \widehat{R}(\psi) \cap (0,K)^n$ we can rewrite this as
\[
\sum_{i=1}^n \mbox{vol}\Bigl(\bigl(\widehat{R}(\psi) \cap (0,K)^n\bigr) \setminus B_i\Bigr) = \int_{S_M} \omega.
\]
Now we let $K \to \infty$.  On the right side we have monotone behavior:
\begin{equation*}
\begin{split}
\int_{S_M} \omega & = \frac{1}{n}\int_{X^{-1}(S_K)}\psi(q)\det(-D^2\psi(q))dp_1 \cdots dq_{n-1}\\
&\nearrow  \frac{1}{n}\int_{D_{n-1}}\psi(q)\det(-D^2\psi(q))dp_1 \cdots dq_{n-1}.
\end{split}
\end{equation*}

\begin{lemma}
	For $K > 0$ write
	\[
	B_i^{(K)} =  \{x \in (0,K)^n : (x_1,\ldots,x_{i-1},K,x_{i+1},\dots,x_n) \in \widehat{R}(\psi) \}.
	\]
	Then
	\[
	\Bigl(\widehat{R}(\psi) \cap (0,K)^n \Bigr) \setminus B_i^{(K)} \nearrow \widehat{R}(\psi) \quad \mbox{ as }  K \nearrow    \infty.
	\]
\end{lemma}
\begin{proof}
	It suffices to prove in the case $i = 1$.  Suppose first $K <L$.  If $x \in \bigl(\widehat{R}(\psi) \cap (0,K)^n \bigr) \setminus B_i^{(K)}$ then $x \in \widehat{R}(\psi) \cap (0,K)^n$ and $(K,x_2, \ldots,x_n) \not\in \widehat{R}(\psi)$.  Then $x \in \widehat{R}(\psi) \cap (0,L)^n$ and $(L,x_2, \ldots,x_n) \not\in \widehat{R}(\psi)$, so that $x \in \bigl(\widehat{R}(\psi) \cap (0,L)^n \bigr) \setminus B_i^L$.  This proves monotonicity.
	
	To complete the proof it suffices to show that if $x \in \widehat{R}(\psi)$ then there exists $K$ such that $x \in (0,K)^n$ and $(K,x_2,\ldots,x_n) \not\in \widehat{R}(\psi)$.
	To show $(K_2,\ldots,x_n) \not\in \widehat{R}(\psi)$ we need to show
	\[
	K p_1 +\sum_{i=2}^n x_ip_i  \ge  \psi(p)
	\]
	for all $p \in \Delta_n$.  Let $\epsilon = \min(x_2,x_3,\ldots,x_n) >0$.  Since $\psi$ has zero boundary values there is $\delta \in (0,1/2)$ such that $p_1 <\delta$ implies $\psi(p) < \epsilon /2$.  Finally we may choose $K > \max(\|\psi\|/\delta,x_1,x_2, \ldots,x_n)$.
	
	If $p_1 \geq \delta$ then $K p_1 +\sum_{i=2}^n x_ip_i  \geq  K\delta > \|\psi\| \geq \psi(p)$.  On the other hand if $p_1 <\delta$ then $\psi(p) < \epsilon$ and $p_2+\cdots + p_n = 1-p_1 \geq 1/2$, so that $Kp_1 +\sum_{i=2}^n x_ip_i \geq \epsilon/2 > \psi(p)$.  Together we have $(K,x_2,\ldots,x_n) \not\in \widehat{R}(\psi)$.  The condition $x \in (0,K)^n$ is trivially checked, and the proof is complete.
\end{proof}

\begin{proof}[Proof of the theorem, continued]
	Taking volumes of the increasing sequences of sets in the lemma, we get
	\[
	\mbox{vol}\Bigl(\bigl(\widehat{R}(\psi) \cap (0,K)^n\bigr) \setminus B_i^{(K)}\Bigr) \nearrow \mbox{vol}(\widehat{R}(\psi)
	\]
	as $K \rightarrow \infty$ for each $1 \le i \le n$, and the proof is complete.
\end{proof}


In the case $n = 2$ (so that $D_{n-1}$ is one-dimensional) the formula \eqref{eqn:volume} has an interesting alternative expression. To simplify the notations we write $q = q_1$.

\begin{corollary} \label{cor:volume.1d}
	Suppose $n = 2$ and let $\psi \in \mathcal{C}$ be $C^2$ and (as a function of $q$) $\psi(0) = \psi(1) = 0$. Assume also
	\[
	\lim_{q \rightarrow 0} \psi(q) \psi'(q) = \lim_{q \rightarrow 1} \psi(q) \psi'(q) = 0.
	\]
	Then
	\begin{equation*}
	\mathrm{vol}(\widehat{R}(\psi)) = \frac{1}{2} \int_0^1 (\psi'(q))^2 dq, \quad \mathsf{T}_{\mu}(\psi) = \exp\left( \frac{-\gamma}{2} \int_0^1 (\psi'(q))^2 dq\right).
	\end{equation*}
\end{corollary}
\begin{proof}
	By Theorem \ref{thm:stoke} we have
	\[
	\mathrm{vol}(\widehat{R}(\psi)) = \frac{1}{2} \int_0^1 \psi(q)(-\psi''(q)) dq.
	\]
	Using integration by parts, this is equal to
	\[
	-\left. \frac{1}{2} \psi'(q) \psi(q) \right|_0^1 + \frac{1}{2} \int_0^1 (\psi'(q))^2 dq = \frac{1}{2} \int_0^1 (\psi'(q))^2 dq.
	\]
\end{proof}

\begin{example}\label{ex:geo.mean}
	Let $n = 2$ and consider the geometric mean $\psi(q) = a \sqrt{q(1 - q)}$ where $a > 0$. See Figure \ref{fig:gm} for an illustration.
	
	\begin{figure}[t]
		\centering
		\includegraphics[scale=0.45]{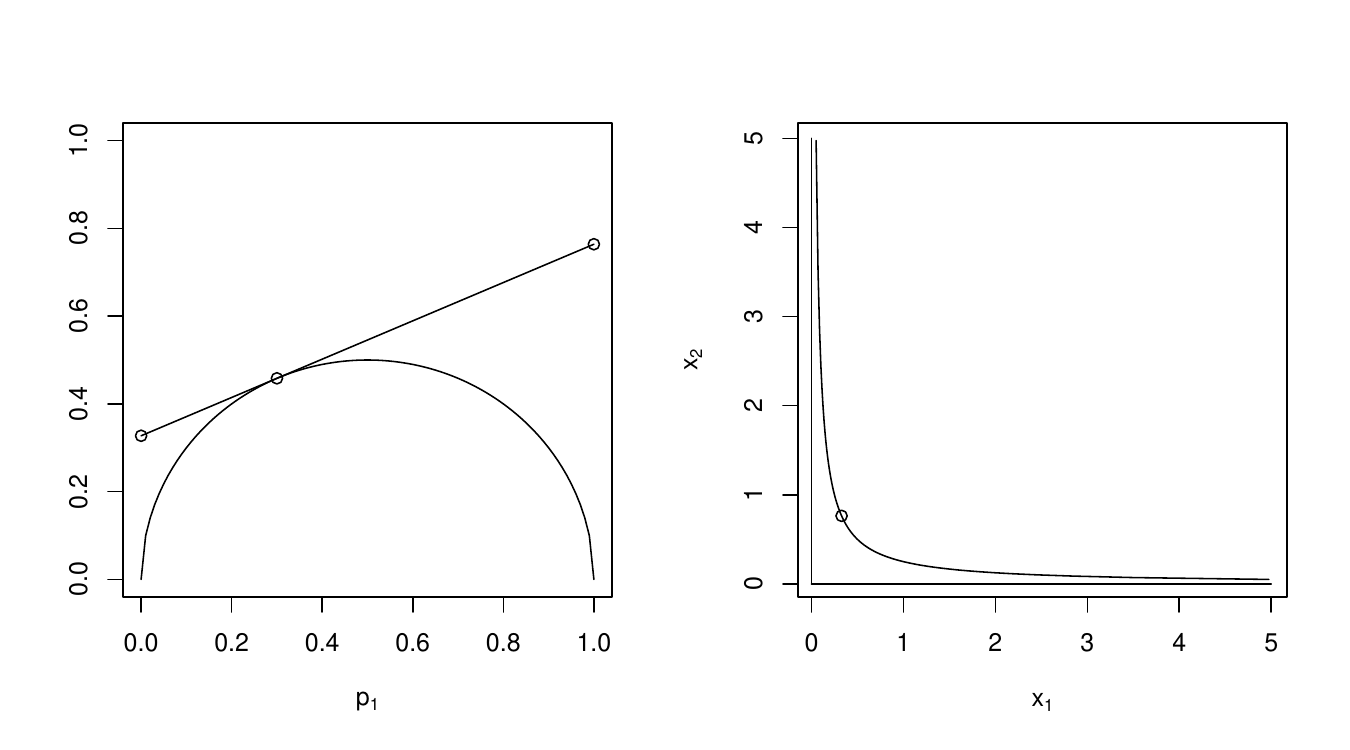}
		\caption{Left: Graph of $\psi(p) = \sqrt{p_1 p_2}$ (here $\gamma = 1$) as well as the tangent line at $p_1 = 0.3$. It defines a point $(x_1, x_2)$ on the boundary of $\widehat{R}(\psi)$. Right: The region $\widehat{R}(\psi)$ and a point on its boundary. The area of $\widehat{R}(\psi)$ is infinite.} \label{fig:gm}
	\end{figure}
	
	Since
	\[
	\psi(q) \psi'(q) = a^2 \sqrt{q(1 - q)} \frac{1 - 2q}{2\sqrt{q(1 - q)}} = a^2 \frac{1 - 2q}{2}
	\]
	does not vanish at $0$ and $1$, we cannot apply Corollary \ref{cor:volume.1d}. Nevertheless, since
	\[
	\int_0^1 \psi(q) (- \psi''(q)) dq = \int_0^1  \frac{a^2}{4x(1 - x)} dx = \infty,
	\]
	by Theorem \ref{thm:stokes2} we have $\mathrm{vol}(\widehat{R}(\psi)) = \infty$. This proves that $\mathsf{T}_{\mu}(\psi) = 0$. This corresponds to the critical exponent in Theorem \ref{thm:boundary.behavior}.
\end{example}

\subsection{Distribution of the portfolio weight} \label{sec:portfolio.weight}
Consider the random function $\Psi$ with distribution $\mu$. It generates, in the sense of \eqref{eqn:fgp}, a portfolio map $\boldsymbol{\pi}$. The portfolio weight $\boldsymbol{\pi}(p)$ is uniquely defined whenever $\Psi$ is differentiable at $p$ (see \cite[Section 2.3]{PW15} for details). As an application of the representation of $\mu$ in terms of the Poisson point process, we derive the distribution of $\boldsymbol{\pi}(p)$ for each fixed $p$.

Let $N$ denote the Poisson point process on $\mathbb{R}_{+}^n$ whose rate measure has density $h$. Thus we may write
\begin{equation} \label{eqn:point.process.representation}
\Psi(p) = \inf_{x \in N} \langle p,x \rangle, \quad p \in \Delta_n.
\end{equation}

Fix $p \in \Delta_n$ and consider the Poisson point process $N(p)$ on $\mathbb{R}_+$ given by $N(p) = \{\langle p,x \rangle: x \in N\}$.  Since the rate measure for $N$ has a density, then so does the rate measure for $N(p)$.  It follows that $N(p)$ has no double points.  In particular with probability 1 there is a unique $Z \in N$ such that $\langle p, Z \rangle $ is the minimum point of $N(p)$, and there exists a (random) $\delta > 0$ such that $\langle p, x \rangle  \geq \langle p, Z \rangle  + \delta$ for all $x \in N \setminus \{Z\}$.

From this observation and the representation \eqref{eqn:point.process.representation}, there exists a (random) neighborhood $U$ of $p$ such that $\Psi(q) = \langle q,Z \rangle$ for $q \in U$. Thus we have the following

\begin{lemma}
	For any fixed $p \in \Delta_n$, the random function $\Psi$ is $\mu$-almost surely differentiable at $p$. Thus the portfolio weight $\boldsymbol{\pi}(p)$ is a.s.-defined.
\end{lemma}

Again let $Z \in N$ be the point described above. Recall from \eqref{eqn:fgp} that the portfolio weight $\boldsymbol{\pi}(p)$ generated by $\Psi$ is given by
\[
\boldsymbol{\pi}_i(p) = p_i (1 + D_{e_i - p} \log \Psi(p)), \quad 1 \leq i \leq n.
\]
Since $\Psi(q) = \langle q, Z \rangle$ for $q$ near $p$, evaluating the derivative gives
\[
\boldsymbol{\pi}(p) = \left(\frac{p_1 Z_1}{\langle p, Z \rangle}, \ldots,  \frac{p_n Z_n}{\langle p, Z \rangle}\right).
\]

\begin{proposition} \label{prop:portfolio.distribution}
	Suppose $h(x)$ satisfies Assumption 3.1.  Under $\mu$, for any $p \in \Delta_n$ fixed, the portfolio weight $\boldsymbol{\pi}(p) = (Y_1,\ldots,Y_n)$ has density $c\, h(y_1/p_1, \ldots, y_n/p_n)$ with respect to the uniform distribution on $\Delta_n$, where $c$ is a normalizing constant. Equivalently, $(Y_1,\ldots,Y_{n-1})$ has density
	$$c \,h\left(\frac{y_1}{p_1}, \ldots, \frac{y_{n-1}}{p_{n-1}},\frac{1-\sum_{i=1}^{n-1}y_i}{p_n}\right)
	$$
	on the set $D_{n-1} = \{y \in \mathbb{R}^{n-1}_+: \sum_{i=1}^{n-1}y_i < 1\}$.
\end{proposition}
\begin{proof}
	For fixed $p \in \Delta_n$ we consider the distribution of $Z$ conditional on the value of $a = \langle p,Z \rangle$.  Constants $c_1, c_2, \ldots$ will denote a sequence of normalizing constants.  Conditioned on the value of $a = \langle p,Z \rangle $, the conditional distribution of $Z = (Z_1,\ldots,Z_n)$ has density $c_1 h(z_1,\ldots,z_n)$ with respect to uniform measure on the simplex $\{z \in \mathbb{R}^n_+ : \sum_{i=1}^n p_i z_i = a\}$.  More precisely, the conditional distribution of $(Z_1, \ldots,Z_{n-1})$ has density
	$$
	c_1h\left(z_1, \ldots,z_{n-1},(a-\sum_{i=1}^{n-1}p_iz_i )/p_n\right)
	$$
	with respect to Lebesgue measure on the set $\{z \in \mathbb{R}^{n-1}_+ : \sum_{i=1}^{n-1} p_i z_i < a \}$.  Now define $Y = (Y_1,\ldots,Y_n) \in \mathbb{R}^n_+$ by $Y_i = p_i Z_i/\langle p,Z\rangle$, so that $\boldsymbol{\pi}(p) = Y$.  Conditioned on $\langle p,Z \rangle = a$, we have $Y_i = p_i Z_i/a$ and so the conditional distribution of $(Y_1, \ldots,Y_{n-1})$ has density
	$$
	c_2 h\left(\frac{ay_1}{p_1}, \ldots, \frac{ay_{n-1}}{p_{n-1}} ,\frac{a-\sum_{i=1}^{n-1}a y_i}{p_n}\right)  =   c_3 h\left(\frac{y_1}{p_1}, \ldots, \frac{y_{n-1}}{p_{n-1}} ,\frac{1-\sum_{i=1}^{n-1} y_i}{p_n}\right)
	$$
	with respect to Lebesgue measure on the set $D_{n-1}$.  On the right side of this expression the points where $h$ is evaluated do not depend on $a$, and the set $D_{n-1}$ does not depend on $a$, so the normalizing constant $c_3$ does not depend on $a$. Since this conditional density does not depend on the value of $a$, it is the unconditional density of $(Y_1, \ldots, Y_{n-1})$, and the proof is complete.
\end{proof}

The next proposition computes the expected value of $\boldsymbol{\pi}(p)$ at a fixed $p$. In words, it states that $\mathbb{E} [\boldsymbol{\pi}(p)]$ is equal to the portfolio $\overline{\boldsymbol{\pi}}(p)$ generated by the expected value $\overline{\Psi} = \mathbb{E} [\Psi(\cdot)]$. It is an interesting problem to study the properties of the random portfolio map $\boldsymbol{\pi}(\cdot): \Delta_n \rightarrow \overline{\Delta}_n$ and their implications in optimal transport and Cover's universal portfolio.

\begin{proposition} \label{prop:expectation.portfolio} Under Assumption 3.1, the expected value of the portfolio weight $\boldsymbol{\pi}(p)$ generated by $\Psi$ is the same as the portfolio weight $\overline{\boldsymbol{\pi}}(p)$ generated by the expectation $\overline{\Psi}(\cdot) = \mathbb{E}\Psi(\cdot)$.
\end{proposition}
\begin{proof}   For $p \in \mathbb{R}^n_+$ define $R(p) = \{x \in \mathbb{R}^n_+: \sum_{i=1}^n p_ix_i < 1\}$ and
	$$
	F(p) = \int_{R(p)} h(x)dx.
	$$
	Recall the set $D_{n-1} = \{y \in R^{n-1}_+: \sum_{i=1}^{n-1} y_i <1\}$ and the notational convenience $y_n = 1-\sum_{i=1}^{n-1} y_i$.  Then (using elementary multivariate analysis and the scaling property of $h$) we have
	$$
	D_{e_i}F(p) = -\frac{1}{p_1\cdots p_n} \int_{D_{n-1}} \frac{y_i}{p_i} h\left(\frac{y_1}{p_1},\cdots ,\frac{y_n}{p_n}\right)dy_1\cdots dy_{n-1}
	$$
	for $1 \le i \le n$, and
	$$
	F(p) = \frac{1}{(n+\alpha)p_1\cdots p_n} \int_{D_{n-1}}  h\left(\frac{y_1}{p_1},\cdots,\frac{y_{n-1}}{p_{n-1}},\frac{y_n}{p_n}\right)dy_1\cdots dy_{n-1}.
	$$
	From Corollary \ref{cor:one.point} we have
	$$
	\overline{\Psi}(p) = \Gamma\left(\frac{1}{n+\alpha} +1\right)(F(p))^{-\frac{1}{n+\alpha}}
	$$
	and then
	\begin{eqnarray*}
		D_{e_i}\log \overline{\Psi}(p) &=& - \frac{D_{e_i}F(p)}{(n+\alpha)F(p)} \\
		&=& \dfrac{ \int_{D_{n-1}} \frac{y_i}{p_i} h\left(\frac{y_1}{p_1},\cdots ,\frac{y_n}{p_n}\right)dy_1\cdots dy_{n-1}}{\int_{D_{n-1}}  h\left(\frac{y_1}{p_1},\cdots ,\frac{y_n}{p_n}\right)dy_1\cdots dy_{n-1}}\\
		&=& \frac{1}{p_i}\mathbb{E}\boldsymbol{\pi}_i(p).
	\end{eqnarray*}
	Therefore, using \eqref{eqn:fgp}, the portfolio weight generated by $\overline{\Psi}$ is   $$
	p_i\left(1+D_{e_i-p} \log \overline{\Psi}(p)\right) = p_i\left(1+\frac{1}{p_i}\mathbb{E}\boldsymbol{\pi}(p)_i-\sum_{j=1}^n
	p_j \frac{1}{p_j}\mathbb{E}\boldsymbol{\pi}(p)_j\right) = \mathbb{E}\boldsymbol{\pi}_i(p),
	$$
	and the proof is complete.
\end{proof}

\begin{example} [Geometric mean and the constant-weighted portfolio]
	Suppose Assumption 4.3 holds. By Proposition \ref{prop:one.point.expectation}, the expected value $\overline{\Psi}(\cdot)$ is a multiple of the geometric mean $p_1^{\pi_1} \cdots p_n^{\pi_n}$ which generates the constant-weighted portfolio $\overline{\boldsymbol{\pi}}(p) \equiv \pi_i$. By Proposition \ref{prop:expectation.portfolio}, if $\boldsymbol{\pi}$ is the (random) portfolio map generated by $\Psi$, then $\mathbb{E} \boldsymbol{\pi}(p) = \overline{\boldsymbol{\pi}}(p) \equiv \pi$.
\end{example}

\section{Diagonal limits} \label{sec:diagonal}
In this final section we study the model \eqref{eqn:phi} where the parameter $\lambda_K$ of the softmin depends on $K$, the number of hyperplanes. We have seen that when $\lambda_K \equiv \lambda < \infty$ is fixed there is a deterministic almost sure limit, and when $\lambda_K \equiv \infty$ a suitable scaling gives a non-trivial weak limit which can be described by a Poisson point process. Here we want to find explicit rates for $\lambda_K$ which give possibly different limiting behaviors.  

\subsection{Main results}
Before stating the main results we first set up some notations. Let $X_1, X_2, \ldots$ be i.i.d.~copies of $C$, where $C$ is the random vector in \eqref{eqn:random.affine}. Thus we may write the $k$-th hyperplane as $\ell_k(p) = \langle p, X_k \rangle$. For $\lambda > 0$ and $K \geq 1$, let $\Phi_K(p) = \min_{1 \leq k \leq K} \langle p, X_k \rangle$ and
\begin{equation} \label{eqn:Phi.lambda.K}
\Phi_{\lambda, K}(p) := \mathsf{m}_{\lambda} \{ \langle p, X_k \rangle : 1 \leq k \leq K\} = \frac{-1}{\lambda} \log \left( \frac{1}{K} \sum_{k = 1}^K e^{-\lambda \langle p, X_k \rangle }\right).
\end{equation}
Throughout Section \ref{sec:diagonal} we work under Assumption \ref{ass:density}. We let $\Psi$ be a random concave function under the limiting distribution $\mu$ in Theorem \ref{thm:weak.convergence}. Note that it is the weak limit of the scaled hardmin $ K^{\frac{1}{n+\alpha}} \Phi_K$.

The following are the main results of this section. First we give
 the case where the weak limit can be related to the hardmin limit $\Psi$.

\begin{theorem}  \label{thm:diagonal12}
	Let Assumption \ref{ass:density} be in force.  Suppose $\frac{\lambda_K}{K^{1/(n+\alpha)}} \to\infty$ as $K \to \infty$.  Then $K^{\frac{1}{n+\alpha}} \Phi_{\lambda_K, K} - \frac{K^{1/(n+\alpha)}\log K}{\lambda_K}$ converges weakly to $\Psi$.  In particular:
	\begin{enumerate}
		\item[(i)] if $\frac{\lambda_K }{ K^{1/(n + \alpha)}  \log K} \rightarrow \infty$ as $K \rightarrow \infty$, then $K^{\frac{1}{n+\alpha}} \Phi_{\lambda_K, K}$ converges weakly to $\Psi$; and
		\item[(ii)] if $\frac{\lambda_K }{ K^{1/(n + \alpha)}  \log K} \rightarrow c$ as $K \rightarrow \infty$, where $c > 0$ is a fixed constant, then $K^{\frac{1}{n+\alpha}} \Phi_{\lambda_K, K}$ converges weakly to $\Psi + c$ which is a deterministic shift of $\Psi$.
	\end{enumerate}
\end{theorem}

Next we consider a case where $\lambda_K$ grows less quickly than in Theorem \ref{thm:diagonal12} above.  Here the effect of the additive normalization is much stronger, and the weak limit may not be supported on $\mathcal{C}$, the space of non-negative (continuous) concave functions on $\overline{\Delta}_n$.  Indeed we will show by an example that the limit may become negative.

\begin{theorem}   \label{thm:diagonal3}
	Let Assumption \ref{ass:density} be in force.  Suppose $\frac{\lambda_K}{K^{1/(n+\alpha)}} \to c$ as $K \to \infty$, where $c > 0$ is a fixed constant.  Then the sequence
	\[
	K^{\frac{1}{n+\alpha}} \Phi_{\lambda_K, K} -  \frac{K^{1/(n+\alpha)}\log K}{\lambda_K}
	\]
	converges weakly to a random concave function $\widetilde{\Psi}_c$ on $\Delta_n$.
\end{theorem}

See Section \ref{sec:diagonal.third.case} for more discussion including a probabilistic representation of the limit $\widetilde{\Psi}_c$ in terms of a Poisson point process (see \eqref{eqn:Psi.tilde.representation}). Although Theorem \ref{thm:diagonal3} does not fit directly under the framework of Section \ref{sec:notations} and thus cannot be used directly in the applications described in Section \ref{sec:spt}, it is mathematically interesting as it gives another limit which is genuinely different from that of the hardmin case. 

\begin{remark}
	In the above results the functions may become negative and so may lie outside $\mathcal{C}$. To be precise, here we are using the topology of local uniform convergence on the space $\widetilde{\mathcal{C}}$ of real-valued continuous concave functions on $\overline{\Delta}_n$. We may continue to use the metric given by \eqref{eqn:C.metric}.
\end{remark}

\subsection{Poisson convergence} \label{sec:Poisson.convergence}
In this subsection we relate $X = \{X_k\}_{k \geq 1}$, regarded as a point process on $\mathbb{R}_+^n$, to a Poisson point process which is the main probabilistic tool of this section. For $K \geq 1$, let
\begin{equation} \label{eqn:Y.def}
Y^{(K)} := \{Y_k^{(K)} : 1 \leq k \leq K\} := \{ K^{\frac{1}{n+\alpha}} X_k : 1 \leq k \leq K \}.
\end{equation}
So $Y^{(K)}$ is also a point process on $\mathbb{R}_+^n$. The following result shows that $Y^{(K)}$ converges in distribution to a Poisson point process. This gives an alternative method to prove Theorem \ref{thm:weak.convergence} but we will not elaborate on this. For the precise meaning of the convergence of point processes we refer the reader to \cite[Section 3.4]{R87}.

\begin{proposition} \label{prop:Poisson.convergence}
Suppose $c_K \to 1$.  As $K \rightarrow \infty$, the point process
\[
c_KY^{(K)}= \{c_K Y^{(K)}_k: 1 \le k \le K\}
\]
converges in distribution to the Poisson process $N$ on $\mathbb{R}_+^n$ with intensity measure $dm = h(x) dx$, where $h$ is the function specified in Assumption \ref{ass:density}.\footnote{We will first apply this result with $c_K \equiv 1$. The general case will be used in Section \ref{sec:diagonal.third.case}.}
\end{proposition}
\begin{proof}
	By \cite[Proposition 3.19]{R87}, it suffices to prove that the Laplace transform of $ c_K Y^{(K)}$ converges to that of $N$. More precisely, we will show that
	\begin{equation} \label{eqn:Laplace.transform.claim}
	\begin{split}
	\lim_{K \rightarrow \infty} \mathbb{E} \left[ e^{-\sum_{k = 1}^K g(c_KY^{(K)}_k )} \right] = \exp\left[-\int_{\mathbb{R}_+^n} (1 - e^{-g(y)}) h(y) dy \right],
	\end{split}
	\end{equation}
	for any continuous function $g: \mathbb{R}^n \rightarrow \mathbb{R}$ with compact support.
	First we note that
	\begin{equation} \label{eqn:Laplace.first.step}
	\begin{split}
	\mathbb{E} \left[ e^{-\sum_{k = 1}^K g(c_K Y^{(K)}_k )} \right] &= \left[ \int_{\mathbb{R}_+^n} e^{-g(c_K K^{1/(n+\alpha)}x)} \rho (x) dx  \right]^K \\
	&= \left[1 -  \int_{\mathbb{R}_+^n} \left(1 - e^{-g(c_K K^{1/(n+\alpha)}x)}\right)  \rho (x) dx  \right]^K,
	\end{split}
	\end{equation}
	where $\rho$ is the common density of the $X_k$ (see Assumption \ref{ass:density}).
	
	Since $g$ is compactly supported, so is $1 - e^{-g}$. By an standard extension of the limit in Lemma \ref{lem:assumption.consequences} with $\lambda = c_K^{-1} K^{-1/(n+\alpha)}$, we have
	\begin{equation*}
	\begin{split}
	& K \int_{\mathbb{R}^n_+} \left(1 - e^{-g(c_K K^{1/(n + \alpha)}x)} \right) \rho(x) dx \\
	&= \frac{K^{\alpha/(n + \alpha)}}{c_K^n} \int_{\mathbb{R}^n_+} \left(1 - e^{-g(y)} \right) \rho(c_K^{-1}K^{-1/(n + \alpha)} y) dy \\
	&\rightarrow \int_{\mathbb{R}_+^n} \left(1 - e^{-g(y)} \right) h(y) dy.
	\end{split}
	\end{equation*}
	Letting $K \rightarrow \infty$ in \eqref{eqn:Laplace.first.step} and using the above limit, we obtain \eqref{eqn:Laplace.transform.claim}.
\end{proof}

By the continuous mapping theorem, we immediately obtain the following corollary.

\begin{corollary} \label{cor:Poisson.convergence}
	Suppose $c_K \to 1$.  For any $p \in \Delta_n$, the one-dimensional point process
	\begin{equation} \label{eqn:Y.K.projected}
	c_KY^{(K)}(p) := \langle p, c_K Y^{(K)} \rangle := \{ \langle p, c_K Y_k^{(K)} \rangle : 1 \leq k \leq K\}
	\end{equation}
	converges weakly to the Poisson point process $N(p) := \langle p, N \rangle := \{ \langle p, x \rangle : x \in N\}$ on $\mathbb{R}_+$ whose intensity measure $\widetilde{m}$ is the pushforward of $dm(x) = h(x) dx$ under the mapping $x \mapsto \langle p, x \rangle$. Explicitly, we have
	\begin{equation} \label{eqn:projected.intensity}
	\widetilde{m}(0, t] = \int_{R(p, t)} h(x) dx = t^{n+\alpha} \int_{R(p, 1)} h(y) dy.
	\end{equation}
\end{corollary}

For $p \in \Delta_n$ fixed, we will often denote $\widetilde{Y}^{(K)} = Y^{(K)}(p)$ and $\widetilde{N} = N(p)$.

\subsection{Proof of Theorem \ref{thm:diagonal12}} \label{sec:diagonal.second.case}

First we quickly settle the relatively trivial case of Theorem \ref{thm:diagonal12}(i). 	By Lemma \ref{lem:softmin}(i), we have
\[
K^{\frac{1}{n+\alpha}} \Phi_K(p) \leq K^{\frac{1}{n+\alpha}} \Phi_{\lambda_K, K}(p) \leq K^{\frac{1}{n+\alpha}} \Phi_K(p) + \frac{1}{\lambda_K} K^{\frac{1}{n+\alpha}} \log K.
\]
Suppose $\frac{\lambda_K }{ K^{1/(n + \alpha)}  \log K} \rightarrow \infty$. Then uniformly in $p$ we have
\[
\left| K^{\frac{1}{n+\alpha}} \Phi_{\lambda_K, K}(p) - K^{\frac{1}{n+\alpha}} \Phi_K(p) \right| \rightarrow 0, \quad K \rightarrow \infty.
\]
From this and Theorem \ref{thm:weak.convergence} we see easily that $K^{\frac{1}{n+\alpha}} \Phi_{\lambda_K, K}$ converges weakly to $\Psi$.

\medskip

For the general case we need a more delicate analysis. For $K \geq 1$ and $p \in \Delta_n$ fixed, relabel the points $\{X_k\}_{1 \leq k \leq K}$ so that
\[
\langle p, X_{(1)} \rangle \leq \langle p, X_{(2)} \rangle \leq \cdots \leq \langle p, X_{(K)} \rangle.
\]
Note that $\langle p, X_{(1)} \rangle = \min(\langle p,X_1\rangle, \ldots ,\langle p, X_K \rangle ) = \Phi_K(p)$. Write
\begin{equation} \label{eqn:computation.first.step}
\Phi_{\lambda, K}(p) = - \frac{1}{\lambda} \log \left( \frac{1}{K} \sum_{k = 1}^K e^{-\lambda \langle p, X_k \rangle } \right) = \frac{\log K}{\lambda} - \frac{1}{\lambda} \log \left( \sum_{k = 1}^K e^{-\lambda \langle p, X_{(k)} \rangle } \right).
\end{equation}
Then
\begin{equation} \label{eqn:computation.Theta}
\begin{split}
& K^{\frac{1}{n+\alpha}} \Phi_{\lambda, K}(p) \\
&= \frac{K^{\frac{1}{n+\alpha}} \log K}{\lambda} + K^{\frac{1}{n+\alpha}} \Phi_K(p) - \frac{K^{\frac{1}{n+\alpha}} }{\lambda} \log \left(1 + \sum_{k = 2}^K e^{-\lambda( \langle p, X_{(k)} \rangle - \langle p, X_{(1)} \rangle )} \right) \\
&=: \frac{K^{\frac{1}{n+\alpha}} \log K}{\lambda} + K^{\frac{1}{n+\alpha}} \Phi_K(p) - \frac{K^{\frac{1}{n+\alpha}} }{\lambda} \Theta_{\lambda, K}(p),
\end{split}
\end{equation}
where the term $\Theta_{\lambda, K}(p)$ is non-negative and satisfies
\begin{equation} \label{eqn:Theta}
\Theta_{\lambda, K}(p) \leq  \sum_{k = 2}^K e^{-\lambda( \langle p, X_{(k)} \rangle - \langle p, X_{(1)} \rangle )} = \sum_{k = 2}^K e^{-(\lambda/K^{1/(n+\alpha)})(\langle p, Y_{(k)}^{(K)}\rangle - \langle p,Y_{(1)}^{(K)}\rangle ) },
\end{equation}
and the $\langle p,Y_{(k)}^{(K)}\rangle $ are the points in $Y^{(K)}(p) = \{\langle p, Y_k^{(K)}\rangle: 1 \le k \le K\}$ arranged in ascending order.

The following is the main technical result needed in the proof of Theorem \ref{thm:diagonal12}.

\begin{theorem} \label{thm:tail.plim}
Fix $p \in \Delta_n$ and write $\widetilde{Y}_{(k)}^{(K)} = \langle p,Y_{(k)}^{(K)} \rangle$. Then
\begin{equation} \label{eqn:tail.plim}
\sum_{k = 2}^K e^{- (\widetilde{Y}_{(k)}^{(K)} - \widetilde{Y}_{(1)}^{(K)})/\epsilon} \rightarrow 0
\end{equation}
in probability as $K \rightarrow \infty$ and $\epsilon \rightarrow 0^+$.
\end{theorem}

To prove Theorem \ref{thm:tail.plim} we need some lemmas. Let $\widetilde{N} = N(p)$ be the one-dimensional Poisson point process given in Corollary \ref{cor:Poisson.convergence}. Since $m$ and hence $\widetilde{m}$ are Radon measures, we may order the points in $\widetilde{N}$ and write
\[
\widetilde{N} = \{ \widetilde{N}_{(1)} \leq \widetilde{N}_{(2)} \leq \cdots \}.
\]
In view of the convergence $\widetilde{Y}^{(K)} \rightarrow \widetilde{N}$, we expect that
\begin{equation} \label{eqn:Poisson.approx}
\sum_{k = 2}^K e^{-(\widetilde{Y}_{(k)}^{(K)} - \widetilde{Y}_{(1)}^{(K)})/\epsilon} \approx \sum_{k = 2}^\infty e^{-(\widetilde{N}_{(k)} - \widetilde{N}_{(1)})/\epsilon},
\end{equation}
and the right side of \eqref{eqn:Poisson.approx}, since it no longer depends on $K$, is easy to analyze as $\epsilon \to 0^+$.  In what follows, we do not attempt to justify \eqref{eqn:Poisson.approx} rigorously, but instead use the convergence $\widetilde{Y}^{(K)} \rightarrow \widetilde{N}$ to convert simple estimates on $\widetilde{N}$ into corresponding ones on $\widetilde{Y}^{(K)}$ for large $K$.

If $Z$ is a point process (with no double points) we let $Z(B)$ be the cardinality of $Z \cap B$.

\begin{lemma}  \label{lem:LMc} Given $\delta_0 >0$ there exist positive constants $L$, $M$ and $c$ such that for $K$ sufficiently large we have
	\begin{eqnarray*}
		\PP(\widetilde{Y}^{(K)}_{(1)} \ge L)  & \le & \delta_0, \\
		\PP(\widetilde{Y}^{(K)}(0,2L) > M)  & \le & \delta_0, \\
		\PP(\widetilde{Y}^{(K)}_{(2)} - \widetilde{Y}^{(K)}_{(1)} < c) & \le & \delta_0.
	\end{eqnarray*}
\end{lemma}
\begin{proof} By \eqref{eqn:projected.intensity}, the intensity $\widetilde{m}$ for $\widetilde{N}$ satisfies $\widetilde{m}(0,L) \to \infty$ as $L \to \infty$ and $\widetilde{m}(0,2L) < \infty$ for all $L$. Thus we have
	$$
	\PP(\widetilde{N}_{(1)} \ge L) = \mathbb{P}(\widetilde{N}(0,L) = 0) \to 0 \quad \mbox{ as }L \to \infty,
	$$
	and then for any given $L$ we have
	$$\PP(\widetilde{N}(0,2L) > M) \to 0  \quad \mbox{ as }M \to \infty.
	$$
	Also since $\widetilde{m}$ has no atoms then $\mathbb{P}(\widetilde{N}_{(2)} > \widetilde{N}_{(1)} ) = 1.$  So given $\delta_0 > 0$, there exist positive constants $L$, $M$ and $c$ such that
	\begin{eqnarray*}
		\mathbb{P}(\widetilde{N}_{(1)} \ge L)  & \le & \delta_0/2, \\
		\mathbb{P}(\widetilde{N}(0,2L) > M)  & \le & \delta_0/2, \\
		\mathbb{P}(\widetilde{N}_{(2)} - \widetilde{N}_{(1)} < c) & \le & \delta_0/2.
	\end{eqnarray*}
	The corresponding estimates for $\widetilde{Y}^{(K)}$ now follow directly from the convergence of $\widetilde{Y}^{(K)}$ to $\widetilde{N}$. In particular, for the convergence of $(\widetilde{Y}^{(K)}_{(1)},\widetilde{Y}^{(K)}_{(2)})$ to $(\widetilde{N}_{(1)},\widetilde{N}_{(2)})$ see \cite[Proposition 3.13]{R87}.
\end{proof}

We will also need an a priori estimate on the $\widetilde{Y}^{(K)}$.  The proof of the following lemma will be given in the Appendix. Note that Lemma \ref{lem:Xestimate}(iii) and Corollary \ref{cor:Yestimate}(ii) below will not be used until Section \ref{sec:diagonal.third.case}.

\begin{lemma} \label{lem:Xestimate} Fix $p \in \Delta_n$.
	\begin{enumerate}
		\item[(i)] The random variable $\langle p,X_1\rangle$ has density $\widetilde{\rho}(t)$ which satisfies
		$$
		\lim_{t \rightarrow 0^+}	\frac{1}{t^{n + \alpha - 1}}	\widetilde{\rho}(t) = (n+\alpha) \int_{R(p,1)} h(x)dx.
		$$
		\item[(ii)] There exists $B < \infty$ such that
		$$
		\gamma^{n+\alpha} \mathbb{E}\left[e^{-\gamma \langle p,X_1\rangle}\right] \le B, \quad \text{for } \gamma > 0.
		$$
		\item[(iii)]
		$$
		\lim_{\gamma, L \rightarrow \infty} \gamma^{n+\alpha} \mathbb{E}\left[e^{-\gamma \langle p,X_1\rangle}1_{\gamma \langle p,X_1\rangle \ge L}\right] = 0.
		$$
	\end{enumerate}
\end{lemma}

\begin{corollary} \label{cor:Yestimate} { \ }
	\begin{enumerate}	
		\item[(i)]  For all $\epsilon > 0$ and $K \ge 1$
		$$
		\mathbb{E}\left[\sum_{k=1}^K e^{-\langle p, Y^{(K)}_k \rangle /\epsilon}\right] \le \epsilon^{n+\alpha}B.
		$$
		\item[(ii)] For each fixed $c > 0$
		$$
		\lim_{K, L \rightarrow \infty} \mathbb{E}\left[\sum_{k=1}^K e^{-c \langle p, Y^{(K)}_k \rangle }1_{\langle p,Y^{(K)}_k \rangle \ge L }\right] = 0.$$
	\end{enumerate}
\end{corollary}

\begin{proof} We have
	$$
	\mathbb{E}\left[\sum_{k=1}^K e^{-\langle p, Y^{(K)}_k \rangle /\epsilon}\right] = K \mathbb{E}\left[e^{-K^{1/(n+\alpha)}\langle p,X_1\rangle /\epsilon}\right]
	$$
	and the first result follows by taking $\gamma = K^{1/(n+\alpha)}/\epsilon$ in Lemma \ref{lem:Xestimate}(ii) above.  The proof of the second result is similar, using Lemma \ref{lem:Xestimate}(iii).
\end{proof}

\begin{proof}[Proof of Theorem \ref{thm:tail.plim}]
	For all $\delta_0 > 0$ and $\delta_1 > 0$ we will prove that there exists $K_0$ and $\epsilon_0 > 0$ such that
	\[
	\mathbb{P}\left( \sum_{k = 2}^K e^{-( \widetilde{Y}_{(k)}^{(K)} - \widetilde{Y}_{(1)}^{(K)})/{\epsilon}}  > \delta_1 \right) < 4\delta_0,
	\]	
	whenever $K \geq K_0$ and $\epsilon < \epsilon_0$.
	
	Let $L$, $M$, $c$ and the constants in Lemma \ref{lem:LMc}, and $K_0$ be a positive integer such that the conclusion of the lemma holds for $K \geq K_0$. Applying Markov's inequality to Corollary
	\ref{cor:Yestimate}, 
	we get
	$$
	\mathbb{P}\left(\sum_{k=1}^K  e^{- \langle p, Y^{(K)}_k \rangle /\epsilon} \ge \delta_1/2 \right) \le \frac{2\epsilon ^{n+\alpha} B}{\delta_1} < \delta_0
	$$
	as long as $\epsilon < \epsilon_1 = (\delta_0\delta_1/(2B))^{1/(n+\alpha)}$.   Define the event
	\begin{equation*}
	\begin{split}
	\Omega_{K,\epsilon} &= \left\{\widetilde{Y}^{(K)}_{(1)} \ge L\right\} \cup \left\{\widetilde{Y}^{(K)}(0,2L) > M\right\} \\
	&\quad \cup \left\{\widetilde{Y}^{(K)}_{(2)} - \widetilde{Y}^{(K)}_{(1)} < c\right\} \cup \left\{\sum_{k=1}^K  e^{- \widetilde{Y}_{(k)}^{(K)}/\epsilon} \ge \delta_1/2\right\}.
	\end{split}
	\end{equation*}
	Then $\mathbb{P}(\Omega_{K,\epsilon}) < 4\delta_0$ if $K \ge K_0$ and $\epsilon < \epsilon_1$.  Now
	\begin{equation*}
	\begin{split}
	\sum_{k=2}^K e^{-(\widetilde{Y}^{(K)}_{(k)}-\widetilde{Y}^{(K)}_{(1)})/\epsilon}
	&= \sum_{k=2}^K e^{-(\widetilde{Y}^{(K)}_{(k)}-\widetilde{Y}^{(K)}_{(1)})/\epsilon}1_{\widetilde{Y}^{(K)}_{(k)} < 2L}\\
	&\quad + \sum_{k=2}^K e^{-(\widetilde{Y}^{(K)}_{(k)}-\widetilde{Y}^{(K)}_{(1)})/\epsilon}1_{\widetilde{Y}^{(K)}_{(k)} \ge 2L}.
	\end{split}
	\end{equation*}
	For $\widetilde{Y}^{(K)} \not\in \Omega_{K,\epsilon}$ the first sum on the right has at most $M$ terms and each term is at most $e^{-c/\epsilon}$.  Also, for $\widetilde{Y}^{(K)} \not\in \Omega_{K,\epsilon}$ each term in the second sum has $\widetilde{Y}^{(K)}_{(1)} < L \le  \widetilde{Y}^{(K)}_{(k)}/2$, so that $\widetilde{Y}^{(K)}_{(k)} - \widetilde{Y}^{(K)}_{(1)} \ge \widetilde{Y}^{(K)}_{(k)}/2$ and the second term is at most
	$$
	\sum_{k=2}^K e^{-\widetilde{Y}^{(K)}_{(k)}/(2\epsilon)}1_{\widetilde{Y}^{(K)}_{(k)} \ge 2L}
	\le \sum_{k=1}^K e^{-\widetilde{Y}^{(K)}_{(k)}/(2\epsilon)} = \sum_{k=1}^K e^{-\langle p, Y^{(K)}_k \rangle /(2\epsilon)}
	< \delta_1/2$$
	so long as $\epsilon < \epsilon_1/2$.  Together, if $K \ge K_0$ and $\epsilon < \epsilon_1/2$ we have
	$$
	\mathbb{P}\left( \sum_{k=2}^K e^{-(\widetilde{Y}^{(K)}_{(k)}-\widetilde{Y}^{(K)}_{(1)})/\epsilon} \ge M e^{-c/\epsilon}+\delta_1/2\right) \le \mathbb{P}(\Omega_{K,\epsilon}) < 4\delta_0.
	$$
	Finally it suffices to choose $\epsilon_2$ so that $Me^{-c/\epsilon_2} \le \delta_1/2$ and take $\epsilon_0 = \min(\epsilon_1/2,\epsilon_2)$.
\end{proof}

Now it is easy to complete the proof of Theorem \ref{thm:diagonal12}.

\begin{proof}[Proof of Theorem \ref{thm:diagonal12}]
	Let $A \subset \Delta_n$ be a finite set. By assumption $\lambda_K/K^{\frac{1}{n+\alpha}} \to\infty$, so using \eqref{eqn:computation.Theta}, \eqref{eqn:Theta} and Theorem \ref{thm:tail.plim}, it is easy to see that
	\[
	\left( K^{\frac{1}{n+\alpha}} \Phi_{\lambda_K, K}(p) - \frac{ K^{\frac{1}{n+\alpha}} \log K}{\lambda_K}\right)_{p \in A} \rightarrow ( \Psi(p))_{p \in A}
	\]
	weakly as $K \rightarrow \infty$. By the argument in the proof of Theorem \ref{thm:weak.convergence}, we have that the random concave function $K^{\frac{1}{n+\alpha}} \Phi_{\lambda_K, K}- \frac{ K^{\frac{1}{n+\alpha}} \log K}{\lambda_K}$ converges weakly to $\Psi$ as $K \rightarrow \infty$.
\end{proof}

\subsection{Proof of Theorem \ref{thm:diagonal3}} \label{sec:diagonal.third.case}
Again we start from the identity \eqref{eqn:computation.first.step}. Writing $\lambda_K/K^{\frac{1}{n+\alpha}} = c_K$, we have
\begin{equation} \label{eqn:diagonal.third.case}
K^{\frac{1}{n+\alpha}} \Phi_{\lambda_K, K} (p) - \frac{K^{\frac{1}{n+\alpha}}\log K}{\lambda_K} = -\frac{1}{c_K}\log  \left( \sum_{k = 1}^K e^{- c_K \langle p, Y^{(K)}_k\rangle } \right),
\end{equation}
where we recall $Y^{(K)}_k = K^{\frac{1}{n+\alpha}}X_k$ for $1 \le k \le K$.
Motivated by Corollary \ref{cor:Poisson.convergence}, since $c_K \to c > 0$ we expect that
\begin{equation} \label{eqn:convergence.log.sum}
-\frac{1}{c_K}\log \left( \sum_{k = 1}^K e^{- c_K \langle p,Y^{(K)}_k\rangle} \right) \rightarrow -\frac{1}{c}\log\left( \sum_{x \in N} e^{- c\langle p, x \rangle } \right) =: \widetilde{\Psi}_c(p),
\end{equation}
where $N$ is the Poisson point process with intensity measure $dm(x) = h(x)dx$. In the following proposition we verify this fact which is the main ingredient of the proof of Theorem \ref{thm:diagonal3}.

\begin{proposition} \label{prop:diagonal.finite.dim.convergence}
	For any finite subset  $A \subset \Delta_n$ we have
	\begin{equation*}
	\left( \sum_{k = 1}^K e^{- c_K \langle p,Y^{(K)}_k\rangle} \right)_{p \in A} \rightarrow \left( \sum_{x \in N} e^{- c\langle p, x \rangle } \right)_{p \in A}
	\end{equation*}
	in distribution as $K \rightarrow \infty$.
\end{proposition}

We first verify that the right hand side is finite almost surely.

\begin{lemma} \label{lem:expNfinite}
	For any $p \in \Delta_n$ we have $\mathbb{E}\left[\sum_{x \in N} e^{- c \langle p, x \rangle}\right] < \infty$ and hence
	\[
	\sum_{x \in N} e^{- c \langle p, x \rangle} < \infty
	\]
	almost surely.
\end{lemma}
\begin{proof}
	By Corollary \ref{cor:Poisson.convergence} and the property of Poisson point process, we have
	\begin{equation} \label{eqn:Poisson.sum.mean}
	\begin{split}
	\mathbb{E} \left[ \sum_{x \in N} e^{- c \langle p, x \rangle } \right] &= \int_0^{\infty} e^{-ct} d \widetilde{m}(t) \\
	&= (n + \alpha) \left(  \int_{R(p, 1)} h(y) dy \right) \int_0^{\infty} t^{n+\alpha-1} e^{-ct}dt \\
	&= \left(  \int_{R(p, 1)} h(y) dy \right)\frac{ \Gamma(n + \alpha + 1)}{c^{n+\alpha}} < \infty.
	\end{split}
	\end{equation}
\end{proof}

\begin{proof} [Proof of Proposition \ref{prop:diagonal.finite.dim.convergence}]
	Write $\widetilde{c}_K = c_K/c$, so that $\widetilde{c}_K \to 1$ as $K \to \infty$.  Since the point process $\widetilde{c}_K Y^{(K)} = \{\widetilde{c}_K K^{\frac{1}{n+\alpha}}X_k : 1 \leq k \leq K \}$ on $\mathbb{R}_+^n$ converges weakly to the Poisson point process $N$ (see Proposition \ref{prop:Poisson.convergence}), by Skorohod's theorem we may construct these processes on the same probability space in such a way that $\widetilde{c}_K Y^{(K)} \rightarrow N$ almost surely. It suffices to show that
	\[
	\sum_{k = 1}^K e^{- c_K \langle p, Y_k^{(K)} \rangle} \rightarrow \sum_{x \in N} e^{- c \langle p, x \rangle}
	\]
	in probability, for each fixed $p \in A$, as $K \rightarrow \infty$. The convergence then extends easily to the joint vector. 	In particular we will show that for all $\delta_0 > 0$ and $\delta_1 > 0$ there exists $K_0$, depending on $p$, such that
	\[
	\mathbb{P} \left(  \left| \sum_{k = 1}^K e^{- c_K \langle p, Y_k^{(K)} \rangle} - \sum_{x \in N} e^{- c \langle p, x \rangle} \right| > \delta_1 \right)< \delta_0
	\]
	for all $K \geq K_0$.


	Applying the dominated convergence theorem to Lemma \ref{lem:expNfinite} gives
	$$
	\mathbb{E} \left[ \sum_{x \in N} e^{- c \langle p, x \rangle} 1_{\langle p,x \rangle \ge L } \right]  \to 0
	$$
	as $L \to \infty$, and then Markov's inequality gives $L_0$ such that
	$$
	\mathbb{P} \left(\sum_{x \in N} e^{- c \langle p, x \rangle} 1_{\langle p,x \rangle \ge L } > \delta_1/3\right) \le \delta_0/3
	$$
	for all $L \ge L_0$.
	
	Since $c_K \to c$, there exists $K_1$ such that $\frac{c}{2} \leq c_K \leq 2c$, and hence $\tilde{c}_K \leq 2$, for $K \ge K_1$. Let $K \geq K_1$. Using Corollary \ref{cor:Yestimate}(ii), we get
	\begin{equation*}
	\begin{split}
		&\PP\left( \sum_{k = 1}^K e^{- c_K \langle p, Y_k^{(K)} \rangle} 1_{ \tilde{c}_K \langle p, Y^{(K)}_k \rangle \ge L}> \delta_1/3 \right) \\
		&\le \PP\left( \sum_{k = 1}^K e^{- (c/2) \langle p, Y_k^{(K)} \rangle} 1_{\langle p, Y^{(K)}_k \rangle \ge  L/2} > \delta_1/3 \right)\\
		&\le \frac{3}{\delta_1}\mathbb{E}\left[ \sum_{k = 1}^K e^{- (c/2) \langle p, Y_k^{(K)} \rangle} 1_{\langle p, Y^{(K)}_k \rangle \ge  L/2} \right] \\
		&\to 0
	\end{split}
	\end{equation*}
	as $K,L \to \infty$.  Therefore there exist $L \ge L_0$ and $K_2 \ge K_1$ such that
	$$
	\mathbb{P}\left( \sum_{k = 1}^K e^{- c_K \langle p, Y_k^{(K)} \rangle} 1_{\tilde{c}_K \langle p, Y^{(K)}_k \rangle \ge L}> \delta_1/3 \right) < \delta_0/3
	$$
	whenever $K \ge K_2$.
	Now let $f:\mathbb{R} \to \mathbb{R}$ be continuous with compact support such that $f(z) = e^{-cz}$
	for $0 < z \leq L$, and $0 \le f(z) \le e^{-cz}$ for all $z \ge L$.  Since $f$ has compact support we have
	\[
	\sum_{k = 1}^K f(\tilde{c}_K \langle p, Y_{k}^{(K)}\rangle) \rightarrow \sum_{x \in N} f(\langle p ,x \rangle)
	\]
	almost surely. In particular there exists $K_0 \ge K_2$ such that
	\[
	\mathbb{P} \left( \left| \sum_{k = 1}^K f(\tilde{c}_K \langle p, Y_{k}^{(K)}\rangle) - \sum_{x \in N} f(\langle p ,x \rangle)  \right|> \delta_1 / 3  \right) < \delta_0 / 3
	\]
	for $K \geq K_0$. Finally, since
	\begin{equation*}
	\begin{split}
	\left| \sum_{k = 1}^K e^{- c_K \langle p, Y_k^{(K)} \rangle} - \sum_{x \in N} e^{- c \langle p, x \rangle} \right|  &\leq \left| \sum_{k = 1}^K f(\tilde{c}_K \langle p, Y_{k}^{(K)}\rangle) - \sum_{x \in N} f(\langle p ,x \rangle)  \right| \\
	& \hspace{4ex}+ \left( \sum_{k = 1}^K e^{- c_K \langle p, Y_k^{(K)} \rangle} 1_{\tilde{c}_K \langle p, Y^{(K)}_k \rangle \ge L}\right)\\
	& \hspace{8ex} + \left(\sum_{x \in N} e^{- c \langle p, x \rangle} 1_{\langle p,x \rangle \ge L }\right),
	\end{split}
	\end{equation*}
	combining the above estimates we have
	\begin{equation*}
	\mathbb{P} \left( \left| \sum_{k = 1}^K e^{- c_K \langle p, Y_k^{(K)} \rangle } - \sum_{x \in N} e^{-c\langle p, x \rangle} \right| > \delta_1 \right) < \delta_0/3 + \delta_0/3 + \delta_0/3 = \delta_0
	\end{equation*}
	whenever $K \geq K_0$ and the proof is complete.
\end{proof}

Now we are ready to prove Theorem \ref{thm:diagonal3}.

\begin{proof}[Proof of Theorem \ref{thm:diagonal3}]
	Consider the sequence
	\[
	\widehat{\Psi}_K = K^{\frac{1}{n+\alpha}} \Phi_{\lambda_K, K} - K^{\frac{1}{n+\alpha}}\log K/\lambda_K
	\]
	given by \eqref{eqn:diagonal.third.case}. By Proposition \ref{prop:diagonal.finite.dim.convergence}, we have weak convergence of the finite dimensional distributions. Similar to the proof of Theorem \ref{thm:weak.convergence}, it remains to show that the sequence, as random elements with values in $\widetilde{\mathcal{C}}$, is tight.

	For $r \ge 1$ define $e_i^{(r)} = (1-1/r)e_i + (1/r)\overline{e}$, where we recall $\overline{e}$ is the barycenter $(1/n,\ldots,1/n)$ of $\Delta_n$, and define $\Omega_r = \mbox{conv}\{e_1^{(r)}, \ldots,e_n^{(r)}\}$.  Then $\Omega_r$ is compact and $\cup_{r \ge 1} \Omega_r = \Delta_n$.
	By \cite[Theorem 10.9]{R70}, the set
	\[
	\left\{ \psi \in \widetilde{\mathcal{C}}: \sup_{p \in \Omega_r} |\psi(p)| \leq M_r \text{ for all } r \geq 1 \right\}
	\]
	is compact in $\widetilde{\mathcal{C}}$ for any sequence $\{M_r\}$ with $M_r > 0$.
	
	\begin{lemma} \label{lem:C tilde bound}
		Suppose $\psi \in \widetilde{\mathcal{C}}$ satisfies $\psi(e_i^{(r)}) \le - L_1$ and $\psi(\overline{e}) \le L_2$.  Then
		$$
		-L_1 \le \psi(p) \le (n-1)L_1+nL_2
		$$
		for all $p \in \Omega_r$.
	\end{lemma}
	
	\begin{proof}  The convexity of $\psi$ implies that $\psi(p) \ge -L_1$ for all $p \in \Omega_r$.  Then the method of proof of Lemma \ref{lem:concave.bound} applied to the non-negative function $\psi+L_1$ on $\Omega_r$ implies that $\psi(p) + L_1 \le n (\psi(\overline{e})+L_1) \le n(L_1+L_2)$ for all $p \in\Omega_r$.
	\end{proof}
	
	Write $A_r = \{e^{(r)}_i: 1 \le i \le n\} \cup \{\overline{e}\}$.  Let $\epsilon > 0$ and $r \ge 1$.  Since
	$$
	\left( \widehat{\Psi}_K(p) \right)_{p \in A_r} \rightarrow \left( -\frac{1}{c}\log  \sum_{x \in N} e^{- c \langle p, x \rangle } \right)_{p \in A_r}
	$$
	in distribution, there exists $J_{r} > 0$ such that
	\begin{equation} \label{eqn:event.high.prob}
	\mathbb{P} \left( | \widehat{\Psi}_K(p) | \leq J_{r} \mbox{ for all } p \in A_r \right) \geq 1 - \frac{\epsilon}{2^{r}}
	\end{equation}
	for all $K \geq 1$.  By Lemma \ref{lem:C tilde bound} there exists $M_r$ such that
	$$
	\PP\left( \sup_{p \in\Omega_r} | \widehat{\Psi}_K(p) | \leq M_{r} \right) \geq 1 - \frac{\epsilon}{2^{r}}
	$$
	and so
	$$
	\PP\left( \sup_{p \in\Omega_r} | \widehat{\Psi}_K(p) | \leq M_{r} \mbox{ for all }r \ge 1\right) \geq 1 -   \epsilon
	$$
	for all $K \geq 1$. This establishes the tightness of the sequence, and completes the proof of Theorem \ref{thm:diagonal3}.
\end{proof}

To finish this paper we point out that the limit $\widetilde{\Psi}_c$ in Theorem \ref{thm:diagonal3} is drastically different from the limit $\Psi$ in the hardmin case. By Proposition \ref{prop:diagonal.finite.dim.convergence}, we may realize $\widetilde{\Psi}_c$ by
\begin{equation} \label{eqn:Psi.tilde.representation}
\widetilde{\Psi}_c(p) = -\frac{1}{c}\log \left( \sum_{x \in N} e^{-c \langle p, x \rangle} \right),
\end{equation}
where $N$ is the Poisson point process on $\mathbb{R}^n_+$ with intensity $dm(x) =h(x)dx$. From \eqref{eqn:Psi.tilde.representation}, it is not difficult to verify that $\widetilde{\Psi}_c$ is differentiable on $\Delta_n$. In contrast, in Theorem \ref{thm:weak.convergence} and Theorem \ref{thm:diagonal12} the random concave function $\Psi$ is locally piecewise affine on $\Delta_n$. Moreover the following result shows very different boundary behavior of $\widetilde{\Psi}_c$ compared with that of $\Psi$ (see Proposition \ref{prop:boundary}).  Intuitively, here we see a non-vanishing effect of the softmin as $K \rightarrow \infty$ when $\lambda_K$ is of order $K^{\frac{1}{n+\alpha}}$.

\begin{proposition} \label{prop:boundary.inf}  Under Assumption \ref{ass:density.B}, for all $c > 0$ we have
	$$
	\mathbb{P}\left(\widetilde{\Psi}_c(p) \to -\infty \mbox{ as } p \to \partial \Delta_n \right) = 1.
	$$
\end{proposition}

\begin{proof}  Using \eqref{eqn:Psi.tilde.representation}, it suffices to show that
	$$
	\mathbb{P}\left( N(R(p,1))  \to \infty \mbox{ as } p \to \partial \Delta_n \right) = 1,
	$$
	where $N(U)$ is the cardinality of $N \cap U$ where $U \subset \mathbb{R}_+^n$ is Borel. Since $p = (p_1,\ldots,p_n) \to \partial \Delta_n$ implies at least one coordinate $p_i$ tends to 0, it suffices to show that
	$$
	\mathbb{P}\left( N(R(p,1))  \to \infty \mbox{ as } p_n \to 0  \right) = 1.
	$$
	Given $\epsilon > 0$, for $p_n < \epsilon$ we have
	\begin{eqnarray*}
		R(p,1) &=& \{x \in \mathbb{R}^n_+: \sum_{i=1}^n p_ix_i < 1\} \\
		&\supseteq& \{x \in \mathbb{R}^n_+: \sum_{i=1}^{n-1} p_ix_i < 1/2 \mbox{ and }p_n x_n < 1/2\}\\
		&\supseteq& \{x \in \mathbb{R}^n_+: \sum_{i=1}^{n-1} x_i < 1/2 \mbox{ and }x_n < 1/(2\epsilon) \}:=R_{n,\epsilon}
	\end{eqnarray*}
	and $R_{n,\epsilon} \nearrow \{x \in \mathbb{R}^n_+: \sum_{i=1}^{n-1} x_i < 1/2\} := R_n$ as $\epsilon \to 0^+$.  Under Assumption \ref{ass:density.B} we have $m(R_n) = \infty$, and so $N(R_n) = \infty$ almost surely.  Therefore $N(R(p,1)) \ge N(R_{n,\epsilon}) \to \infty$ as $\epsilon \to 0^+$, and the proof is complete.
\end{proof}

\begin{figure}[t]
	\centering
	\includegraphics[scale=0.5]{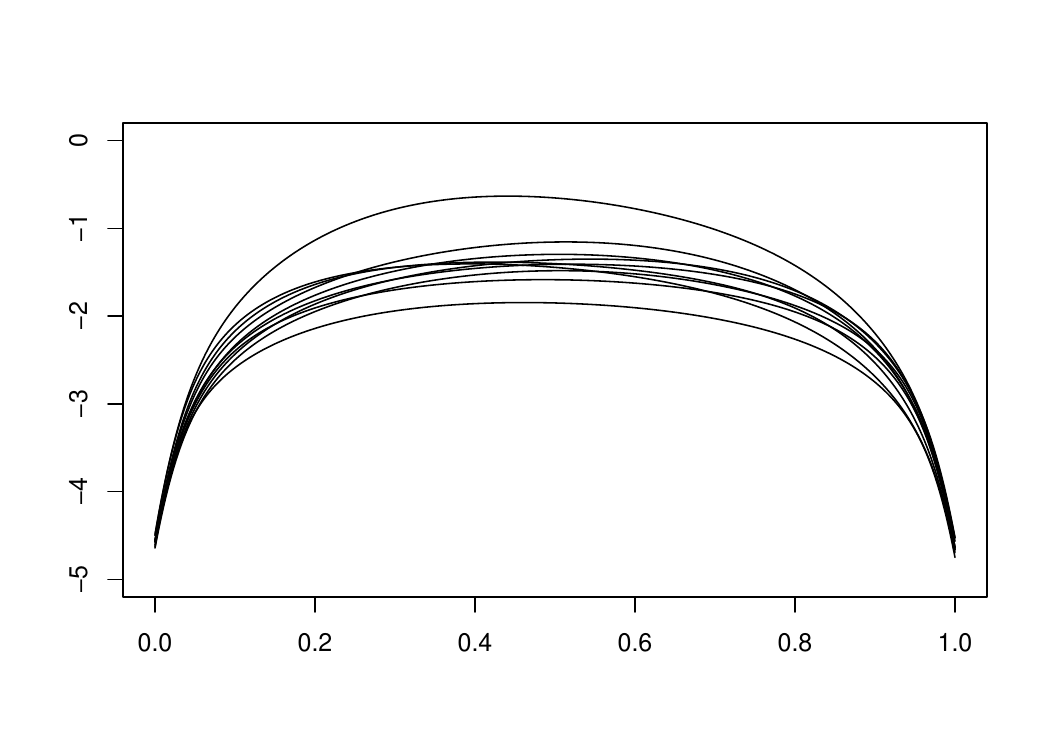}
	\vspace{-0.5cm}
	\caption{Ten approximate samples of $\widetilde{\Psi}_1$ when $n = 2$ and $m$ is the Lebesgue measure on $\mathbb{R}_+^n$ (i.e., $\gamma = 1$ in Example \ref{eg:last.example}). Note that the functions are differentiable in $p$. Here the samples are approximate because the Poisson point process is restricted to $[0, M]^n$ where $M > 0$ is a constant. When $M \rightarrow \infty$ the functions explode to $-\infty$ on the boundary of the simplex.} \label{fig:diagonal}
\end{figure}

\begin{example} \label{eg:last.example}
	Suppose that $h(x) \equiv \gamma > 0$ is constant, so that the intensity measure of $N$ is proportional to the Lebesgue measure on $\mathbb{R}_+^n$. Some approximate examples of $\widetilde{\Psi}_1$ are shown in Figure \ref{fig:diagonal}.  Specializing \eqref{eqn:Poisson.sum.mean} to this case, we have
	\[
	\mathbb{E} \left[ \sum_{x \in N} e^{- \langle p, x \rangle } \right] = \gamma \int_{\mathbb{R}_+^n} e^{-\langle p, x \rangle} dx = \frac{\gamma}{p_1 p_2 \cdots p_n} \to \infty \quad  \mbox{ as } p \to \partial \Delta_n.
	\]
	The proposition above shows that this blow-up also occurs in an almost-sure sense.

\end{example}

\appendix
\section*{Miscellaneous proofs} \label{app}
\setcounter{equation}{0}

\begin{proof}[Proof of Lemma \ref{lem:concave.bound}]
	(i) Let $q \in\Delta_n$ be given and we will prove \eqref{eqn:concave.bound} for $p \neq q$. For $p \in \overline{\Delta}_n$ with $p \neq q$ the half-ray $\{q+t(q-p): t >  0\}$ intersects the boundary $\partial \Delta_n$ at a unique point $p'$, say.  Suppose $p'=q +\lambda(q-p)$, so that $q = \frac{\lambda}{1+\lambda}p+ \frac{1}{1+\lambda}p'$.  For $\psi \in \mathcal{C}$ the concavity of $\psi$ along the line segment $[p,p']$ gives $\psi(q) \ge  \frac{\lambda}{1+\lambda}\psi(p)+ \frac{1}{1+\lambda}\psi(p')$.  Since $\psi(p') \ge 0$ we get $\psi(q) \ge \frac{\lambda}{1+\lambda}\psi(p)$, so that  $\psi(p) \le \frac{1+\lambda}{\lambda} \psi(q)$.
	
	Now $p'-q = \lambda(q-p)$ so that $\lambda = \|p'-q\|/\|q-p\|$ and then
	\[
	\frac{1+\lambda}{\lambda} = \frac{\|p'-p\|}{\|p'-q\|} \leq \frac{\mbox{diam}(\overline{\Delta}_n)}{\mbox{dist}(q, \partial \Delta_n)} =: M_q.
	\]
	
	(ii) Now suppose $q = c^{(j,\epsilon)}$ and $0 \le p_j \leq \epsilon$ with $\epsilon \leq 1/n$.  The value $t = (1-\epsilon)/(n-2+\epsilon)$ gives a point $\widetilde{p} = c^{(j,\epsilon)}+t(c^{(j,\epsilon)}-p)$ and it can be checked easily that $\widetilde{p}_i \geq 0$ for all $i$.  Thus $\widetilde{p} \in \overline{\Delta}_n$ and so $\lambda \geq t$.  Then
	\[
	\frac{1+\lambda}{\lambda} \leq \frac{1+t}{t} = \frac{n-1}{1-\epsilon} \leq n.
	\]
	The final statement now follows by taking $\epsilon = 1/n$ and $j = \arg\min\{p_i: 1 \leq i \leq n\}$, and noting that $\min\{p_i: 1 \leq i \leq n\} \leq 1/n$.
\end{proof}

\begin{proof} [Proof of Lemma \ref{lem:assumption.consequences}]
	Replacing $x$ by $cx$ in \eqref{eqn:origin} gives $h(cx) = c^\alpha h(x)$ so that $h$ is homogeneous of order $\alpha$. 	Since $\int_{\mathbb{R}_+^n} h(x) dx > 0$, $h$ is not identically zero and it is easy to see that $\alpha$ is unique. If $A$ is a bounded subset of $\mathbb{R}_+^n$ then there exists $\sigma > 0$ such that $\sum_{i=1}^n y_i \le \sigma$ for all $y \in A$.  Given $\epsilon > 0$ there is $\kappa_0$ such that
	\[
	\left| \frac{1}{\kappa^{\alpha}} \rho(\kappa x) - h(x) \right| < \epsilon
	\]
	whenever $x \in\Delta_n$ and $0 < \kappa < \kappa_0$.   Let $y \in A$ then $y = cx$ for some $0 < c \le \sigma$ and $x \in\Delta_n$.  For $\kappa < \kappa_0/\sigma$, so that $c \kappa < \kappa_0$, we have (using the homogeneity of $h$)
	\begin{equation*}
	\left| \frac{1}{\kappa^{\alpha}} \rho(\kappa y) - h(y) \right| < c^{\alpha} \epsilon \leq \sigma^{\alpha} \epsilon.
	\end{equation*}
	Thus $\frac{1}{\kappa^{\alpha}} \rho(\kappa x) \rightarrow h(x)$ uniformly on all bounded subsets of $\mathbb{R}_+^n$, and the limit \eqref{eqn:h.integral} follows immediately.

	Taking $A$ to be the set $R = \{x \in \mathbb{R}^n_+: \sum_{i=1}^n x_i \leq 1\}$, we get
	\begin{equation} \label{eqn:h.argument.estimate}
	\mathbb{P}(C \in \kappa R) = \int_{\kappa R} \rho(x)\,dx = \kappa^n \int_R \rho(\kappa y)\,dy  \ge \kappa^{n + \alpha} \left( \int_R h(y) dy - \epsilon \mathrm{vol}(R) \right)
	\end{equation}
	for $\kappa < \kappa_0$.  Since the left hand side is bounded above by 1, we have $\int_R h(y)\,dy <\infty$. Since $h$ is homogeneous, a simple scaling argument gives $\int_A h(y) dy < \infty$ for all bounded sets $A \subset \mathbb{R}_+^n$, completing the proof of \eqref{eqn:h.integral}.
	
	A similar scaling argument shows that $\int_R h(y) dy > 0$, since otherwise $\int_{\mathbb{R}_+^n} h(y) dy = 0$ and Assumption \ref{ass:density} is violated. So $\epsilon > 0$ can be chosen so that $\int_R h(y) dy - \epsilon \mathrm{vol}(R) > 0$. Since the left hand side of \eqref{eqn:h.argument.estimate} converges to $0$ as $\kappa \rightarrow 0^+$, we see that $\kappa^{n + \alpha} \rightarrow 0$ so that $n + \alpha > 0$ as desired.  
\end{proof}

\begin{proof}[Proof of Lemma \ref{lem:Xestimate}] Recall that the random vector $X_1$ has density $\rho$ on $\mathbb{R}^n_+$. Given $p \in \Delta_n$, we define
	\[
	D(p) = \{(u_1, \ldots, u_{n-1}) \in \mathbb{R}_+^{n-1} : \sum_{j = 1}^{n-1} p_j u_j < 1\}.
	\]
	By a straightforward computation, it can be verified that the random variable $\langle p, X_1 \rangle$ has density
	$$
	\widetilde{\rho}(t) = \frac{t^{n-1}}{p_n} \int_{D(p)} \rho(t u_1, t u_2, \ldots, t u_{n-1}, t (1 - \sum_{j = 1}^{n-1} p_j u_j)/p_n) du_1 \cdots d u_{n-1}
	$$
	for $t > 0$.  By Lemma 3.3, as $t \rightarrow 0^+$ we have
	\begin{eqnarray*}
		\widetilde{\rho}(t) &\sim & \frac{t^{n + \alpha - 1}}{p_n} \int_{D(p)} h(u_1, \ldots, u_{n-1}, (1 - \sum_{j = 1}^{n-1} p_j u_j)/p_n) du_1 \cdots du_{n-1} \\
		&= & (n + \alpha) t^{n + \alpha - 1} \int_{R(p, 1)} h(x) dx.
	\end{eqnarray*}
	This proves (i).
	
	Next we consider (ii) and (iii). By (i) there exist $\delta >0$ and $A > 0$ such that $\widetilde{\rho}(t) \le At^{n+\alpha-1}$ for $0 < t \le \delta$.  Then
	\begin{eqnarray*}
		\gamma^{n+\alpha}\E\left[e^{-\gamma \langle p,X_1\rangle} 1_{\gamma \langle p, X_1 \rangle \ge L} \right]
		&=&
		\gamma^{n+\alpha} \int_{L/\gamma}^\infty e^{- \gamma t} \widetilde{\rho}(t)dt\\
		&=&
		\gamma^{n+\alpha} \int_{L/\gamma}^{\max(\delta,L/\gamma)} e^{- \gamma t} \widetilde{\rho}(t)dt \\
		& & \mbox{}
		+  \gamma^{n+\alpha} \int_{\max(\delta,L/\gamma)}^\infty e^{- \gamma t} \widetilde{\rho}(t)dt\\
		&=:& I_1 + I_2,
	\end{eqnarray*}
	say.  If $L/\gamma \ge  \delta$ then $I_1 = 0$. Otherwise
	\begin{eqnarray*}
		I_1  =  \gamma^{n+\alpha} \int_{L/\gamma}^\delta e^{- \gamma t} \widetilde{\rho}(t)dt
		& \le & A\gamma^{n+\alpha} \int_{L/\gamma}^\delta e^{- \gamma t} t^{n+\alpha-1}dt \\
		& = & A \int_L^{\gamma \delta} e^{- u} u^{n+\alpha-1}du \\
		& \leq & A \int_L^\infty e^{- u} u^{n+\alpha-1}du.
	\end{eqnarray*}
	Also $I_2 \le \gamma^{n+\alpha}e^{-\gamma \delta}$.  Together we get
	$$
	\gamma^{n+\alpha}\E\left[e^{-\gamma \langle p,X_1\rangle }1_{\gamma \langle p, X_1 \rangle \ge L} \right]
	\le  A \int_L^\infty e^{- u} u^{n+\alpha-1}du + \gamma^{n+\alpha}e^{-\gamma \delta}.
	$$
	Thus (iii) follows immediately, and taking $L = 0$ we obtain (ii) with $B = A\Gamma(n+\alpha) + \delta^{-(n+\alpha)} \sup\{v^{n+\alpha}e^{-v}: v > 0\}$.
\end{proof}

\section*{Acknowledgements}
This project started when Leonard Wong was a postdoc at the University of Southern California. Part of the research was carried out when he was visiting the Faculty of Mathematics at the University of Vienna. He thanks Walter Schachermayer and Christa Cuchiero for many helpful discussions. His research is partially supported by NSERC Grant RGPIN-2019-04419. We thank the associate editor and the anonymous referees for their helpful comments which greatly improved the paper.

\bibliographystyle{imsart-nameyear}
\bibliography{geometry.ref}

\begin{thebibliography}{39}

\bibitem[\protect\citeauthoryear{Brenier}{1991}]{B91}
\begin{barticle}[author]
\bauthor{\bsnm{Brenier},~\bfnm{Yann}\binits{Y.}}
(\byear{1991}).
\btitle{Polar factorization and monotone rearrangement of vector-valued
  functions}.
\bjournal{Communications on Pure and Applied Mathematics}
\bvolume{44}
\bpages{375--417}.
\end{barticle}
\endbibitem

\bibitem[\protect\citeauthoryear{Cover}{1991}]{C91}
\begin{barticle}[author]
\bauthor{\bsnm{Cover},~\bfnm{Thomas~M.}\binits{T.~M.}}
(\byear{1991}).
\btitle{Universal portfolios}.
\bjournal{Mathematical Finance}
\bvolume{1}
\bpages{1--29}.
\end{barticle}
\endbibitem

\bibitem[\protect\citeauthoryear{Cuchiero, Schachermayer and
  Wong}{2019}]{CSW16}
\begin{barticle}[author]
\bauthor{\bsnm{Cuchiero},~\bfnm{Christa}\binits{C.}},
  \bauthor{\bsnm{Schachermayer},~\bfnm{Walter}\binits{W.}} \AND
  \bauthor{\bsnm{Wong},~\bfnm{Ting-Kam~Leonard}\binits{T.-K.~L.}}
(\byear{2019}).
\btitle{Cover's universal portfolio, stochastic portfolio theory, and the
  num{\'e}raire portfolio}.
\bjournal{Mathematical Finance}
\bvolume{29}
\bpages{773--803}.
\end{barticle}
\endbibitem

\bibitem[\protect\citeauthoryear{Cule, Samworth and Stewart}{2010}]{CSS10}
\begin{barticle}[author]
\bauthor{\bsnm{Cule},~\bfnm{Madeleine}\binits{M.}},
  \bauthor{\bsnm{Samworth},~\bfnm{Richard}\binits{R.}} \AND
  \bauthor{\bsnm{Stewart},~\bfnm{Michael}\binits{M.}}
(\byear{2010}).
\btitle{Maximum likelihood estimation of a multi-dimensional log-concave
  density}.
\bjournal{Journal of the Royal Statistical Society: Series B (Statistical
  Methodology)}
\bvolume{72}
\bpages{545--607}.
\end{barticle}
\endbibitem

\bibitem[\protect\citeauthoryear{De~Haan and Ferreira}{2007}]{DF07}
\begin{bbook}[author]
\bauthor{\bsnm{De~Haan},~\bfnm{Laurens}\binits{L.}} \AND
  \bauthor{\bsnm{Ferreira},~\bfnm{Ana}\binits{A.}}
(\byear{2007}).
\btitle{Extreme Value Theory: An Introduction}.
\bpublisher{Springer}.
\end{bbook}
\endbibitem

\bibitem[\protect\citeauthoryear{De~Philippis and Figalli}{2014}]{DF14}
\begin{barticle}[author]
\bauthor{\bsnm{De~Philippis},~\bfnm{Guido}\binits{G.}} \AND
  \bauthor{\bsnm{Figalli},~\bfnm{Alessio}\binits{A.}}
(\byear{2014}).
\btitle{The Monge--Amp{\`e}re equation and its link to optimal transportation}.
\bjournal{Bulletin of the American Mathematical Society}
\bvolume{51}
\bpages{527--580}.
\end{barticle}
\endbibitem

\bibitem[\protect\citeauthoryear{D{\"u}mbgen and Rufibach}{2009}]{DR09}
\begin{barticle}[author]
\bauthor{\bsnm{D{\"u}mbgen},~\bfnm{Lutz}\binits{L.}} \AND
  \bauthor{\bsnm{Rufibach},~\bfnm{Kaspar}\binits{K.}}
(\byear{2009}).
\btitle{Maximum likelihood estimation of a log-concave density and its
  distribution function: Basic properties and uniform consistency}.
\bjournal{Bernoulli}
\bvolume{15}
\bpages{40--68}.
\end{barticle}
\endbibitem

\bibitem[\protect\citeauthoryear{Evans and Gariepy}{2015}]{EG15}
\begin{bbook}[author]
\bauthor{\bsnm{Evans},~\bfnm{Lawrence~Craig}\binits{L.~C.}} \AND
  \bauthor{\bsnm{Gariepy},~\bfnm{Ronald~F}\binits{R.~F.}}
(\byear{2015}).
\btitle{Measure Theory and Fine Properties of Functions}.
\bpublisher{CRC Press}.
\end{bbook}
\endbibitem

\bibitem[\protect\citeauthoryear{Ferguson}{1973}]{F73}
\begin{barticle}[author]
\bauthor{\bsnm{Ferguson},~\bfnm{Thomas~S}\binits{T.~S.}}
(\byear{1973}).
\btitle{A {B}ayesian analysis of some nonparametric problems}.
\bjournal{The Annals of Statistics}
\bvolume{1}
\bpages{209--230}.
\end{barticle}
\endbibitem

\bibitem[\protect\citeauthoryear{Fernholz}{2002}]{F02}
\begin{bbook}[author]
\bauthor{\bsnm{Fernholz},~\bfnm{E.~Robert}\binits{E.~R.}}
(\byear{2002}).
\btitle{Stochastic Portfolio Theory}.
\bpublisher{Springer}.
\end{bbook}
\endbibitem

\bibitem[\protect\citeauthoryear{Gin{\'e}, Hahn and Vatan}{1990}]{GHV90}
\begin{barticle}[author]
\bauthor{\bsnm{Gin{\'e}},~\bfnm{Evarist}\binits{E.}},
  \bauthor{\bsnm{Hahn},~\bfnm{Marjorie~G}\binits{M.~G.}} \AND
  \bauthor{\bsnm{Vatan},~\bfnm{Pirooz}\binits{P.}}
(\byear{1990}).
\btitle{Max-infinitely divisible and max-stable sample continuous processes}.
\bjournal{Probability theory and related fields}
\bvolume{87}
\bpages{139--165}.
\end{barticle}
\endbibitem

\bibitem[\protect\citeauthoryear{Hannah and Dunson}{2011}]{HD11}
\begin{barticle}[author]
\bauthor{\bsnm{Hannah},~\bfnm{Lauren~A}\binits{L.~A.}} \AND
  \bauthor{\bsnm{Dunson},~\bfnm{David~B}\binits{D.~B.}}
(\byear{2011}).
\btitle{Bayesian nonparametric multivariate convex regression}.
\bjournal{arXiv preprint arXiv:1109.0322}.
\end{barticle}
\endbibitem

\bibitem[\protect\citeauthoryear{Hannah and Dunson}{2013}]{HD13}
\begin{barticle}[author]
\bauthor{\bsnm{Hannah},~\bfnm{Lauren~A}\binits{L.~A.}} \AND
  \bauthor{\bsnm{Dunson},~\bfnm{David~B}\binits{D.~B.}}
(\byear{2013}).
\btitle{Multivariate convex regression with adaptive partitioning}.
\bjournal{The Journal of Machine Learning Research}
\bvolume{14}
\bpages{3261--3294}.
\end{barticle}
\endbibitem

\bibitem[\protect\citeauthoryear{Itkin and Larsson}{2020}]{IL20}
\begin{barticle}[author]
\bauthor{\bsnm{Itkin},~\bfnm{David}\binits{D.}} \AND
  \bauthor{\bsnm{Larsson},~\bfnm{Martin}\binits{M.}}
(\byear{2020}).
\btitle{Robust asymptotic growth in stochastic portfolio theory under long-only
  constraints}.
\bjournal{arXiv preprint arXiv:2009.08533}.
\end{barticle}
\endbibitem

\bibitem[\protect\citeauthoryear{Johnson and Jiang}{2018}]{JD18}
\begin{barticle}[author]
\bauthor{\bsnm{Johnson},~\bfnm{Andrew~L}\binits{A.~L.}} \AND
  \bauthor{\bsnm{Jiang},~\bfnm{Daniel~R}\binits{D.~R.}}
(\byear{2018}).
\btitle{Shape constraints in economics and operations research}.
\bjournal{Statistical Science}
\bvolume{33}
\bpages{527--546}.
\end{barticle}
\endbibitem

\bibitem[\protect\citeauthoryear{Kabluchko et~al.}{}]{KMTT19}
\begin{barticle}[author]
\bauthor{\bsnm{Kabluchko},~\bfnm{Zakhar}\binits{Z.}},
  \bauthor{\bsnm{Marynych},~\bfnm{Alexander}\binits{A.}},
  \bauthor{\bsnm{Temesvari},~\bfnm{Daniel}\binits{D.}} \AND
  \bauthor{\bsnm{Th{\"a}le},~\bfnm{Christoph}\binits{C.}}
\btitle{Cones generated by random points on half-spheres and convex hulls of
  Poisson point processes}.
\bjournal{Probability Theory and Related Fields}
\bvolume{175}
\bpages{1021--1061}.
\end{barticle}
\endbibitem

\bibitem[\protect\citeauthoryear{Karatzas and Fernholz}{2009}]{KF09}
\begin{bincollection}[author]
\bauthor{\bsnm{Karatzas},~\bfnm{Ioannis}\binits{I.}} \AND
  \bauthor{\bsnm{Fernholz},~\bfnm{Robert}\binits{R.}}
(\byear{2009}).
\btitle{Stochastic portfolio theory: an overview}.
In \bbooktitle{Handbook of Numerical Analysis},
\bvolume{15}
\bpages{89--167}.
\bpublisher{Elsevier}.
\end{bincollection}
\endbibitem

\bibitem[\protect\citeauthoryear{Kingman}{1992}]{K92}
\begin{bbook}[author]
\bauthor{\bsnm{Kingman},~\bfnm{John Frank~Charles}\binits{J.~F.~C.}}
(\byear{1992}).
\btitle{Poisson Processes}.
\bpublisher{Clarendon Press}.
\end{bbook}
\endbibitem

\bibitem[\protect\citeauthoryear{Leadbetter, Lindgren and
  Rootz{\'e}n}{1983}]{LLR82}
\begin{bbook}[author]
\bauthor{\bsnm{Leadbetter},~\bfnm{Malcolm~R}\binits{M.~R.}},
  \bauthor{\bsnm{Lindgren},~\bfnm{Georg}\binits{G.}} \AND
  \bauthor{\bsnm{Rootz{\'e}n},~\bfnm{Holger}\binits{H.}}
(\byear{1983}).
\btitle{Extremes and Related Properties of Random Sequences and Processes}.
\bpublisher{Springer}.
\end{bbook}
\endbibitem

\bibitem[\protect\citeauthoryear{Lee}{2012}]{Lee12}
\begin{bbook}[author]
\bauthor{\bsnm{Lee},~\bfnm{John~M}\binits{J.~M.}}
(\byear{2012}).
\btitle{Introduction to Smooth Manifolds}.
\bpublisher{Springer}.
\end{bbook}
\endbibitem

\bibitem[\protect\citeauthoryear{Mariucci, Ray and Szab{\'o}}{2020}]{MRS17}
\begin{barticle}[author]
\bauthor{\bsnm{Mariucci},~\bfnm{Ester}\binits{E.}},
  \bauthor{\bsnm{Ray},~\bfnm{Kolyan}\binits{K.}} \AND
  \bauthor{\bsnm{Szab{\'o}},~\bfnm{Botond}\binits{B.}}
(\byear{2020}).
\btitle{A Bayesian nonparametric approach to log-concave density estimation}.
\bjournal{Bernoulli}
\bvolume{26}
\bpages{1070--1097}.
\end{barticle}
\endbibitem

\bibitem[\protect\citeauthoryear{Molchanov and Molchanov}{2005}]{MM05}
\begin{bbook}[author]
\bauthor{\bsnm{Molchanov},~\bfnm{Ilya}\binits{I.}} \AND
  \bauthor{\bsnm{Molchanov},~\bfnm{Ilya~S}\binits{I.~S.}}
(\byear{2005}).
\btitle{Theory of Random Sets}.
\bpublisher{Springer}.
\end{bbook}
\endbibitem

\bibitem[\protect\citeauthoryear{Pal and Wong}{}]{PW18}
\begin{barticle}[author]
\bauthor{\bsnm{Pal},~\bfnm{Soumik}\binits{S.}} \AND
  \bauthor{\bsnm{Wong},~\bfnm{Ting-Kam~Leonard}\binits{T.-K.~L.}}
\btitle{Multiplicative {S}chr\"{o}odinger problem and the {D}irichlet
  transport}.
\bjournal{Probability Theory and Related Fields}
\bvolume{178}
\bpages{613--654}.
\end{barticle}
\endbibitem

\bibitem[\protect\citeauthoryear{Pal and Wong}{2016}]{PW15}
\begin{barticle}[author]
\bauthor{\bsnm{Pal},~\bfnm{Soumik}\binits{S.}} \AND
  \bauthor{\bsnm{Wong},~\bfnm{Ting-Kam~Leonard}\binits{T.-K.~L.}}
(\byear{2016}).
\btitle{The geometry of relative arbitrage}.
\bjournal{Mathematics and Financial Economics}
\bvolume{10}
\bpages{263--293}.
\end{barticle}
\endbibitem

\bibitem[\protect\citeauthoryear{Pal and Wong}{2018}]{PW16}
\begin{barticle}[author]
\bauthor{\bsnm{Pal},~\bfnm{Soumik}\binits{S.}} \AND
  \bauthor{\bsnm{Wong},~\bfnm{Ting-Kam~Leonard}\binits{T.-K.~L.}}
(\byear{2018}).
\btitle{Exponentially concave functions and a new information geometry}.
\bjournal{The Annals of Probability}
\bvolume{46}
\bpages{1070--1113}.
\end{barticle}
\endbibitem

\bibitem[\protect\citeauthoryear{Resnick}{1987}]{R87}
\begin{bbook}[author]
\bauthor{\bsnm{Resnick},~\bfnm{Sidney~I.}\binits{S.~I.}}
(\byear{1987}).
\btitle{Extreme Values, Regular Variation and Point Processes}.
\bpublisher{Springer}.
\end{bbook}
\endbibitem

\bibitem[\protect\citeauthoryear{Rockafellar}{1997}]{R70}
\begin{bbook}[author]
\bauthor{\bsnm{Rockafellar},~\bfnm{{R. }~{Tyrrell}}\binits{R.~T.}}
(\byear{1997}).
\btitle{Convex Analysis}.
\bpublisher{Princeton University Press}.
\end{bbook}
\endbibitem

\bibitem[\protect\citeauthoryear{Samworth}{2018}]{S17}
\begin{barticle}[author]
\bauthor{\bsnm{Samworth},~\bfnm{Richard}\binits{R.}}
(\byear{2018}).
\btitle{Recent progress in log-concave density estimation}.
\bjournal{Statistical Science}
\bvolume{33}
\bpages{493--509}.
\end{barticle}
\endbibitem

\bibitem[\protect\citeauthoryear{Saumard and Wellner}{2014}]{SW14}
\begin{barticle}[author]
\bauthor{\bsnm{Saumard},~\bfnm{Adrien}\binits{A.}} \AND
  \bauthor{\bsnm{Wellner},~\bfnm{Jon~A}\binits{J.~A.}}
(\byear{2014}).
\btitle{Log-concavity and strong log-concavity: a review}.
\bjournal{Statistics Surveys}
\bvolume{8}
\bpages{45}.
\end{barticle}
\endbibitem

\bibitem[\protect\citeauthoryear{Schneider and Weil}{2008}]{SW08}
\begin{bbook}[author]
\bauthor{\bsnm{Schneider},~\bfnm{Rolf}\binits{R.}} \AND
  \bauthor{\bsnm{Weil},~\bfnm{Wolfgang}\binits{W.}}
(\byear{2008}).
\btitle{Stochastic and Integral Geometry}.
\bpublisher{Springer}.
\end{bbook}
\endbibitem

\bibitem[\protect\citeauthoryear{Seijo and Sen}{2011}]{SS11}
\begin{barticle}[author]
\bauthor{\bsnm{Seijo},~\bfnm{Emilio}\binits{E.}} \AND
  \bauthor{\bsnm{Sen},~\bfnm{Bodhisattva}\binits{B.}}
(\byear{2011}).
\btitle{Nonparametric least squares estimation of a multivariate convex
  regression function}.
\bjournal{The Annals of Statistics}
\bvolume{39}
\bpages{1633--1657}.
\end{barticle}
\endbibitem

\bibitem[\protect\citeauthoryear{Villani}{2003}]{V03}
\begin{bbook}[author]
\bauthor{\bsnm{Villani},~\bfnm{C{\'e}dric}\binits{C.}}
(\byear{2003}).
\btitle{Topics in Optimal Transportation}.
\bpublisher{American Mathematical Society}.
\end{bbook}
\endbibitem

\bibitem[\protect\citeauthoryear{Villani}{2008}]{V08}
\begin{bbook}[author]
\bauthor{\bsnm{Villani},~\bfnm{C{\'e}dric}\binits{C.}}
(\byear{2008}).
\btitle{Optimal Transport: Old and New}.
\bpublisher{Springer}.
\end{bbook}
\endbibitem

\bibitem[\protect\citeauthoryear{von Renesse and Sturm}{2009}]{vRS09}
\begin{barticle}[author]
\bauthor{\bparticle{von} \bsnm{Renesse},~\bfnm{Max-K}\binits{M.-K.}} \AND
  \bauthor{\bsnm{Sturm},~\bfnm{Karl-Theodor}\binits{K.-T.}}
(\byear{2009}).
\btitle{Entropic measure and {W}asserstein diffusion}.
\bjournal{The Annals of Probability}
\bvolume{37}
\bpages{1114--1191}.
\end{barticle}
\endbibitem

\bibitem[\protect\citeauthoryear{Wong}{2015a}]{wong2015universal}
\begin{barticle}[author]
\bauthor{\bsnm{Wong},~\bfnm{Ting-Kam~Leonard}\binits{T.-K.~L.}}
(\byear{2015}a).
\btitle{Universal portfolios in stochastic portfolio theory}.
\bjournal{preprint arXiv:1510.02808}.
\end{barticle}
\endbibitem

\bibitem[\protect\citeauthoryear{Wong}{2015b}]{W15}
\begin{barticle}[author]
\bauthor{\bsnm{Wong},~\bfnm{Ting-Kam~Leonard}\binits{T.-K.~L.}}
(\byear{2015}b).
\btitle{Optimization of relative arbitrage}.
\bjournal{Annals of Finance}
\bvolume{11}
\bpages{345--382}.
\end{barticle}
\endbibitem

\bibitem[\protect\citeauthoryear{Wong}{2018}]{W18}
\begin{barticle}[author]
\bauthor{\bsnm{Wong},~\bfnm{Ting-Kam~Leonard}\binits{T.-K.~L.}}
(\byear{2018}).
\btitle{Logarithmic divergences from optimal transport and {R}{\'e}nyi
  geometry}.
\bjournal{Information Geometry}
\bvolume{1}
\bpages{39--78}.
\end{barticle}
\endbibitem

\bibitem[\protect\citeauthoryear{Wong}{2019}]{W19}
\begin{bincollection}[author]
\bauthor{\bsnm{Wong},~\bfnm{Ting-Kam~Leonard}\binits{T.-K.~L.}}
(\byear{2019}).
\btitle{Information geometry in portfolio theory}.
In \bbooktitle{Geometric Structures of Information}
\bpages{105--136}.
\bpublisher{Springer}.
\end{bincollection}
\endbibitem

\bibitem[\protect\citeauthoryear{Wong and Yang}{2019}]{WY19}
\begin{barticle}[author]
\bauthor{\bsnm{Wong},~\bfnm{Ting-Kam~Leonard}\binits{T.-K.~L.}} \AND
  \bauthor{\bsnm{Yang},~\bfnm{Jiaowen}\binits{J.}}
(\byear{2019}).
\btitle{Pseudo-{R}iemannian geometry embeds information geometry in optimal
  transport}.
\bjournal{arXiv preprint arXiv:1906.00030}.
\end{barticle}
\endbibitem

\end{thebibliography}

\end{document}